\documentclass[11pt]{amsart}
\usepackage{amssymb,amsmath,amsthm,amsfonts,mathrsfs}
\usepackage[hmargin=3cm,vmargin=3.5cm]{geometry}
\usepackage[dvipsnames,table,xcdraw]{xcolor}
\usepackage[colorlinks = true,
            linkcolor = Fuchsia,
            urlcolor  = ForestGreen,
            citecolor = Plum,
            anchorcolor = blue]{hyperref}
\usepackage{graphicx,psfrag}
\usepackage{amscd}  
\usepackage[shellescape]{gmp}
\usepackage{stmaryrd}  %% double square brackets
\usepackage[all,2cell]{xy} \UseAllTwocells \SilentMatrices
\usepackage{tikz-cd}
\usepackage[dvips]{epsfig}
\usepackage{MnSymbol}
\usepackage{comment}
\usepackage{enumitem}
\usepackage{kbordermatrix}
\usepackage[colorinlistoftodos]{todonotes}

\renewcommand{\kbldelim}{(}% Left delimiter
\renewcommand{\kbrdelim}{)}% Right delimiter

\usetikzlibrary{arrows,automata}
\usetikzlibrary{decorations.markings}

\makeatletter
\def\@adminfootnotes{%
  \let\@makefnmark\relax  \let\@thefnmark\relax
  \ifx\@empty\@date\else \@footnotetext{\@setdate}\fi%%   <------ added
  \ifx\@empty\@subjclass\else \@footnotetext{\@setsubjclass}\fi
  \ifx\@empty\@keywords\else \@footnotetext{\@setkeywords}\fi
  \ifx\@empty\thankses\else \@footnotetext{%
    \def\par{\let\par\@par}\@setthanks}%
  \fi
}
\makeatother

\newtheorem{theorem}{Theorem}[section]

\newtheorem{prop}[theorem]{Proposition}

\newtheorem{corollary}[theorem]{Corollary}

\newtheorem{question}[theorem]{Question}

\theoremstyle{definition}

\newtheorem{example}[theorem]{Example}

\theoremstyle{remark}
\newtheorem{remark}[theorem]{Remark}

\numberwithin{equation}{section}

\let\oldtocsection=\tocsection
\let\oldtocsubsection=\tocsubsection
\renewcommand{\tocsection}[2]{\hspace{0em}\oldtocsection{#1}{#2}}
\renewcommand{\tocsubsection}[2]{\hspace{1em}\oldtocsubsection{#1}{#2}}
\usepackage{chngcntr}
\counterwithin{figure}{subsection}

\makeatletter
\@namedef{subjclassname@2020}{%
  \textup{2020} Mathematics Subject Classification}

\makeatother
 
\title[Automata and one-dimensional TQFTs]{Automata and one-dimensional TQFTs with defects}

\author[P. Gustafson]{Paul Gustafson}
\address{Penn Engineering, GRASP Laboratory, University of Pennsylvania, PA 19104, USA}
\email{\href{mailto:pgu@seas.upenn.edu}{pgu@seas.upenn.edu}}

\author[M.S. Im]{Mee Seong Im}
\address{Department of Mathematics, United States Naval Academy, Annapolis, MD 21402, USA}
\email{\href{mailto:meeseongim@gmail.com}{meeseongim@gmail.com}}
\thanks{}

\author[R. Kaldawy]{Remy Kaldawy} 
\address{Department of Mathematics, Columbia University, New York, NY 10027, USA}
\email{\href{mailto:rmk2179@columbia.edu}{rmk2179@columbia.edu}}
 
\author[M. Khovanov]{Mikhail Khovanov} 
\address{Department of Mathematics, Columbia University, New York, NY 10027, USA}
\email{\href{mailto:khovanov@math.columbia.edu}{khovanov@math.columbia.edu}}
\thanks{}

\author[Z. Lihn]{Zachary Lihn}
\address{Department of Mathematics, Columbia University, New York, NY 10027, USA}
\email{\href{mailto:zal2111@columbia.edu}{zal2111@columbia.edu}}

\subjclass[2020]{Primary: 
57K16, % Finite-type and quantum invariants, topological quantum field theories (TQFT)
68Q45, % Formal languages and automata
18M10, % Traced monoidal categories, compact closed categories, star-autonomous categories
18M30; % String diagrams and graphical calculi
Secondary: 06A12, %  Semilattices
68Q70, % Algebraic theory of languages and automata
18B20} % Categories of machines, automata
\date{February 27, 2023}

\providecommand{\keywords}[1]{\textbf{\textit{Key words and phrases.}} #1}

\keywords{Regular language, automaton, topological theory, universal construction, TQFT, monoidal category, Boolean semiring, semilattice, lattice.}

\begin{document}

\def\AND{\mathsf{AND}}
\def\concatenate{\mathsf{concatenate}}
\def\gen{\mathsf{generators}}
\def\GL{\mathsf{GL}}
\def\init{\mathsf{in}}
\def\t{\mathsf{t}}
\def\out{\mathsf{out}}
\def\I{\mathsf I}
\def\R{\mathbb R}
\def\Q{\mathbb Q}
\def\Z{\mathbb Z}
\def\mc{\mathcal}  %% confusion here 
\def\finite{\mathsf{finite}}
\def\infinite{\mathsf{infinite}}
\def\N{\mathbb N} 
\def\C{\mathbb C}
\def\S{\mathbb S}
\def\SS{\mathbb S} 
\def\CP{\mathbb P}
\def\Ob{\mathsf{Ob}}
\def\op{\mathsf{op}}
\def\new{\mathsf{new}}
\def\old{\mathsf{old}}
\def\OR{\mathsf{OR}}
\def\AND{\mathsf{AND}}
\def\rat{\mathsf{rat}}
\def\rec{\mathsf{rec}}
\def\Tau{\mathcal{T}}
\def\tail{\mathsf{tail}}
\def\coev{\mathsf{coev}}
\def\ev{\mathsf{ev}}
\def\id{\mathsf{id}}
\def\s{\mathsf{s}}
\def\t{\mathsf{t}}
\def\start{\textsf{starting}}
\def\Notation{\textsf{Notation}}
\def\circleft{\raisebox{-.18ex}{\scalebox{1}[2.25]{\rotatebox[origin=c]{180}{$\curvearrowright$}}}}
\renewcommand\SS{\ensuremath{\mathbb{S}}}
\newcommand{\kllS}{\kk\llangle  S \rrangle} %% power ser
\newcommand{\kllSS}[1]{\kk\llangle  #1 \rrangle}
\newcommand{\klS}{\kk\langle S\rangle}  % nc polynomials
\newcommand{\aver}{\mathsf{av}}  % average 
\newcommand{\ophana}{\overline{\phantom{a}}}
\newcommand{\Bool}{\mathbb{B}}
\newcommand{\dmod}{\mathsf{-mod}}
\newcommand{\lang}{\mathsf{lang}}
\newcommand{\pfmod}{\mathsf{-pfmod}}
\newcommand{\pdmod}{\mathsf{-pmod}}
\newcommand{\primitive}{\mathsf{irr}}
\newcommand{\Bmod}{\Bool\mathsf{-mod}}  % B-module 
\newcommand{\Bmodo}[1]{\Bool_{#1}\mathsf{-mod}}  
\newcommand{\Bfmod}{\Bool\mathsf{-fmod}} % finite B-modules 
\newcommand{\Bfpmod}{\Bool\mathsf{-fpmod}} % finite projective B-modules
\newcommand{\Bfsmod}{\Bool\mathsf{-}\underline{\mathsf{fmod}}}  % stable category 
\newcommand{\undvar}{\underline{\varepsilon}} %sequence of varepsilons, not using anymore
\newcommand{\RLang}{\mathsf{RLang}}
\newcommand{\undotimes}{\underline{\otimes}}
\newcommand{\sigmaacirc}{\Sigma^{\ast}_{\circ}} % equiv classes of words under rotation 
\newcommand{\cl}{\mathsf{cl}}
\newcommand{\PP}{\mathcal{P}} % powerset 
\newcommand{\wedgezero}{\{ \vee ,0\} } % semilattices
\newcommand{\whA}{\widehat{A}}
\newcommand{\whC}{\widehat{C}}
\newcommand{\whM}{\widehat{M}}
\newcommand{\mcH}{\mathcal{H}}
\newcommand{\mcU}{\mathcal{U}}
\newcommand{\mcUX}{\mathcal{U}(X)}
\newcommand{\mcFQ}{\mcF_{(Q)}}
\newcommand{\CobSI}{\Cob_{\Sigma,\I}}
\newcommand{\CobS}{\Cob_{\Sigma}}

\newcommand{\alphai}{\alpha_{\I}}  % alpha vertical
\newcommand{\alphac}{\alpha_{\circ}}  % alpha circle 
\newcommand{\alphap}{(\alphai,\alphac)} % alpha pair 
\newcommand{\alphaiQ}{\alpha_{\I,(Q)}} 
\newcommand{\alphacQ}{\alpha_{\circ,(Q)}}
\newcommand{\LcQ}{L_{\circ,(Q)}}  % circular language 
\newcommand{\kkvect}{\kk\mathsf{-vect}}

% redefine emptyset symbol 
\let\oldemptyset\emptyset
\let\emptyset\varnothing

\newcommand{\undempty}{\underline{\emptyset}}
\def\basis{\mathsf{basis}}
\def\irr{\mathsf{irr}} % recognizable series 
\def\spanning{\mathsf{spanning}}
\def\elmt{\mathsf{elmt}}

\def\l{\lbrace}
\def\r{\rbrace}
\def\o{\otimes}
\def\lra{\longrightarrow}
\def\Ext{\mathsf{Ext}}
\def\mf{\mathfrak} 
\def\mcC{\mathcal{C}}
\def\mcF{\mathcal{F}}
\def\mcO{\mathcal{O}}
\def\Fr{\mathsf{Fr}}

\def\ovb{\overline{b}}
\def\tr{{\sf tr}} 
\def\det{{\sf det }} 
\def\tral{\tr_{\alpha}}
\def\one{\mathbf{1}}   % unit  object of category 

\def\lra{\longrightarrow}
\def\twoheadlra{\longrightarrow\hspace{-4.6mm}\longrightarrow}
\def\hooklra{\raisebox{.2ex}{$\subset$}\!\!\!\raisebox{-0.21ex}{$\longrightarrow$}}
\def\kk{\mathbf{k}}  %% base field  
\def\gdim{\mathsf{gdim}}  %% graded dimension 
\def\rk{\mathsf{rk}}
\def\undep{\underline{\epsilon}}
\def\mathM{\mathbf{M}}  % boolean matrix 

% cobordism categories 
\def\CCC{\mathcal{C}} % cat of cobordisms 
\def\wCCC{\widehat{\CCC}}  % completed category

\def\complement{\mathsf{comp}}
\def\Rec{\mathsf{Rec}} % recognizable series  

\def\Cob{\mathsf{Cob}} 
\def\Kar{\mathsf{Kar}}   % Karoubi envelope 

\def\dmod{\mathsf{-mod}}   % modules  
\def\pmod{\mathsf{-pmod}}    % projective modules 
\def\fmod{\mathsf{-fmod}}    % free modules 

\newcommand{\brak}[1]{\ensuremath{\left\langle #1\right\rangle}}
\newcommand{\oplusop}[1]{{\mathop{\oplus}\limits_{#1}}}
\newcommand{\ang}[1]{\langle #1 \rangle } 
\newcommand{\ppartial}[1]{\frac{\partial}{\partial #1}} %partial derivative 
\newcommand{\checkr}{{\bf \color{red} CHECK IT}}
\newcommand{\checkb}{{\bf \color{blue} CHECK IT}}
\newcommand{\checkk}[1]{{\bf \color{red} #1}}
\newcommand{\FCob}{\mathsf{FCob}}

\newcommand{\mcA}{{\mathcal A}} 
\newcommand{\cZ}{{\mathcal Z}}
\newcommand{\sq}{$\square$}
\newcommand{\bi}{\bar \imath}
\newcommand{\bj}{\bar \jmath}
\newcommand{\iin}{\mathsf{in}} % initial states 

\newcommand{\undn}{\underline{n}}
\newcommand{\undm}{\underline{m}}
\newcommand{\undzero}{\underline{0}}
\newcommand{\undone}{\underline{1}}
\newcommand{\undtwo}{\underline{2}}

\newcommand{\cob}{\mathsf{cob}} % cobordism 
\newcommand{\comp}{\mathsf{comp}} % complementary

\newcommand{\Aut}{\mathsf{Aut}}
\newcommand{\Hom}{\mathsf{Hom}}
\newcommand{\Idem}{\mathsf{Idem}}
\newcommand{\Ind}{\mbox{Ind}}
\newcommand{\Id}{\textsf{Id}}
\newcommand{\End}{\mathsf{End}}
\newcommand{\iHom}{\underline{\mathsf{Hom}}}
\newcommand{\Bools}{\Bool^{\mathfrak{s}}}
\newcommand{\mfs}{\mathfrak{s}}

\newcommand{\drawing}[1]{
\begin{center}{\psfig{figure=fig/#1}}\end{center}}

\def\endomCempt{\End_{\mcC}(\emptyset_{n-1})}

\def\MS#1{{\color{blue}[MS: #1]}}
\def\MK#1{{\color{red}[MK: #1]}}
\def\PG#1{{\color{magenta}[PG: #1]}}

\begin{abstract} This paper explains how any nondeterministic automaton for a regular language $L$ gives rise to a one-dimensional oriented Topological Quantum Field Theory (TQFT) with inner endpoints and zero-dimensional defects labelled by letters of the alphabet for $L$. The TQFT is defined over the Boolean semiring $\Bool$. Different automata for a fixed language $L$ produce TQFTs that differ by their values on decorated circles, while the values on decorated intervals are described by the language $L$.  The language $L$ and the TQFT associated to an automaton can be given a path integral interpretation.  

In this TQFT the state space of a one-point 0-manifold is a free module over $\Bool$ with the basis of states of the automaton. Replacing a free module by a finite projective $\Bool$-module $P$ allows to generalize automata and this type of TQFT to a structure where defects act on open subsets of a finite topological space. Intersection of open subsets induces a multiplication on $P$ allowing to extend the TQFT to a TQFT for one-dimensional foams (oriented graphs with defects modulo a suitable equivalence relation).   

A linear version of these constructions is also explained, with the Boolean semiring replaced by a commutative ring.  
\end{abstract}

\maketitle
\tableofcontents

%%%%%%%%%%%%%%%%%%%%%%%
%
%  Intro  
%
%%%%%%%%%%%%%%%%%%%%%%%

\section{Introduction}
\label{section:intro}

The punchline of this paper is that  any nondeterministic finite state automaton $(Q)$ recognizing a regular language $L$ gives rise to a Boolean one-dimensional topological quantum field theory (TQFT) $\mcFQ$ with defects. The state space $\mcFQ(+)$ of a positively-oriented point in this TQFT is the free (semi)module $\Bool Q$ on the set of states $Q$ of the automaton $(Q)$ over the Boolean semiring $\Bool=\{0,1|1+1=1\}$.  
Vice versa, a TQFT of this form, where state spaces are finite free Boolean modules, describes an automaton. 

TQFT $\mcFQ$ is a symmetric monoidal functor from the category $\CobSI$ of oriented one-dimensional cobordisms with $\Sigma$-labelled defects and inner endpoints taking values in the category $\Bool\fmod$ of free $\Bool$-modules. 
In category $\CobSI$, closed cobordisms are disjoint unions of intervals and circles with defects. A defect is a point (a zero-dimensional submanifold) of a one-manifold with a label from $\Sigma$ on it. 

An interval with defects defines a word $\omega$, given by reading the defects along the orientation of the interval. 
In the TQFT $\mcFQ$, an interval evaluates to $1$ if and only if $\omega$ is in the language $L$, otherwise it evaluates to $0$. A circle with defects defines a circular word $\omega'$ which evaluates to $1$ if and only if there is a cycle that reads $\omega'$ in the automaton  $(Q)$. Numbers $0$ and $1$ are viewed as elements of the Boolean semiring $\Bool$. 

In the TQFT $\mcFQ$ a defect labelled $a$ on an interval with two boundary points induces an endomorphism of $\Bool Q$ encoded by the transition function of the automaton, while sets of initial and accepting states determine the maps of state spaces near the inner (not boundary) points of the automaton. 

Thus, NFA or nondeterministic finite automaton $(Q)$ for a regular language $L$ gives rise to a one-dimensional Boolean TQFT with defects, where the language $L$ is encoded by evaluations of decorated intervals. This is explained in Section~\ref{subset_oned_from}, with the correspondence summarized in Proposition~\ref{prop_bijection}. 

One can fix $L$ and consider various automata $(Q)$ for $L$. The corresponding TQFTs $\mcFQ$ have the same evaluation on intervals but usually differ in their evaluation of circles with defects. Each such TQFT defines a \emph{circular language}, a $\Bool$-valued function on the set $\Sigma^{\ast}$  of words modulo the rotation equivalence relation. In Section~\ref{subsec_dependence} we show that, for a fixed $L$, many circular languages are possible and suggest the open problem to determine all possible circular languages for automata that describe a fixed regular language $L$. 

\vspace{0.07in} 

To determine whether a word $\omega$ is accepted by an automaton $(Q)$ one can sum (in the Boolean semiring) over all maps from the graph $I(\omega)$ which is a chain with edges labelled by letters in $\omega$ to $(Q)$ taking vertices to vertices and edges to edges. The map evaluates to $1\in \Bool$ if the letters of all edges match and the initial and terminal vertices of $I(\omega)$ go to vertices in the corresponding subsets of states of $(Q)$. This expression is reminiscent of the path integral (sum over all maps) and is explained in Section~\ref{subsec_path}. 

\vspace{0.07in} 

Earlier, in Section~\ref{sec_linear}, we discuss the linear version of this construction and classify TQFT functors from $\CobSI$ to the category $\R\dmod$ of modules over a commutative ring $R$. Isomorphism classes of these functors are in a bijection with finitely-generated projective $R$-modules $P$ with a choice of a vector, a covector, and endomorphisms of $P$, one for each element of $\Sigma$. 
Section~\ref{sec_linear} can be skipped to go directly to Section~\ref{sec_aut_TQFT} that relates automata and 1D TQFT over the Boolean semiring.

Section~\ref{sec_linear} is closely related to~\cite{Kh3,IK-22-linear,KS3}  that consider field-valued \emph{universal construction} for the category $\CobSI$. 
Universal construction gives rise to \emph{topological theories} for the category $\CobSI$, where one starts with an evaluation of closed morphisms (endomorphisms of the unit object $\one$ of the category of one-cobordisms, which is the empty $0$-manifold) and builds state spaces for $0$-manifolds from the evaluation. Topological theories are weaker than TQFTs and can be called \emph{lax TQFTs}. In a TQFT $\mcF$ the state space of the disjoint union is isomorphic to the direct product of state spaces for the individual components: 
\[\mcF(N_0\sqcup N_1)\cong \mcF(N_0)\otimes \mcF(N_1).
\] 
In a topological theory, there are only maps 
\[\mcF(N_0)\otimes \mcF(N_1)\lra \mcF(N_0\sqcup N_1),
\] 
injective for a theory defined over a field, which are rarely isomorphisms. Relation between the present paper and~\cite{IK-top-automata} is explained in Section~\ref{subsec_relations}. 

\vspace{0.07in} 

A general one-dimensional TQFT with defects, as above, assign a projective, 
not necessarily free, $\Bool$-module to a $+$ point, see  
Section~\ref{subsec_topological}. In Section~\ref{subsec_aut_top} we explain how the correspondence
\begin{center}
\emph{automata $\Longleftrightarrow$ TQFT}
\end{center} 
extends to this setup. One replaces the free module $\Bool Q$ on the set of states of an automaton $(Q)$ by a finite projective $\Bool$-module $P$. This module is isomorphic to the module $\mcUX$ of open sets of a finite topological space $X$, with addition given by the union of sets. Now letters $a\in \Sigma$ act on $\mcUX$ by taking open sets to open sets respecting the union operation, there is an initial open set and a functional which is analogous to the set $Q_{\t}$ of accepting states. Such a structure may be called \emph{quasi-automaton} or \emph{$\Tau$-automaton} (``T" for \emph{topological space}). 
When topological space $X$ is discrete, one recovers the familiar notion of nondeterministic automaton. 

A $\Bool$-semimodule of the form $\mcUX$ for a finite topological space $X$ is a distributive lattice, with the join operator $U\vee V := U \cup V$ and meet operation $U\wedge V := U \cap V$. Distributivity of the meet operation over the join turn $\Bool$-module $\mcUX$ into a commutative semiring with the multiplication $\wedge$. This allows to extend TQFT associated to $\mcUX$ to a graph TQFT where trivalent merge and split vertices are taken to the multiplication and its dual comultiplication on $\mcUX$, as explained in Section~\ref{subsec_foams}. These graphs satisfy associativity and coassociativity relations and the TQFT can be naturally viewed as a TQFT for one-foams. Two-dimensional foams appear throughout link homology~\cite{Kh1,MV,RW16,RW1,KR21}
, and here one seems to encounter their one-dimensional counterpart. It is then straighforward to add defects to one-foam TQFTs, combining the two extensions: to quasi-automata and to one-foams. 

Remark~\ref{remark_pi} in Section~\ref{subsec_path} explains how some of the constructions in the present paper can be generalized from a free category on one object and a set $\Sigma$ of generating morphisms to an arbitrary small category. 

\vspace{0.07in} 

There are many interesting open problems here, including finding higher-dimensional analogues of the correspondence 
\begin{center}
Finite State Automata $\Longleftrightarrow$ Boolean 1D TQFT with defects. 
\end{center} 
It might also be interesting to compare the present construction with G.~'t Hooft's cellular automata interpretation of quantum mechanics~\cite{Hooft16}, see also a brief discussion of defect diagrammatics for quantum mechanics in Section~\ref{subsec_floating}. 

\vspace{0.1in} 

{\bf Acknowledgments.} P.G. was supported by AFRL grant FA865015D1845 (subcontract 669737-1) and ONR grant N00014-16-1-2817, a Vannevar Bush Fellowship held by Dan Koditschek and sponsored by the Basic Research Office of the Assistant Secretary of Defense for Research and Engineering. M.K. gratefully acknowledges partial support from NSF grant DMS-2204033 and Simons Collaboration Award 994328.

%%%%%%%%%%%%
%
%  TQFTs with defects over a ring 
%
%%%%%%%%%%%%

\section{One-dimensional TQFTs with inner endpoints and defects over a commutative ring}\label{sec_linear}

%%%%%%%%%%%%%%%%%%
% 1d and projectives
%%%%%%%%%%%%%%%%%%

\subsection{One-dimensional TQFT and finitely-generated projective modules}
\label{subsection:projectives}

Let $R$ be a commutative ring. The category $R\dmod$ of $R$-modules is symmetric monoidal with respect to the tensor product $M\otimes_R N$.  A \emph{one-dimensional TQFT} over $R$ is a symmetric monoidal functor
\begin{equation}
\mcF \ : \ \Cob \lra R\dmod
\end{equation}
from the category $\Cob$ of oriented one-dimensional cobordisms to $R\dmod$.

Category $\Cob$ has generating objects $+$ and $-$, which are a positively- and a negatively-oriented point, respectively. The cup and cap cobordisms, together with  permutation cobordisms, are generating morphisms, see Figure~\ref{figure-0.1}.

A one-dimensional TQFT $\mcF$, as above, takes $+$ and $-$ to $R$-modules $M:=\mcF(+)$ and $N:=\mcF(-)$, an arbitrary finite sign sequence to the corresponding tensor products of $M$'s and $N$'s, and suitably oriented cup and cap cobordisms to the maps $\cup$ and $\cap$ in \eqref{eq_cup_cap} below.
Thus, $\mcF$
amounts to choosing two $R$-modules $M$,$N$ together with the \emph{cup} and \emph{cap} maps
\begin{equation}\label{eq_cup_cap}
    \cup : R\lra M\otimes N, \ \ \ \cap: N\otimes M \lra R
\end{equation}
such that
\begin{equation}\label{eq_two_isot}
( \id_{M} \otimes \cap)\circ(\cup\otimes \id_{M}) = \id_{M}, \ \ \
   (\cap \otimes \id_{N} )\circ(\id_{N}\otimes \cup) = \id_{N},
\end{equation}
see Figure~\ref{figure-0.1} bottom left and right, respectively. The above relations are the isotopy relations on these maps.

\vspace{0.1in}

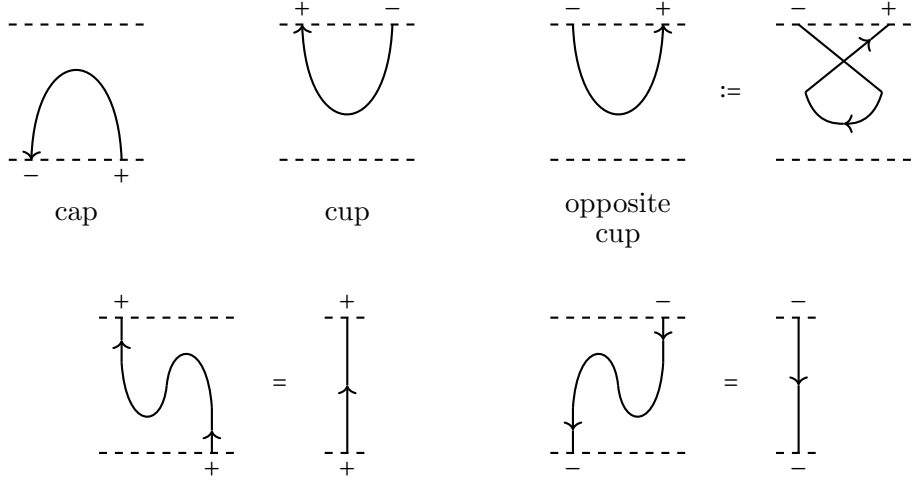
\begin{figure}
    \centering
    \begin{tikzpicture}[scale=0.6]
    \begin{scope}[shift={(0,0)}]
    %\draw[thin,yellow] (0,0) grid (4,4);
    \draw[thick,dashed] (0,3) -- (3,3);
    \draw[thick,dashed] (0,0) -- (3,0);
    
    \node at (0.5,-0.35) {$-$};
    \node at (2.5,-0.35) {$+$};
    
    \draw[thick,<-] (0.5,0) .. controls (0.6,2.65) and (2.4,2.65) .. (2.5,0);
    \node at (1.5,-1.25) {cap};
    \end{scope}

    \begin{scope}[shift={(6,0)}]
    %\draw[thin,yellow] (0,0) grid (4,4);
    \draw[thick,dashed] (0,3) -- (3,3);
    \draw[thick,dashed] (0,0) -- (3,0);
    
    \node at (0.5,3.35) {$+$};
    \node at (2.5,3.35) {$-$};
    \draw[thick,<-] (0.5,3) .. controls (0.6,0.35) and (2.4,0.35) .. (2.5,3);
    \node at (1.5,-1.25) {cup};
    \end{scope}

    \begin{scope}[shift={(12,0)}]
    %\draw[thin,yellow] (0,0) grid (4,4);
    \draw[thick,dashed] (0,3) -- (3,3);
    \draw[thick,dashed] (0,0) -- (3,0);
    
    \node at (0.5,3.35) {$-$};
    \node at (2.5,3.35) {$+$};
    \draw[thick,->] (0.5,3) .. controls (0.6,0.35) and (2.4,0.35) .. (2.5,3);
    \node at (1.5,-1) {opposite};
    \node at (1.5,-1.75) {cup};
    
    \node at (4,1.5) {$:=$};
    \end{scope}

    \begin{scope}[shift={(17,0)}]
    %\draw[thin,yellow] (0,0) grid (4,4);
    \draw[thick,dashed] (0,3) -- (3,3);
    \draw[thick,dashed] (0,0) -- (3,0);
    
    \node at (0.5,3.35) {$-$};
    \node at (2.5,3.35) {$+$};
    \draw[thick] (0.65,1.5) .. controls (0.85,0.8) and (1.3,0.8) .. (1.5,0.8);
    \draw[thick,<-] (1.5,0.8) .. controls (1.7,0.8) and (2.15,0.8) .. (2.35,1.5); 
    \draw[thick,->] (0.65,1.5) -- (2.1,2.67);
    \draw[thick] (2.1,2.67) -- (2.5,3);
    \draw[thick] (2.35,1.5) -- (0.5,3);
    \end{scope} 
    
    \begin{scope}[shift={(2,-6.5)}]
    %\draw[thin,yellow] (0,0) grid (4,4);
    \draw[thick,dashed] (0,3) -- (3,3);
    \draw[thick,dashed] (0,0) -- (3,0);
    
    \node at (0.5, 3.35) {$+$};
    \node at (2.5,-0.35) {$+$};
    
    \draw[thick] (0.5,3) -- (0.5,2.5);
    \draw[thick,<-] (0.5,2.5) -- (0.5,2.0);    
    \draw[thick,<-] (2.5,0.5) -- (2.5,0);
    \draw[thick] (2.5,1.0) -- (2.5,0.5); 
    
    \draw[thick] (0.5,2) .. controls (0.6,0.5) and (1.4,0.5) .. (1.5,1.5);
    \draw[thick] (1.5,1.5) .. controls (1.6,2.5) and (2.4,2.5) .. (2.5,1);    
    \node at (4,1.5) {$=$};
    \end{scope}
    
    \begin{scope}[shift={(7,-6.5)}]
    %\draw[thin,yellow] (0,0) grid (4,4);
    \draw[thick,dashed] (0,3) -- (1,3);
    \draw[thick,dashed] (0,0) -- (1,0);
    \node at (0.5, 3.35) {$+$};
    \node at (0.5,-0.35) {$+$};
    \draw[thick,<-] (0.5,1.5) -- (0.5,0);
    \draw[thick] (0.5,3) -- (0.5,1.5);
    \end{scope}
    
    \begin{scope}[shift={(12,-6.5)}]
%    \draw[thin,yellow] (0,0) grid (4,4);
    \draw[thick,dashed] (0,3) -- (3,3);
    \draw[thick,dashed] (0,0) -- (3,0);
    
    \node at (2.5, 3.35) {$-$};
    \node at (0.5,-0.35) {$-$};
    
    \draw[thick,->] (2.5,3) -- (2.5,2.5);
    \draw[thick] (2.5,2.5) -- (2.5,2.0);    
    \draw[thick] (0.5,0.5) -- (0.5,0);
    \draw[thick,->] (0.5,1.0) -- (0.5,0.5); 

    \draw[thick] (0.5,1) .. controls (0.6,2.5) and (1.4,2.5) .. (1.5,1.5);
    \draw[thick] (1.5,1.5) .. controls (1.6,0.5) and (2.4,0.5) .. (2.5,2);   
    \node at (4,1.5) {$=$};
    \end{scope}
    
    \begin{scope}[shift={(17,-6.5)}]
    %\draw[thin,yellow] (0,0) grid (4,4);
    \draw[thick,dashed] (0,3) -- (1,3);
    \draw[thick,dashed] (0,0) -- (1,0);
    \node at (0.5, 3.35) {$-$};
    \node at (0.5,-0.35) {$-$};
    \draw[thick] (0.5,1.5) -- (0.5,0);
    \draw[thick,->] (0.5,3) -- (0.5,1.5);
    \end{scope}

    \end{tikzpicture}
    \caption{Top left: the cap and cup cobordisms. The cap and cup cobordisms for the opposite orientation are obtained by composing these two with the permutation cobordisms, see top right for the cup cobordism for the opposite orientation. Bottom row shows isotopy relations on cup and cap cobordisms.}
    \label{figure-0.1}
\end{figure}

We can write
\[ \cup(1)=\sum_{i=1}^n m_i \otimes n_i, \hspace{0.5cm} 
m_i\in M, \hspace{0.5cm} 
n_i \in N.
\]
From the first relation in \eqref{eq_two_isot},
\begin{equation} 
\label{eq_m_decomp}
m = \sum_{i=1}^k  \cap(n_i\otimes m) \, m_i, \hspace{0.5cm}   m\in M,
\end{equation}
so that $M$ is generated by $m_1,\ldots, m_k$. Let $p:R^k\lra M$ be the surjective $R$-module map taking standard generators of $R^k$ to $m_1,\ldots, m_k$.
Each $n_i\in N$ defines an $R$-module map
\[
\iota_i: M\lra R, \hspace{0.5cm} \iota_i(m) = \cap(n_i\otimes m) ,
\]
and together they give a map
\[
\iota =(\iota_1,\ldots, \iota_k)^T \ : \ M \lra R^k.
\]
The composition $p\circ \iota =\id_{M}$, due to the
formula \eqref{eq_m_decomp}, where $T$ is the transpose of the $k$-tuple. Consequently, maps $p$ and $\iota$ realize $M$ as a direct summand of the free module $R^k$. Consequently, $M$ is a finitely-generated projective module, and so is $N$.

The dual module $M^{\ast}=\Hom_R(M,R)$ is also finitely-generated projective, being a direct summand of $(R^k)^{\ast}\cong R^k$.
There is a natural evaluation map
\begin{equation} \label{eq_eval}
\cap_M \ : \ M^{\ast}\otimes_R M\lra R,  \hspace{1cm} f\otimes m \mapsto f(m)\in R.
\end{equation}
It is now natural to compare $\cap_M$ to the map $\cap$ in \eqref{eq_cup_cap}.

There is an $R$-module map $\psi: N\lra M^{\ast}$ since each $n\in N$ gives the element $\psi(n)\in M^{\ast}$ satisfying $\psi(n)(m)=\cap(n\otimes m)$. Similarly, there is a map $\psi':M^{\ast}\lra N$ such that 
\[ 
\psi'(m^{\ast}) = \sum_{i=1}^k m^{\ast}(m_i) n_i.
\]
The composition $\psi'\psi$ takes $n\in N$ to
\[  
\sum_{i=1}^k \cap (n\otimes m_i) n_i = n,
\]
where the equality follows from the second relation in \eqref{eq_two_isot}. Consequently, $\psi'\psi=\id_N$. The other composition $\psi\psi'$ takes $m^{\ast}$ to
\[ 
\psi\psi'(m^{\ast}) = \sum_{i=1}^k m^{\ast}(m_i) \psi(n_i),
\]
which on $m'\in M$ evaluates to
\[ \sum_{i=1}^k m^{\ast}(m_i)\,  \cap(n_i\otimes m') = m^{\ast}(\sum_{i=1}^k   \cap(n_i\otimes m') m_i ) = m^{\ast}(m')
\]
due to \eqref{eq_m_decomp}. Thus, $\psi\psi'=\id_M$ and $\psi,\psi'$ are mutually-inverse isomorphisms between $N$ and $M^{\ast}$. Replacing $N$ by $M^{\ast}$ in \eqref{eq_cup_cap} via these isomorphisms produces cup and cap maps
\begin{equation}\label{eq_cup_cap_2}
    \cup_M : R\lra M\otimes M^{\ast}, \hspace{0.75cm}  \cap_M: M^{\ast} \otimes M \lra R,
\end{equation}
the \emph{coevaluation} and \emph{evaluation} maps, respectively. The map $\cap_M$ is given by \eqref{eq_eval}, while we can write the coevaluation map $\cup_M$ by fixing a realization of a finitely-generated projective module $M$ as a direct summand of free module $R^k$ via maps
\begin{equation} \label{eq_i_p}
M \stackrel{\iota}{\lra} R^k \stackrel{p}{\lra} M , \hspace{0.75cm} 
p\, \iota = \id_M ,
\end{equation}
and denoting the standard basis of $R^k$ by $(v_1,\ldots, v_k).$ The coevaluation map for $R^k$ is
\[  \mathsf{coev}_{R^k} \ : \ 1 \lra \sum_{i=1}^k v_k\otimes v^k,
\]
where $(v^1,\ldots, v^k)$ is the dual basis of $(R^k)^{\ast}\cong R^k$. The coevaluation map for $M$ is given by composing the above map with $p$ and $\iota^{\ast}$:
\begin{equation} \label{eq_coev} \mathsf{coev}_{M} \ : \ 1 \lra \sum_{i=1}^k p(v_k)\otimes \iota^{\ast}(v^k),  
\hspace{0.5cm} 
\iota^{\ast}: (R^k)^{\ast} \lra M^{\ast}.
\end{equation}
The coevaluation map $\cup_M$ given by \eqref{eq_coev}   does not depend on the choice of factorization \eqref{eq_i_p} since it is uniquely determined by the map $\cap_M$ in \eqref{eq_eval} and the relations
\begin{equation}\label{eq_two_isot_2}
( \id_{M} \otimes \cap_M)\circ(\cup_M\otimes \id_{M}) = \id_{M}, \ \ \
   (\cap_M \otimes \id_{M^{\ast}} )\circ(\id_{M^{\ast}}\otimes \cup_M) = \id_{M^{\ast}}.
\end{equation}
Thus, the coevaluation map  is computed  by writing $M$ as a direct summand of $R^k$ for some $k$ and identifying $M^{\ast}$ with the corresponding summand of $(R^k)^{\ast}\cong R^k$. Projection
\[ p\otimes\iota^{\ast}:R^k\otimes (R^k)^{\ast}\lra M\otimes M^{\ast}
\]
produces the coevaluation map for $M$ from that of  $R^k$.
We obtain the following well-known observation.

\begin{prop} \label{prop_classify} Let $R$ be a commutative ring. One-dimensional oriented TQFTs taking values in $R\dmod$ are classified by finitely-generated projective $R$-modules $P$. Such a TQFT associates $P$ to a positively-oriented point, $P^{\ast}$ to a negatively-oriented point, and evaluation and coevaluation maps $\cap_P, \cup_P$ in equations \eqref{eq_eval} and \eqref{eq_coev} with $M=P$ to the cup and cap cobordisms.
\end{prop}

A related observation is that the tensor product endofunctor $V\longmapsto M\otimes V$ in the category $R\dmod$ has a left adjoint if and only if $M$ is a finitely-generated projective $R$-module, see the answer by Q.~Yuan in ~\cite{Math-StackExchange-Brenin}. 

More generally, given a symmetric monoidal category $\mcC$, symmetric monoidal functors $\mcF: \Cob\lra \mcC$ correspond to pairs $M,N$ of objects of $\mcC$ such that 
\begin{itemize}
    \item the endofunctor of $\mcC$ of tensoring with $M$ is left adjoint to that of
    tensoring with $N$,
    \item a choice of the adjointness morphisms is made.
\end{itemize}
This is the one-dimensional case of the Baez--Dolan \emph{cobordism hypothesis}, see~\cite{BD95, Freed13}. Such a functor $\mcF$ selects a \emph{rigid} symmetric monoidal subcategory $\mcC_{\mcF}$ in $\mcC$, namely the full monoidal subcategory generated by objects $M$ and $N$ (so that objects of $\mcC_{\mcF}$ are arbitrary tensor products of $M$ and $N$). The category $\mcC$ itself may not be rigid, as the above example demonstrates ($R\dmod$, restricted to finitely-generated modules, is rigid if and only if $R$ is semisimple).

Note that the circle in this theory evaluates to the Hattori--Stallings (HS) rank 
\begin{equation}\label{eq_HS}
\mathsf{rk}(P) \ := \ \sum_{i=1}^k m_i^{\ast}(m_i)\in R
\end{equation} 
of $P$ where 
\begin{equation}
\cup_P(1) = \sum_{i=1}^k m_i^{\ast}\otimes m_i \in P\otimes P^{\ast}. 
\end{equation}  
When $P\cong R^k$ is a free module, its HS rank is $k\in R$. 

\begin{remark}
Replace $R\dmod$ by the category $\mathcal{C}^b(R\pdmod)$ of bounded complexes of projective $R$-modules up to chain homotopies. Then any bounded complex
\[
P=(\ldots \lra P_i \stackrel{d_i}{\lra} P_{i+1}\lra \ldots )
\]
such that all terms of $P$ are finitely-generated gives rise to a one-dimensional oriented TQFT with values in  $\mathcal{C}^b(R\pdmod)$, with $\mcF(+)=P$ and $\mcF(-)=P^{\ast}$, the termwise dual of $P$ with the induced differential. Complex $P^{\ast}$ has the module $P_i^{\ast}$ in degree $-i$ and the differential $d^{\ast}$ with $d_{-i}^{\ast}:P_i^{\ast}\lra P_{i-1}^{\ast}$ being the dual of $d_{i-1}$.
\end{remark}

{\it Generalizing to commutative semirings.}
The arguments leading to Proposition~\ref{prop_classify} did not use subtraction in the commutative ring $R$ and extend immediately to any commutative semiring $R$. The latter has addition and multiplication operations and elements $0,1$, satisfying the usual commutativity, associativity, distributivity axioms, but does not, in general, have subtraction. It is straightforward to define the category $R\dmod$ of $R$-modules (also called $R$-semimodules). We say that an $R$-module $P$ is \emph{finitely-generated projective} if it is a retract of a finite rank free module, so that maps $\iota,p$ below exist:  
\begin{equation} \label{eq_i_p_2}
P \stackrel{\iota}{\lra} R^k \stackrel{p}{\lra} P , \hspace{0.75cm} 
p\, \iota = \id_P . 
\end{equation}
Category $R\dmod$ is symmetric monoidal under the standard tensor product operation. The tensor product $\otimes$ of $R$-modules is especially inconvenient to work with when $R$ is not a ring,  and $M\otimes N$ is easier to understand when one of the modules is projective. 
Proposition~\ref{prop_classify} extends to TQFTs over commutative semirings: 

\begin{prop} \label{prop_classify2} Let $R$ be a commutative semiring. Isomorphism classes of one-dimensional oriented TQFTs taking values in $R\dmod$ are 
in a bijection with isomorphism classes of  finitely-generated projective $R$-modules $P$. Such a TQFT associates $P$ to a positively-oriented point, $P^{\ast}:=\Hom_R(P,R)$ to a negatively-oriented point, and evaluation and coevaluation maps $\cap_P, \cup_P$ in equations \eqref{eq_eval} and \eqref{eq_coev} with $M=P$ to the cup and cap cobordisms.
\end{prop}

The notion of the Hattori-Stallings rank $\mathsf{rk}(P)$ extends to f.g. projective modules over a commutative semiring, via \eqref{eq_HS}.

%%%%%%%%%%%%%%%%%%%%
% Adding floating endpoints
%%%%%%%%%%%%%%%%%%%%

\subsection{Floating endpoints, defects, and networks}
\label{subsec_floating} 

\quad 
\vspace{0.05in} 

{\it Half-intervals and floating endpoints.}
To slightly enrich the category $\Cob$ of one-cobordisms, one can allow one-manifolds to have some endpoints ``inside'' the cobordism rather than on the (outer) boundary. Such \emph{inside} points may also be called \emph{floating} endpoints of a cobordism. This leads to \emph{half-intervals}, cobordisms between the empty 0-manifold and a single point with the positive or the negative orientation, as shown in Figure~\ref{figure-1}. Half-intervals appear in~\cite{IK-top-automata,Kh3}, and, in a related context, in~\cite{KS1}. 

\vspace{0.1in} 

\begin{figure}
    \centering
\begin{tikzpicture}[scale=0.6]
\begin{scope}[shift={(0,0)}]
%\draw[thin,yellow] (0,0) grid (4,4);
\draw[thick,dashed] (0,3) -- (1,3);
\draw[thick,dashed] (0,0) -- (1,0);
\node at (0.5,3.35) {$+$};
\draw[thick,<-] (0.5,3) -- (0.5,1.5); 
\end{scope}

\begin{scope}[shift={(3.5,0)}]
%\draw[thin,yellow] (0,0) grid (4,4);
\draw[thick,dashed] (0,3) -- (1,3);
\draw[thick,dashed] (0,0) -- (1,0);
\node at (0.5,-0.35) {$+$};
\draw[thick,<-] (0.5,1.5) -- (0.5,0);
\end{scope}

\begin{scope}[shift={(7,0)}]
%\draw[thin,yellow] (0,0) grid (4,4);
\draw[thick,dashed] (0,3) -- (1,3);
\draw[thick,dashed] (0,0) -- (1,0);
\node at (0.5,3.35) {$-$};
\draw[thick,->] (0.5,3) -- (0.5,1.1); 
\end{scope}

\begin{scope}[shift={(10.5,0)}]
%\draw[thin,yellow] (0,0) grid (4,4);
\draw[thick,dashed] (0,3) -- (1,3);
\draw[thick,dashed] (0,0) -- (1,0);
\node at (0.5,-0.35) {$-$};
\draw[thick,->] (0.5,1.5) -- (0.5,0);
\end{scope}

\begin{scope}[shift={(14,0)}]
%\draw[thin,yellow] (0,0) grid (4,4);
\draw[thick,dashed] (0,3) -- (1,3);
\draw[thick,dashed] (0,1.5) -- (1,1.5);
\draw[thick,dashed] (0,0) -- (1,0);
\draw[thick,->] (0.5,2.5) -- (0.5,1.5);
\draw[thick,->] (0.5,1.5) -- (0.5,0.5);

\node at (1.5,1.5) {$=$};

\draw[thick,dashed] (2,3) -- (3,3);
\draw[thick,dashed] (2,0) -- (3,0);
\draw[thick,->] (2.5,2.5) -- (2.5,1.5);
\draw[thick] (2.5,1.5) -- (2.5,0.5);

\node at (3.5,1.5) {$=$};

\draw[thick,dashed] (4,3) -- (5,3);
\draw[thick,dashed] (4,1.5) -- (5,1.5);
\draw[thick,dashed] (4,0) -- (5,0);
\draw[thick,<-] (4.5,2.5) -- (4.5,1.5);
\draw[thick,<-] (4.5,1.5) -- (4.5,0.5);
\end{scope}

\end{tikzpicture}
    \caption{Left: four types of half-intervals. Right: composing two half-intervals results in a floating interval, with both endpoints inner.}
    \label{figure-1}
\end{figure}
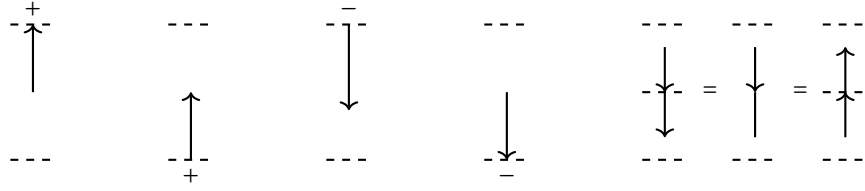

Composing two half-intervals results in a \emph{floating} interval, an oriented interval with both endpoints inside the cobordism, see Figure~\ref{figure-2}. A floating interval is an endomorphism of the empty 0-manifold $\emptyset_0$. It is clear how to define the corresponding category of oriented 1-cobordisms with inner endpoints. We denote this category by $\Cob_{\I}$. It has the same objects as $\Cob$.  It is rigid symmetric monoidal, just like $\Cob$, and contains the latter as a subcategory.

A TQFT for the category $\Cob_{\I}$ is a symmetric monoidal functor 
\begin{equation}\label{eq_func_int}
  \mcF \ : \ \Cob_{\I} \lra R\dmod . 
\end{equation}
It is determined by the same data as in Proposition~\ref{prop_classify} plus a choice of an element $v_0\in P$, giving an $R$-module map $R\lra P$,  and a $R$-module map $\widetilde{v}_0: P\lra R$.  These maps are associated to the two half-intervals that bound a $+$ boundary point, see Figure~\ref{figure-2}. The maps for the other two half-intervals are obtained by dualizing these maps. 

A floating interval evaluates to $\widetilde{v}_0(v_0)\in R$. A circle evaluates to $\mathsf{rk}(P)$, as before, see \eqref{eq_HS}. 

\begin{figure}
    \centering
\begin{tikzpicture}[scale=0.6]

\begin{scope}[shift={(0,0)}]
%\draw[thin,yellow] (0,0) grid (4,4);
\node at (0.5,3.35) {$+$};
\draw[thick,dashed] (0,3) -- (1,3);
\draw[thick,dashed] (0,0) -- (1,0);
\draw[thick,<-] (0.5,3) -- (0.5,1.5); 

\node at (2,3) {$P$};
\draw[thick,<-] (2,2.5) -- (2,0.5);
\node at (2,0) {$R$};

\node at (2.8,3) {$\ni$};
\node at (3.5,3) {$v_0$};
\draw[thick,<-|] (3.5,2.5) -- (3.5,0.5);
\node at (2.8,0) {$\ni$};
\node at (3.5,0) {$1$};
\end{scope}

\begin{scope}[shift={(8,0)}]
%\draw[thin,yellow] (0,0) grid (4,4);
\node at (0.5,-0.35) {$+$};
\draw[thick,dashed] (0,3) -- (1,3);
\draw[thick,dashed] (0,0) -- (1,0);
\draw[thick,->] (0.5,0) -- (0.5,1.5); 

\node at (2,3) {$R$};
\draw[thick,<-] (2,2.5) -- (2,0.5);
\node at (2.60,1.5) {$\widetilde{v}_0$};
\node at (2,0) {$P$};

\node at (2.8,3) {$\ni$};
\node at (3.0,0) {$\ni$};

\node at (4,3) {$\widetilde{v}_0(v)$};
\draw[thick,<-|] (4,2.5) -- (4,0.5);
\node at (4,0) {$v$};
\end{scope}

\begin{scope}[shift={(17,0)}]
%\draw[thin,yellow] (0,0) grid (4,4);
\draw[thick,dashed] (0,3) -- (1,3);
\draw[thick,dashed] (0,0) -- (1,0);
\draw[thick,<-] (0.5,2.35) -- (0.5,0.65);
\node at (1.5,1.5) {$=$};
\node at (3,1.5) {$\widetilde{v}_0(v_0)$};
\draw[thick,dashed] (4,3) -- (5,3);
\draw[thick,dashed] (4,0) -- (5,0);
\end{scope}

\end{tikzpicture}
    \caption{Left: two half-intervals and associated maps. Right: evaluation of an interval.}
    \label{figure-2}
\end{figure}
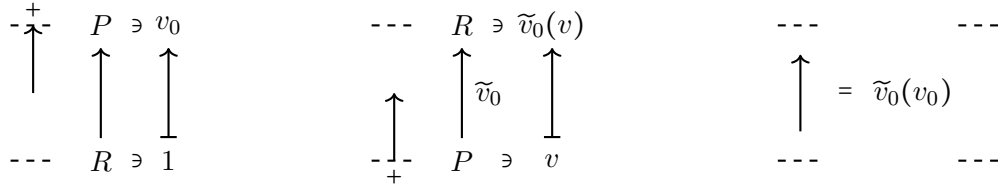

\vspace{0.1in} 

{\it Defects.}
Alternatively, we can extend the category $\Cob$ by adding defect points labelled by elements of a set $\Sigma$, resulting in a category $\CobS$.  To extend a functor $\mcF:\Cob\lra R\dmod$ to a functor $\mcF:\CobS\lra R\dmod$ (using the same notation for both functors), we pick an endomorphism $m_a:P\lra P$ for each $a\in \Sigma$. To a defect point labelled $a$ on an upward-oriented  interval we associate the map $m_a$ and to a defect point labelled $a$ on a downward-oriented interval associate the dual map $m_a^{\ast}:P^{\ast}\lra P^{\ast}$, see Figure~\ref{figure-3}. 

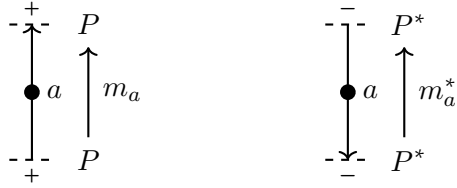
\begin{figure}
    \centering
\begin{tikzpicture}[scale=0.6]
\begin{scope}[shift={(0,0)}]
%\draw[thin,yellow] (0,0) grid (4,4);
\node at (0.5, 3.35) {$+$};
\node at (0.5,-0.35) {$+$};
\draw[thick,dashed] (0,3) -- (1,3);
\draw[thick,dashed] (0,0) -- (1,0);
\draw[thick,->] (0.5,0) -- (0.5,3); 
\draw[thick,fill] (0.65,1.5) arc (0:360:1.5mm);
\node at (1.0,1.5) {$a$};

\node at (1.75,3) {$P$};
\draw[thick,<-] (1.75,2.5) -- (1.75,0.5);
\node at (2.5,1.5) {$m_a$};
\node at (1.75,0) {$P$};
\end{scope}

\begin{scope}[shift={(7,0)}]
%\draw[thin,yellow] (0,0) grid (4,4);
\node at (0.5, 3.35) {$-$};
\node at (0.5,-0.35) {$-$};
\draw[thick,dashed] (0,3) -- (1,3);
\draw[thick,dashed] (0,0) -- (1,0);
\draw[thick,<-] (0.5,0) -- (0.5,3); 
\draw[thick,fill] (0.65,1.5) arc (0:360:1.5mm);
\node at (1.0,1.5) {$a$};

\node at (1.85,3) {$P^*$};
\draw[thick,<-] (1.75,2.5) -- (1.75,0.5);
\node at (2.5,1.5) {$m_a^*$};
\node at (1.85,0) {$P^*$};
\end{scope}

\end{tikzpicture}
    \caption{To a labelled dot on an interval associate an endomorphism of $P$ and the dual endomorphism of $P^{\ast}$ for the oppositely oriented interval.}
    \label{figure-3}
\end{figure}
 
\vspace{0.1in}

Now an upward-oriented interval decorated by a word $\omega=a_1\cdots a_n$, $a_i\in \Sigma$, goes to the map $m_{\omega}=m_{a_1}\cdots \, m_{a_n}:P\lra P$ under the  functor $\mcF$, see Figure~\ref{figure-4}.  A circle decorated by $\omega$ evaluates to $\tr_P(m_\omega)$, the trace of operator $m_{\omega}$ on $P$.  

\vspace{0.1in}

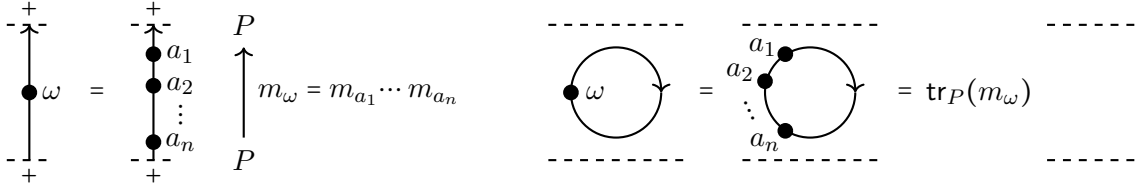
\begin{figure}
    \centering
\begin{tikzpicture}[scale=0.6]
\begin{scope}[shift={(0,0)}]
%\draw[thin,yellow] (0,0) grid (4,4);
\node at (0.5, 3.35) {$+$};
\node at (0.5,-0.35) {$+$};
\draw[thick,dashed] (0,3) -- (1,3);
\draw[thick,dashed] (0,0) -- (1,0);
\draw[thick,->] (0.5,0) -- (0.5,3); 
\draw[thick,fill] (0.65,1.5) arc (0:360:1.5mm);
\node at (1.0,1.5) {$\omega$};

\node at (2,1.5) {$=$};

\node at (3.25, 3.35) {$+$};
\node at (3.25,-0.35) {$+$};
\draw[thick,dashed] (2.75,3) -- (3.75,3);
\draw[thick,dashed] (2.75,0) -- (3.75,0);
\draw[thick,->] (3.25,0) -- (3.25,3);

\draw[thick,fill] (3.40,2.35) arc (0:360:1.5mm);
\node at (3.85,2.35) {$a_1$};
\draw[thick,fill] (3.40,1.65) arc (0:360:1.5mm);
\node at (3.85,1.65) {$a_2$};
\node at (3.85,1.0) {$\vdots$};
\draw[thick,fill] (3.40,0.40) arc (0:360:1.5mm);
\node at (3.85,0.40) {$a_n$};

\node at (5.25,3) {$P$};
\draw[thick,<-] (5.25,2.5) -- (5.25,0.5);
\node at (7.80,1.45) {$m_{\omega} = m_{a_1}\cdots \:m_{a_n}$};
\node at (5.25,0) {$P$};
\end{scope}

\begin{scope}[shift={(12,0)}]
%\draw[thin,yellow] (0,0) grid (4,4);
\draw[thick,dashed] (0,3) -- (3,3);
\draw[thick,dashed] (0,0) -- (3,0);
\draw[thick,<-] (2.5,1.5) arc (0:360:1);
\draw[thick,fill] (0.65,1.5) arc (0:360:1.5mm);
\node at (1.05,1.5) {$\omega$};

\node at (3.5,1.5) {$=$};

\begin{scope}[shift={(4.3,0)}]
%\draw[thin,yellow] (0,0) grid (4,4);
\draw[thick,dashed] (0,3) -- (3,3);
\draw[thick,dashed] (0,0) -- (3,0);
\draw[thick,<-] (2.5,1.5) arc (0:360:1);

\draw[thick,fill] (1.10,2.35) arc (0:360:1.5mm);
\node at (0.45,2.50) {$a_1$};
\draw[thick,fill] (0.65,1.75) arc (0:360:1.5mm);
\node at (-0.05,1.95) {$a_2$};
\node at (0.15,1.1) {\rotatebox[origin=c]{-20}{$\ddots$}};
\draw[thick,fill] (1.10,0.65) arc (0:360:1.5mm);
\node at (0.50,0.40) {$a_n$};

\node at (3.6,1.5) {$=$};
\node at (5.25,1.5) {$\tr_P(m_{\omega})$};
\draw[thick,dashed] (6.75,3) -- (8.75,3);
\draw[thick,dashed] (6.75,0) -- (8.75,0);
\end{scope}

\end{scope}

\end{tikzpicture}
    \caption{Map $m_{\omega}$ is on the left. Evaluation of $\omega$-decorated circle is on the right. }
    \label{figure-4}
\end{figure}

\begin{figure}
    \centering
\begin{tikzpicture}[scale=0.6]
\begin{scope}[shift={(0,0)}]
%\draw[thin,yellow] (0,0) grid (6.5,4);
\draw[thick,dashed] (0,4) -- (6.5,4);
\draw[thick,dashed] (0,0) -- (6.5,0);

\node at (0.5,4.35) {$+$};
\node at (1.5,4.35) {$-$};
\node at (2.5,4.35) {$-$};

\node at (0.5,-0.35) {$+$};
\node at (1.5,-0.35) {$-$};
\node at (2.5,-0.35) {$-$};
\node at (3.5,-0.35) {$+$};

\draw[thick,->] (0.5,0) .. controls (0.6,2) and (2.4,2) .. (2.5,0);
\draw[thick,fill] (0.85,0.85) arc (0:360:1.5mm);
\node at (0.20,1.05) {$a$};
\draw[thick,fill] (2.48,0.85) arc (0:360:1.5mm);
\node at (2.75,1.05) {$b$};

\draw[thick,<-] (0.5,4) .. controls (0,2.75) and (0.75,2.5) .. (0.5,2);
\draw[thick,fill] (0.48,3) arc (0:360:1.5mm); 
\node at (-0.25,3) {$a$};

% Old version
% \draw[thick,->] (1.5,4) .. controls (1.75,3) and (1.75,2.75) .. (1.75,2.49);
% \draw[thick,fill] (1.80,3.30) arc (0:360:1.5mm); 
% \node at (2.10,3.35) {$b$};
% \draw[thick] (1.75,2.5) .. controls (1.75,2.25) and (1.75,2.00) .. (1.5,1.75);
% \draw[thick] (1.4,1.15) .. controls (1.25,0.85) and (1.25,0.25) .. (1.5,0);

% New version
\draw[thick] (1.5,4) -- (1.5,1.65);
\draw[thick,->] (1.5,1.4) -- (1.5,0);
\draw[thick,fill] (1.65,3) arc (0:360:1.5mm); 
\node at (1.9,3) {$b$};

\draw[thick,->] (2.5,4) .. controls (3.5,3.5) and (3.5,2.25) .. (3.5,2);
\draw[thick,fill] (3.40,3.25) arc (0:360:1.5mm); 
\node at (3.65,3.5) {$a$};

\draw[thick,<-] (2.75,1.75) .. controls (3.5,1.50) and (3.5,0.25) .. (3.5,0);
\draw[thick,fill] (3.65,0.75) arc (0:360:1.5mm);
\node at (3.88,1) {$c$};

\draw[thick,->] (4.5,2.75) .. controls (4.75,3.25) and (5.75,2.25) .. (6.50,3.25);
\draw[thick,fill] (5.15,2.85) arc (0:360:1.5mm);
\node at (5,3.35) {$b$};
\draw[thick,fill] (6.00,2.80) arc (0:360:1.5mm);
\node at (5.85,3.25) {$a$};

\draw[thick,<-] (6.55,1) arc (0:360:0.75); 
\draw[thick,fill] (6.45,1.55) arc (0:360:1.5mm);
\node at (6.65,1.95) {$b$};
\draw[thick,fill] (5.35,1.45) arc (0:360:1.5mm);
\node at (4.85,1.85) {$a$};
\draw[thick,fill] (5.40,0.50) arc (0:360:1.5mm);
\node at (4.80,0.35) {$c$};

\end{scope}

\end{tikzpicture}
    \caption{A morphism from oriented 1-manifold $(+--+)$ to $(+--)$ in category $\Cob_{\Sigma,\I}$.}
    \label{figure-A1}
\end{figure}
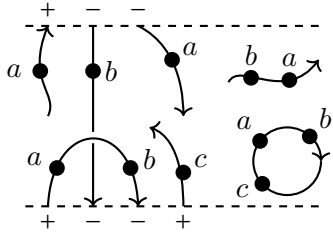

Finally, it is easy to combine defects with endpoints. This leads to the category $\CobSI$ of oriented one-cobordisms with $\Sigma$-defects and  inner endpoints. Note that defects are placed away from the endpoints. Figure~\ref{figure-A1} shows an example of a morphism in this category. 

There is a commutative square of faithful inclusions of categories, where the inclusions are identities on objects. In all four categories the objects are  are sign sequences. 
\begin{equation} \label{eq_cd_1}
\begin{CD}
\CobS  @>>> \CobSI \\
  @AAA   @AAA  \\    
 \Cob  @>>>  \Cob_{\I}
\end{CD}
\end{equation}

The functor 
\begin{equation}\mcF:\Cob_{\Sigma,\I}\lra R\dmod
\end{equation}
takes half-intervals to maps given by an element $v_0\in P$ and  a module map $\widetilde{v}_0 \in P^{\ast}$ (same as for the functor in \eqref{eq_func_int}) and an upward-oriented interval decorated by $a$ to $m_a$, as in Figure~\ref{figure-3}.  
A floating interval with a word $\omega$ on it evaluates to $\widetilde{v}_0(m_\omega v_0)$, see Figure~\ref{figure-5}. 

\vspace{0.1in}

\begin{figure}
    \centering
\begin{tikzpicture}[scale=0.6]
\begin{scope}[shift={(0,0)}]
%\draw[thin,yellow] (0,0) grid (4,4);
\draw[thick,dashed] (0,3) -- (3,3);
\draw[thick,dashed] (0,0) -- (3,0);

\draw[thick,<-] (0,1.5) -- (3,1.5);
\draw[thick,fill] (1.65,1.5) arc (0:360:1.5mm);
\node at (1.5,0.9) {$\omega$};

\node at (4,1.5) {$=$};

\begin{scope}[shift={(5,0)}]
%\draw[thin,yellow] (0,0) grid (4,4);
\draw[thick,dashed] (0,3) -- (3,3);
\draw[thick,dashed] (0,0) -- (3,0);

\draw[thick,<-] (0,1.5) -- (3,1.5);

\draw[thick,fill] (0.85,1.5) arc (0:360:1.5mm);
\node at (0.70,0.9) {$a_1$};
%\draw[thick,fill] (1.65,1.5) arc (0:360:1.5mm);
\node at (1.50,0.9) {$\cdots$};
\draw[thick,fill] (2.65,1.5) arc (0:360:1.5mm);
\node at (2.50,0.9) {$a_n$};

\node at (4,1.5) {$=$};

\node at (6.25,1.5) {$\widetilde{v}_0(m_{\omega}v_0)$};
\draw[thick,dashed] (8,3) -- (11,3);
\draw[thick,dashed] (8,0) -- (11,0);

\node at (12,1.5) {$=$};
\node at (15,1.5) {$\widetilde{v}_0(m_{\omega}v_0)\,\id_{\emptyset_0}$};
\end{scope}

\end{scope}

\end{tikzpicture}
    \caption{Evaluation of a floating $\omega$-decorated interval (both endpoints are inner). }
    \label{figure-5}
\end{figure}

\begin{remark} 
If $R=\kk$ is a field, the above choices are further simplified. That is, $P=V\cong R^n$ is a finite-dimensional $\kk$-vector space, $v_0\in V$ is a vector and $\widetilde{v}_0\in V^{\ast}$ is a covector. Maps $m_a:V\lra V$ are linear operators on $V$. The rank $\mathsf{rk}(V)=n\in \kk$. 
\end{remark} 

{\it Commutative semirings.}
One can replace the commutative ring $R$ by a commutative semiring $R$ and work with the category $\R\dmod$ of (semi)modules over $R$. Propositions~\ref{prop_classify} and~\ref{prop_classify2} can then be extended as follows. 

\begin{prop} \label{prop_classify3} Let $R$ be a commutative semiring. One-dimensional oriented TQFTs with inner points and $\Sigma$-defects taking values in $R\dmod$, 
that is, symmetric monoidal functors 
\[
\mcF \ : \ \CobSI \lra \R\dmod
\]
are classified by finitely-generated projective $R$-modules $P$ equipped with a vector $v_0\in R$, a covector $\widetilde{v}_0:P\lra R$, and endomorphisms $m_a\in \End_R(P), a\in \Sigma$. 
Functor $\mcF$ associates $P$ to a positively-oriented point, $P^{\ast}$ to a negatively-oriented point, evaluation and coevaluation maps $\cap_P, \cup_P$ in equations \eqref{eq_eval} and \eqref{eq_coev} with $M=P$ to the cup and cap cobordisms, map $R\lra P, 1\mapsto v_0$ and covector $\widetilde{v}_0$ to half-interval cobordisms and the map $m_a$ to a dot labelled $a$ on the upward-oriented interval. 
\end{prop}

Symmetric monoidal functors from the other two categories $\CobS,\Cob_{\I}$ in \eqref{eq_cd_1} to $R\dmod$ are classified similarly. 

\vspace{0.1in} 

{\it Quantum mechanics and one-dimensional TQFTs with defects.}
Quantum mechanics can be interpreted as a one-dimensional Quantum Field Theory, see, \textit{e.g.},~\cite{Freed13,Booz07,Skinner2018}, \cite[Chapter 10]{hori03}.  
Part of the structure of quantum mechanics is a separable Hilbert space $\mcH$, the ground state $\Omega\in\mcH$, a collection of self-adjoint operators $\{\mcO\}$, and the Hamiltonian, which is a self-adjoint operator $H:\mcH\lra \mcH$ giving rise to unitary \emph{evolution} operators $U_t = \exp(-i t \mcH/\hbar)$. Operators $\{\mcO\}$ and $\mcH$ are the \emph{observables} of the system. 

Information about the system is encoded in \emph{(vacuum) expectation values} 
\begin{equation}\label{eq_exp_value}
\langle \Omega, U_{t_n}\mcO_n U_{t_{n-1}}\cdots U_{t_1}\mcO_1U_{t_0} \Omega \rangle ,
\end{equation}
see Figure~\ref{figure-A2} (also see~\cite[Figure 6]{Freed13}), 
where the state $\Omega$ evolves for times $t_0,t_1,\ldots, t_n$ and in between acted upon by operators $\mcO_1,\ldots, \mcO_n$. At the end the inner product with $\Omega$ is computed. 

\vspace{0.1in} 

\begin{figure}
    \centering
\begin{tikzpicture}[scale=0.6]
\begin{scope}[shift={(0,0)}]
%\draw[thin,yellow] (0,0) grid (6,4);
\draw[thick,dashed] (0,3) -- (6,3);
\draw[thick,dashed] (0,0) -- (6,0);

\draw[thick] (0,1.5) -- (6,1.5);
\draw[thick,fill] (0.15,1.5) arc (0:360:1.5mm);
\draw[thick,fill,white] (1.40,1.5) arc (0:360:1.5mm);
\draw[thick] (1.40,1.5) arc (0:360:1.5mm);
\draw[thick,fill,white] (2.65,1.5) arc (0:360:1.5mm);
\draw[thick] (2.65,1.5) arc (0:360:1.5mm);
\draw[thick,fill,white] (4.85,1.5) arc (0:360:1.5mm);
\draw[thick] (4.85,1.5) arc (0:360:1.5mm);
\draw[thick,fill] (6.15,1.5) arc (0:360:1.5mm);

\node at (-0.30,1) {$\Omega$};
\node at (0.65,1) {$t_0$};
\node at (1.95,1) {$t_1$};
\node at (3.7,0.9) {$\cdots$};
\node at (5.45,1) {$t_n$};
\node at (6.30,1) {$\Omega$};

\node at (1.30,2.15) {$\mathcal{O}_1$};
\node at (2.55,2.15) {$\mathcal{O}_2$};
\node at (3.7,2.15) {$\cdots$};
\node at (4.7,2.15) {$\mathcal{O}_n$};

\end{scope}    
\end{tikzpicture}
    \caption{Vacuum expectation values in quantum mechanics, see~\cite[Figure 6]{Freed13}. 
    }
    \label{figure-A2}
\end{figure}
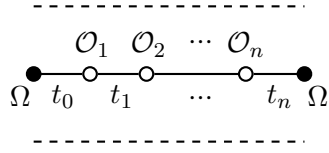

For a finite-dimensional Hilbert space, this setup is very close to the one discussed above, see Figure~\ref{figure-5} in particular. A parameter analogous to time can be added there by picking a commutative group or monoid $G$ with a homomorphism $\phi: G\lra \GL(P)$ into the group of automorphisms of $P$. Elements $g$ of $G$ now play the role of time $t\in\R$, and the analogue of time evolution in Figure~\ref{figure-A2} is shown in Figure~\ref{figure-A3} on the left, with $g_i\in G$. 

\vspace{0.1in} 

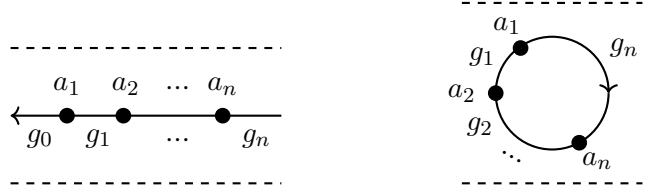
\begin{figure}
    \centering
\begin{tikzpicture}[scale=0.6]
\begin{scope}[shift={(0,0)}]
%\draw[thin,yellow] (0,0) grid (6,4);
\draw[thick,dashed] (0,3) -- (6,3);
\draw[thick,dashed] (0,0) -- (6,0);

\draw[thick,<-] (0,1.5) -- (6,1.5);

\draw[thick,fill] (1.40,1.5) arc (0:360:1.5mm);
\draw[thick,fill] (2.65,1.5) arc (0:360:1.5mm);
\draw[thick,fill] (4.85,1.5) arc (0:360:1.5mm);
 
\node at (0.65,1) {$g_0$};
\node at (1.95,1) {$g_1$};
\node at (3.7,0.9) {$\cdots$};
\node at (5.45,1) {$g_n$};
 
\node at (1.25,2.15) {$a_1$};
\node at (2.55,2.15) {$a_2$};
\node at (3.7,2.05) {$\cdots$};
\node at (4.7,2.15) {$a_n$};
\end{scope}    

\begin{scope}[shift={(10,0)}]
%\draw[thin,yellow] (0,0) grid (4,4);
\draw[thick,dashed] (0,4) -- (4,4);
\draw[thick,dashed] (0,0) -- (4,0);

\draw[thick,<-] (3.25,2) arc (0:360:1.25);

\draw[thick,fill] (1.45,3) arc (0:360:1.5mm);
\draw[thick,fill] (0.9,2) arc (0:360:1.5mm);
\draw[thick,fill] (2.75,0.9) arc (0:360:1.5mm);

\node at (0.4,1.25) {$g_2$};
\node at (0.45,2.80) {$g_1$};
\node at (1.10,0.60) {\rotatebox[origin=c]{60}{$\vdots$}};
\node at (3.6,3) {$g_n$};

\node at (1.00,3.50) {$a_1$};
\node at (0.00,2.00) {$a_2$};
\node at (3,0.5) {$a_n$};

\end{scope}

\end{tikzpicture}
    \caption{Left: Interval diagram for a decorated one-dimensional TQFT with defects and time-like parameters added. Right: A circle diagram with these decorations.}
    \label{figure-A3}
\end{figure}

The diagram in Figure~\ref{figure-A3} (left) evaluates to 
\[
 \widetilde{v}_0(\phi(g_0)  m_{a_1} \phi(g_1) m_{a_2} \cdots m_{a_n} \phi(g_n) v_0).
\]
Elements of $G$ can be inserted into a circle with defects as well, decorating intervals between defects also. The circle then evaluates to the trace of the operator 
$\phi(g_n)m_{a_1}\phi(g_1)\cdots \phi(g_{n-1})m_{a_n}$ on $P$, see Figure~\ref{figure-A3} (right). 

The bigger category where intervals between defects are labelled by elements of $G$ can be denoted $\Cob_{\Sigma,\I,G}$, and the above choices give a functor from this category to $R\dmod$. Category $\CobSI$ is a subcategory of $\Cob_{\Sigma,\I,G}$, with morphisms obtained  by specializing to decorations $g_i=1$ for all intervals. 

\begin{remark} For a commutative ring or semiring $R$, one-dimensional defect TQFTs associated to projective $R$-modules $P$ with additional data as above can be unified into an oriented graph TQFT, see Figure~\ref{figure-A4}. Vertices of a graph can be decorated by intertwiners between tensor products of projective modules, one for each edge of the graph. Edges of a graph will carry defects, as above. For a given finitely-generated projective module $P$, one can allow more than one type of endpoints by picking a set of elements $\{v_i\}, v_i\in P$, instead of $v_0$ and labelling the ``in'' endpoints of $P$-intervals by $i$, and likewise for the ``out'' endpoints of $P$-intervals. In particular, vertices of valency $2$ then correspond to intertwiners between projective modules, and defects of the original form correspond to endomorphisms of projective modules. 

\begin{figure}
    \centering
\begin{tikzpicture}[scale=0.6]
\begin{scope}[shift={(0,0)}]
%\draw[thin,yellow] (0,0) grid (4,4);
\draw[thick,dashed] (0,3) -- (3,3);
\draw[thick,dashed] (0,0) -- (3,0);

\draw[thick,->] (0.5,0.5) -- (1,1);
\draw[thick] (1,1) -- (1.5,1.5);

\draw[thick,->] (2.5,0.5) -- (2,1);
\draw[thick] (2,1) -- (1.5,1.5);
\draw[thick,fill] (1.65,1.5) arc (0:360:1.5mm);
\draw[thick,->] (1.5,1.5) -- (1.5,2.15);
\draw[thick] (1.5,2.15) -- (1.5,2.5);

\node at (0,0.5) {$P_1$};
\node at (3,0.5) {$P_2$};
\node at (2.15,2.35) {$P_3$};

\node at (5,2.35) {$P_3$};
\draw[thick,<-] (5,1.85) -- (5,0.85);
\node at (5,0.5) {$P_1\otimes P_2$};
\end{scope}

\begin{scope}[shift={(8.5,0)}]
%\draw[thin,yellow] (0,0) grid (4,4);
\draw[thick,dashed] (0,3) -- (1,3);
\draw[thick,dashed] (0,0) -- (1,0);

\node at (1,2.35) {$P_2$};
\draw[thick] (0.5,0.5) -- (0.5,2.5);
\draw[thick,fill] (0.65,1.5) arc (0:360:1.5mm);
\node at (1,0.5) {$P_1$};

\node at (2.5,2.35) {$P_2$};
\draw[thick,<-] (2.5,1.85) -- (2.5,0.95);
\node at (2.5,0.5) {$P_1$};
\end{scope}

\begin{scope}[shift={(14,0)}]
%\draw[thin,yellow] (0,0) grid (4,4);
\node at (-0.20,2.05) {$P$};
\node at (0.5,3.45) {$+$};
\draw[thick,dashed] (0,3) -- (1,3);
\draw[thick,dashed] (0,0) -- (1,0);

\draw[thick,->] (0.5,1) -- (0.5,2.25);
\draw[thick] (0.5,2.20) -- (0.5,3.0);
\node at (1,1) {$i$};

\node at (2.5,3.35) {$-$};
\draw[thick,dashed] (2,3) -- (3,3);
\draw[thick,dashed] (2,0) -- (3,0);

\node at (3.1,2) {$P$};
\draw[thick] (2.5,1) -- (2.5,2.05);
\draw[thick,<-] (2.5,2.0) -- (2.5,3.0);
\node at (3,1) {$j$};
\end{scope}

\begin{scope}[shift={(19.5,0)}]
%\draw[thin,yellow] (0,0) grid (4,4);

\draw[thick,->] (1,1) .. controls (0.15,1.25) and (0.15,2.75) .. (1,3);
\draw[thick,fill] (0.65,2.55) arc (0:360:1.5mm);
\draw[thick,fill] (0.65,1.45) arc (0:360:1.5mm);
\node at (0.85,2) {$P_1$};

\draw[thick] (1,1) .. controls (1.85,1.25) and (1.85,2.75) .. (1,3);
\draw[thick,fill] (1.65,2.55) arc (0:360:1.5mm);
\draw[thick,fill] (1.65,1.45) arc (0:360:1.5mm);
\node at (2.10,2) {$P_2$};

\draw[thick,->] (1,3) .. controls (1.25,4) and (2.25,4) .. (2.5,4);
\draw[thick] (1,1) .. controls (1.25,0) and (2.25,0) .. (2.5,0);

\draw[thick] (2.5,4) .. controls (4,3.75) and (4,0.25) .. (2.5,0);
\draw[thick,fill] (3.5,3.25) arc (0:360:1.5mm);
\draw[thick,fill] (2.65,0) arc (0:360:1.5mm);
\node at (3.30,4.15) {$P_3$};

\draw[thick,fill] (3.80,2) arc (0:360:1.5mm);
\draw[thick] (3.65,2) -- (4.15,2.55);
\draw[thick,<-] (4.145,2.545) -- (4.5,3);
\node at (4.65,3.4) {$i$};
\node at (4.35,2.00) {$P_4$};

\end{scope}

\end{tikzpicture}
    \caption{From left to right: a trivalent vertex for an intertwiner $P_1\otimes P_2\lra P_3$, a valency two vertex is an intertwiner between two modules, a vector $v_i$  in $P$ and a covector $v_j:P\lra R$,  a closed network of intertwiners.  }
    \label{figure-A4}
\end{figure}
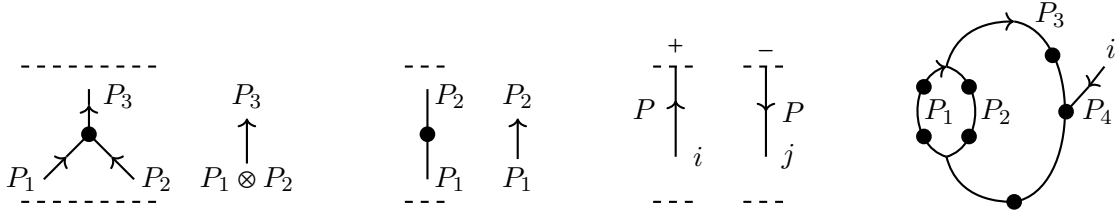

\end{remark}

%%%%%%%%%%%%%%
%
%  FSA and TQFT 
%
%%%%%%%%%%%%%%

\section{Finite state automata and one-dimensional TQFTs over the Boolean semiring}\label{sec_aut_TQFT}

%%%%%%%%%%%%%%%
% TQFT FROM NFA 
%%%%%%%%%%%%%%%

\subsection{A one-dimensional TQFT from a nondeterministic finite automaton} \label{subset_oned_from}
\quad 

\vspace{0.05in} 

{\it Regular languages and automata.}
Given a finite set $\Sigma$ of letters, by a \emph{language} or \emph{interval language} we mean a subset $L\subset \Sigma^{\ast}$ of the free monoid on $\Sigma$. A language is called \emph{regular} if it is accepted by a finite-state automaton, equivalently, if it can be described by a regular expression \cite{Kleene56, eilenberg1974automata, Conway71}.  

Suppose $L_\I\subset \Sigma^{\ast}$  is a regular language and $(Q,\delta,Q_{\init},Q_{\t})$ is a nondeterministic finite automaton (NFA) accepting $L_\I$. Here $Q$ is a finite set of states, $\delta:Q\times \Sigma\lra \mathscr{P}(Q)$ is the transition function $(\mathscr{P}(Q)$ is the powerset of $Q$), and $Q_{\init},Q_{\t}\subset Q$ are the subsets of initial, respectively accepting, states.   
We denote $(Q,\delta,Q_{\iin},Q_{\t})$ by $(Q)$, for short. It can be thought of as a decorated oriented graph, denoted $\Gamma_{(Q)}$ or $\Gamma(Q)$,  with the set of vertices $Q$, a directed edge from each state $q$ to each $q'\in \delta(q,a)$ marked by $a\in \Sigma$, and subsets $Q_{\iin},Q_{\t}$ of distinguished vertices. 

A word $\omega\in L_\I$ if and only if there is a path in $\Gamma(Q)$ from some initial to some accepting state where consequent letters $a_1,\ldots, a_n$ in the path read $\omega=a_1\cdots a_n$. 

\vspace{0.1in} 

{\it $\Bool$-modules.} 
Let $\Bool=\{0,1|1+1=1\}$ be the Boolean semiring on two elements.

A $\Bool$-module $M$ is an abelian idempotent monoid: $x+x=x$ for any $x\in M$, and the unit element is denoted $0$, $0+x=x$. Such $M$ comes with a partial order $x\le y$ if and only if $x+y=y$ making $M$ into a \emph{sup-semilattice} with $0$, where $x\vee y := x+y$. 
Morphisms in the category $\Bool\dmod$ of $\Bool$-modules take $0$ to $0$ and respect addition. Finite free $\Bool$-module is isomorphic to $\Bool^n$, the column module with elements $(x_1,\ldots, x_n)^T, x_i\in \Bool$, with termwise addition and multiplication by elements of $B$. Morphisms $\Bool^m\lra \Bool^n$ are classified by $n\times m$ Boolean matrices, with the usual addition and product rules.   

For further information and references on Boolean (semi)modules, we refer to~\cite[Section 3]{IK-top-automata}.  

\vspace{0.1in} 

{\it Boolean linear algebra from an automaton.}
To automaton $(Q)$ associate the free $\Bool$-module $\Bool Q$ with the basis $Q$ and consider the dual free module $(\Bool Q)^{\ast}\cong \Bool {Q^{\ast}}$, where set $Q^{\ast}$ consists of elements $q^{\ast}$, $q\in Q$. The bilinear pairing between these two free $\Bool$-modules is given by $q^{\ast}(q')=\delta_{q,q'}$. 

Transition function $\delta$ describes a right action of the free monoid $\Sigma^{\ast}$ on $\Bool Q$, with 
\begin{equation} \label{eq_sum} 
q\, a=\sum_{q'\in \delta(q,a)} q'.
\end{equation} 

The set $Q_{\iin}$ of initial states gives the initial vector $\sum_{q\in Q_{\iin}} q$, also denoted $Q_{\iin}$.  

The set $Q_{\t}$ of accepting states describes a $\Bool$-linear map 
\begin{equation}\label{eq_trace}
Q_{\t}^{\ast} :\Bool Q\lra \Bool, \hspace{1cm} 
Q_{\t}^{\ast}(q)=
\begin{cases} 1 & \mathrm{if} \ q\in Q_{\t}, \\
   0  & \mathrm{otherwise}. 
   \end{cases} 
\end{equation}
Any word $\omega\in\Sigma^{\ast}$ can be applied letter-by-letter to the initial state $Q_{\iin}\in \Bool Q$ and then evaluated via $Q_{\t}^{\ast}$ , resulting in a map  
\[  \alphaiQ \ : \ \Sigma^{\ast}\lra \Bool, 
\hspace{0.5cm}
\omega \longmapsto Q_{\t}^{\ast}(Q_{\iin}\omega), 
\hspace{0.5cm} 
\omega \in \Sigma^{\ast}. 
\]

A word $\omega \in L_\I$ if and only if  $\alphaiQ(\omega):=Q_{\t}^{\ast}(Q_{\iin}\omega)=1$, which is a  way to rephrase that automaton $(Q)$ describes the regular language $L_\I$: 
\[
L_\I \ = \ \alpha^{-1}_{\I,(Q)} (1) \ \subset  \ \Sigma^{\ast}. 
\]

We may also write $\alpha_\I$ in place of $\alphaiQ$, for short, if the automaton $(Q)$ is fixed. 

\vspace{0.07in} 

The action of $\Sigma^{\ast}$ on $\Bool Q$  can be described via Boolean $Q\times Q$ matrices, that is, matrices with coefficients in $\Bool$ with rows and columns enumerated by states $q\in Q$ of $(Q)$. A matrix $\mathsf{M}\in \mathsf{Mat}_Q(\Bool)$ acts by right multiplication on $\Bool$-valued row vectors, which constitute a free $\Bool$-module isomorphic to $\Bool Q$.  To $a\in \Sigma$ associate the matrix $\mathsf{M}_a$ with the coefficient $\mathsf{M}_{a,q,q'}=1$ if and only if $q'\in \delta(q,a)$.

The vector $Q_{\init}$ corresponds to a Boolean row matrix with $1$ in positions $q\in Q_{\init}$ and the covector $Q_{\t}^{\ast}$ to a column matrix with $1$ in positions $q\in Q_{\t}$. 

\vspace{0.1in}

{\it A one-dimensional TQFT $\mcFQ$.}
We build a one-dimensional TQFT $\mcFQ$ with defects and inner endpoints associated to $(Q)$ by assigning $\Bool Q$ to a positively-oriented point $+$  and the dual module $\Bool Q^{\ast}$ to a negatively-oriented point $-$. We call these $\Bool$-modules \emph{the state spaces} of $+$ and $-$, respectively, and write 
\begin{equation}\label{eq_funcFQ} 
\mcFQ(+) \ := \ \Bool Q, \hspace{0.75cm}  \mcFQ(-) \ := \ \Bool Q^{\ast}.
\end{equation}
\begin{remark}
Our sign convention is opposite to that of~\cite{IK-top-automata} but matches the notations in Section~\ref{subsection:projectives}.  We use the right action of $\Sigma^{\ast}$ on $\Bool Q$ above for a better match with the automata theory literature, although switching to the left action would be a better match with the literature from the mathematics side.  
\end{remark}

We may write $\mcF$ instead of $\mcFQ$ if an automaton $(Q)$ is fixed. Our TQFT will be a symmetric monoidal functor 
\begin{equation}
\mcFQ\ : \ \Cob_{\Sigma,\I}\lra \Bool\dmod
\end{equation}
from the category of oriented one-dimensional cobordisms with defects and inner endpoints to the category of $\Bool$-modules.  In fact, objects in the image of $\mcFQ$ will be finite free $\Bool$-modules. 

To a sign sequence $\underline{\varepsilon}=(\varepsilon_1,\ldots, \varepsilon_k)$, $\varepsilon_i\in \{+,-\}$, we assign the state space
\begin{equation}\label{eq_F_eps}
 \mcFQ(\varepsilon) \ := \ \mcFQ(\varepsilon_1)\otimes \cdots\otimes  \mcFQ(\varepsilon_k), 
\end{equation}
which is a tensor product of free $\Bool$-modules $\Bool Q$ and $\Bool Q^{\ast}$. In particular, $\mcFQ(\varepsilon)$ is a  free $\Bool$-module of rank $|Q|^k$. 

\begin{figure}
    \centering
\begin{tikzpicture}[scale=0.6]
\begin{scope}[shift={(-0.5,0)}]
%\draw[thin,yellow] (0,0) grid (4,4);
\draw[thick,dashed] (0,3) -- (1,3);
\draw[thick,dashed] (0,0) -- (1,0);
\node at (0.5,3.35) {$+$};
\draw[thick,->] (0.5,0) -- (0.5,3);
\draw[thick,fill] (0.65,1.45) arc (0:360:1.5mm);
\node at (1,1.45) {$a$};
\node at (0.5,-0.35) {$+$};

\begin{scope}[shift={(0.55,0)}]
\node at (2,3) {$\Bool Q$};
\draw[thick,->] (1.95,0.45) -- (1.95,2.65);
\node at (2.60,1.45) {$m_a$};
\node at (2,0) {$\Bool Q$};

\node at (3.55,3) {$qa$};
\node at (2.80,3) {$\ni$};
\draw[thick,|->] (3.55,0.45) -- (3.55,2.65);
\node at (2.80,0) {$\ni$};
\node at (3.55,0) {$q$};
\end{scope}
\end{scope}

\begin{scope}[shift={(7,0)}]
%\draw[thin,yellow] (0,0) grid (4,4);
\draw[thick,dashed] (0,3) -- (1,3);
\draw[thick,dashed] (0,0) -- (1,0);
\node at (0.5,3.35) {$+$};
\draw[thick,->] (0.5,0) -- (0.5,3);
\draw[thick,fill] (0.65,1.45) arc (0:360:1.5mm);
\node at (1,1.45) {$\omega$};
\node at (0.5,-0.35) {$+$};

\node at (1.75,1.45) {$=$};

\begin{scope}[shift={(2,0)}]
%\draw[thin,yellow] (0,0) grid (4,4);
\draw[thick,dashed] (0,3) -- (1,3);
\draw[thick,dashed] (0,0) -- (1,0);
\node at (0.5,3.35) {$+$};
\draw[thick,->] (0.5,0) -- (0.5,3);
\draw[thick,fill] (0.65,2.15) arc (0:360:1.5mm);
\node at (1.15,2.15) {$a_n$};
\node at (1,1.5) {$\vdots$};

\draw[thick,fill] (0.65,0.85) arc (0:360:1.5mm);
\node at (1.15,0.80) {$a_1$};
\node at (0.5,-0.35) {$+$};
\end{scope}

\begin{scope}[shift={(3.05,0)}]

\node at (1.75,3) {$\Bool Q$};
\draw[thick,->] (1.7,0.45) -- (1.7,2.65);
\node at (2.35,1.45) {$m_{\omega}$};
\node at (1.75,0) {$\Bool Q$};

\node at (5.15,3) {$q\omega=q a_1 a_2\cdots a_n$};
\node at (2.63,3) {$\ni$};
\draw[thick,|->] (5,0.45) -- (5,2.65);
\node at (3.5,0) {$\ni$};
\node at (5,0) {$q$};
\end{scope}
\end{scope}

\begin{scope}[shift={(19,0)}]
%\draw[thin,yellow] (0,0) grid (4,4);
\draw[thick,dashed] (0,3) -- (1,3);
\draw[thick,dashed] (0,0) -- (1,0);
\node at (0.5,3.35) {$-$};
\draw[thick,<-] (0.5,0) -- (0.5,3);
\draw[thick,fill] (0.65,1.45) arc (0:360:1.5mm);
\node at (1,1.45) {$a$};
\node at (0.5,-0.35) {$-$};

\begin{scope}[shift={(0.85,0)}]
\node at (1.75,3) {$\Bool Q^*$};
\draw[thick,->] (1.7,0.45) -- (1.7,2.65);
\node at (2.35,1.45) {$m_a^*$};
\node at (1.75,0) {$\Bool Q^*$};

\node at (3.53,3) {$aq^*$};
\node at (2.63,3) {$\ni$};
\draw[thick,|->] (3.4,0.45) -- (3.4,2.65);
\node at (2.63,0) {$\ni$};
\node at (3.45,0) {$q^*$};
\end{scope}
\end{scope}

\end{tikzpicture}
    \caption{Action of $\Sigma^{\ast}$ on $\Bool Q$ and $\Bool Q^{\ast}$, with word $\omega=a_1\cdots a_n$. }
    \label{figure-1.2}
\end{figure}

To a point labelled $a\in\Sigma$ on an upward-oriented vertical line, the functor $\mcFQ$ assigns the operator $m_a: \Bool Q\lra \Bool Q$ of multiplication by $a$ (right action in \eqref{eq_sum}), see Figure~\ref{figure-1.2} on the left. 
A sequence of points labelled $a_1,\ldots, a_n$ on a upward line describes a word $\omega=a_1\cdots a_n$ which, upon applying $\mcFQ$, acts by the composition $m_{\omega}$ of these operators, $m_{\omega}(q)=q\omega=q a_1\cdots a_n$, see Figure~\ref{figure-1.2} in the middle. 

To a point labelled $a$ on a downward-oriented vertical line we assign the dual operator $m_a^{\ast}: \Bool Q^{\ast}\lra \Bool Q^{\ast}$. Writing $m_a$ via the Boolean square matrix $\mathsf{M}_a$ in the unique basis $Q$ of $\Bool Q$, the dual operator $m_a^{\ast}$ is given by the transposed matrix $\mathsf{M}_a^{T}$ in the dual basis $Q^{\ast}$ of $\Bool Q^{\ast}$.

\begin{figure}
    \centering
\begin{tikzpicture}[scale=0.6]
\begin{scope}[shift={(10.5,0)}]
%\draw[thin,yellow] (0,0) grid (4,4);
\draw[thick,dashed] (0,3) -- (3,3);
\draw[thick,dashed] (0,0) -- (3,0);
\node at (0.5,3.35) {$+$};
\node at (2.5,3.35) {$-$};
\draw[thick,<-] (0.5,3) .. controls (0.75,1) and (2.25,1) .. (2.5,3);

\node at (4.90,3) {$\Bool Q \otimes \Bool Q^*$};
\draw[thick,->] (4.75,0.45) -- (4.75,2.65);
\node at (4.75,0) {$\Bool$};
\node at (5.65,1.45) {$\coev$};
\end{scope}

\begin{scope}[shift={(21,0)}]
%\draw[thin,yellow] (0,0) grid (4,4);
\draw[thick,dashed] (0,3) -- (3,3);
\draw[thick,dashed] (0,0) -- (3,0);
\node at (0.5,-0.35) {$-$};
\node at (2.5,-0.35) {$+$};
\draw[thick,<-] (0.5,0) .. controls (0.75,2) and (2.25,2) .. (2.5,0);

\node at (4.95,0) {$\Bool Q^* \otimes \Bool Q$};
\draw[thick,->] (5.1,0.45) -- (5.1,2.65);
\node at (5.1,3) {$\Bool$};
\node at (5.70,1.45) {$\ev$};
\end{scope}    
\end{tikzpicture}
    \caption{Coevaluation and evaluation morphisms associated to \emph{cup} and \emph{cap} cobordisms.}
    \label{figure-1.1a}
\end{figure}
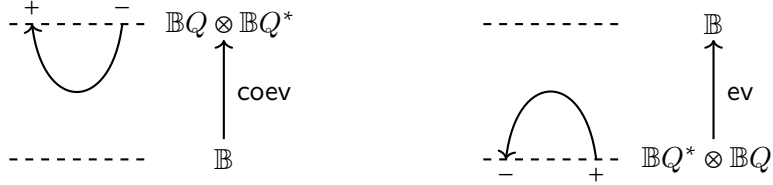

To a \emph{cup}, respectively a \emph{cap}, cobordism, see Figure~\ref{figure-1.1a}, the functor $\mcFQ$ associates the coevaluation map, respectively the evaluation map 
\begin{eqnarray}
\mathsf{coev} & : &  \Bool \lra \Bool Q \otimes \Bool Q^{\ast}, \hspace{0.75cm}  1 \longmapsto \sum_{q\in Q} q\otimes q^{\ast}, \\
\mathsf{ev} & : & \Bool Q^{\ast} \otimes \Bool Q \lra \Bool, \hspace{0.75cm}   q_1^{\ast} \otimes q_2 \longmapsto \delta_{q_1,q_2}, 
\end{eqnarray} 

These maps are compatible with maps $m_a$ and $m_a^{\ast}$ induced by a dot labelled $a$ (see Figure~\ref{figure-1.3} bottom row) and they satisfy the isotopy relations in Figure~\ref{figure-1.3} in top right, middle and bottom rows.

\begin{figure}
    \centering
\begin{tikzpicture}[scale=0.6]

\begin{scope}[shift={(0,0)}]
%\draw[thin,yellow] (0,0) grid (4,4);
\draw[thick,dashed] (0,3) -- (1,3);
\draw[thick,dashed] (0,0) -- (1,0);
\node at (0.5,3.35) {$+$};
\draw[thick,->] (0.5,1.5) -- (0.5,3);
\end{scope}

\begin{scope}[shift={(0.8,0)}]
\node at (1.75,3) {$\Bool Q$};
\draw[thick,->] (1.7,0.45) -- (1.7,2.65);
\node at (1.7,0) {$\Bool$};

\node at (3.30,3) {$Q_{\init}$};
\node at (2.55,3) {$\ni$};
\draw[thick,|->] (3.30,0.45) -- (3.30,2.65);
\node at (2.55,0) {$\ni$};
\node at (3.30,0) {$1$};
\end{scope}

\begin{scope}[shift={(7,0)}]
%\draw[thin,yellow] (0,0) grid (4,4);
\draw[thick,dashed] (0,3) -- (1,3);
\draw[thick,dashed] (0,0) -- (1,0);
\node at (0.5,-0.35) {$+$};
\draw[thick] (0.5,0) -- (0.5,0.75);
\draw[thick,->] (0.5,0.75) -- (0.5,1.5);
\end{scope}

\begin{scope}[shift={(7.8,0)}]
\node at (1.7,3) {$\Bool$};
\draw[thick,->] (1.7,0.45) -- (1.7,2.65);
\node at (1.75,0) {$\Bool Q$};
\node at (2.30,1.45) {$Q_{\t}^{\ast}$};

%\node at (3.25,3) {$Q_{\init}$};
%\node at (2.5,3) {$\ni$};
%\draw[thick,|->] (3.25,0.45) -- (3.25,2.65);
%\node at (2.5,0) {$\ni$};
%\node at (3.25,0) {$1$};
\end{scope}

\begin{scope}[shift={(12.5,0)}]
%\draw[thin,yellow] (0,0) grid (3,3);
\draw[thick,dashed] (0,3) -- (1,3);
\draw[thick,dashed] (0,0) -- (1,0);
\node at (0.5,3.35) {$-$};
\draw[thick,<-] (0.5,1.5) -- (0.5,3);
\node at (1.5,1.5) {$:=$};

\begin{scope}[shift={(2,0)}]
%\draw[thin,yellow] (0,0) grid (3,3);
\draw[thick,dashed] (0,3) -- (2,3);
\draw[thick,dashed] (0,0) -- (2,0);
\node at (1.5,3.35) {$-$};
\draw[thick,<-] (1.5,1.5) -- (1.5,3);

\draw[thick] (0.5,1.5) .. controls (0.55,0.5) and (1.45,0.5) .. (1.5,1.5);

\end{scope}
\end{scope}

\begin{scope}[shift={(20,0)}]
%\draw[thin,yellow] (0,0) grid (3,3);
\draw[thick,dashed] (0,3) -- (1,3);
\draw[thick,dashed] (0,0) -- (1,0);
\node at (0.5,-0.35) {$-$};
\draw[thick,->] (0.5,1.5) -- (0.5,0);
\node at (1.5,1.5) {$:=$};

\begin{scope}[shift={(2,0)}]
%\draw[thin,yellow] (0,0) grid (3,3);
\draw[thick,dashed] (0,3) -- (2,3);
\draw[thick,dashed] (0,0) -- (2,0);
\node at (1.5,-0.35) {$-$};
\draw[thick] (1.5,1.5) -- (1.5,0);

\draw[thick,->] (0.5,1.5) .. controls (0.55,2.5) and (1.45,2.5) .. (1.5,1.5);

\end{scope}
\end{scope}

\end{tikzpicture}
    \caption{Left: maps assigned to the half-intervals with a $+$ boundary points. Right: defining maps for half-intervals with a $-$ boundary point.}
    \label{figure-1.1}
\end{figure}
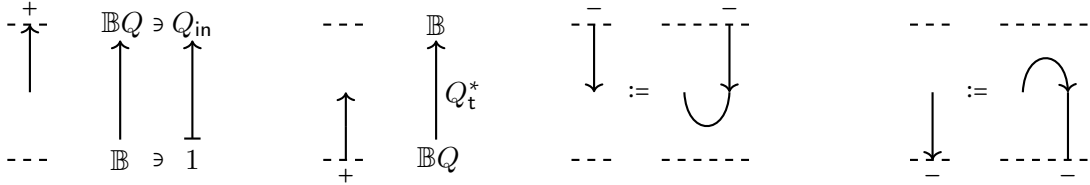

\vspace{0.07in}

A \emph{half-interval} is a connected component of a one-cobordism which has one (outer) boundary and one inner endpoint, see Section~\ref{subsec_floating}. 

To a half-interval ending in $+$ at the top boundary we assign $Q_{\init}\in \Bool Q$, thinking of it as describing a map $\Bool\lra \Bool Q$ which takes $1$ to $Q_{\init}$. To a half-interval ending in $+$ at the bottom we assign $\Bool$-linear map $Q_{\t}^{\ast}$ in \eqref{eq_trace}. Figure~\ref{figure-1.1} on the left explains these assignments. The other two half-intervals (those with a $-$ boundary endpoint) are given by composing the intervals with a $+$ endpoint with a cup or a cap, respectively, see Figure~\ref{figure-1.1} on the right. The map for the half-interval terminating with $-$ at the top, respectively at the bottom, is the dual $Q_{\t}:\Bool \lra \Bool Q^{\ast}$ of the trace map, $Q_{\t}(1)=\sum_{q\in Q_{\t}}q$, respectively, the dual $Q_{\init}^{\ast}:\Bool Q^{\ast}\lra \Bool$ of the unit map. 

The functor $\mcFQ$ takes the transposition cobordism given by a crossing with various orientations to the transposition isomorphism $V\otimes W \lra W\otimes V$ of the tensor products of $\Bool$-modules $V,W\in \{\Bool Q,\Bool Q^{\ast}\}$.  The following proposition is straightforward to check.

\begin{prop} \label{prop_bijection}  $ $
\begin{enumerate}
\item A nondeterministic automaton $(Q)$ gives rise to a symmetric monoidal functor 
\[\mcFQ \ : \ \CobSI \lra \Bool\fmod 
\]
from $\CobSI$ to the category of free $\Bool$-modules.  
\item Isomorphism classes of such functors are in a bijection with isomorphism classes of nondeterministic automata.  
\end{enumerate}
\end{prop}

We call automata $(Q_1)$ and $(Q_2)$ over the same set of letters $\Sigma$ \emph{isomorphic} if there is a bijection between their states that converts the transition function, initial and accepting states for one automaton into the transition function, initial and accepting states for the other automaton. 

In \eqref{eq_funcFQ} one can replace $\Bool\fmod$ by the smaller full subcategory of finite free $\Bool$-modules. 
Functor $\mcFQ$ intertwines monoidal structures on the two categories due to \eqref{eq_F_eps} and intertwines rigid and symmetric structures of the two categories as well. 
$\square$

\vspace{0.1in}

{\it Evaluation of intervals.}
One can now arbitrary compose these generating cobordisms. By a closed cobordism we mean a cobordism from the empty sequence $\emptyset_0$ to itself. Such a cobordism is a disjoint union of oriented intervals and circles with defects. 

An interval with defects is determined by the word $\omega$ read along it in the orientation direction and evaluates to $\alphai(Q_{\init}\omega)\in \Bool$, see Figure~\ref{figure-B1}. The evaluation is $1$ if $\omega\in L_\I$ and $0$ otherwise. Isotopy relations in Figure~\ref{figure-1.3} ensure that the evaluation of the interval does not depend on its presentation as a monoidal concatenation of basis morphisms. 

\vspace{0.1in} 

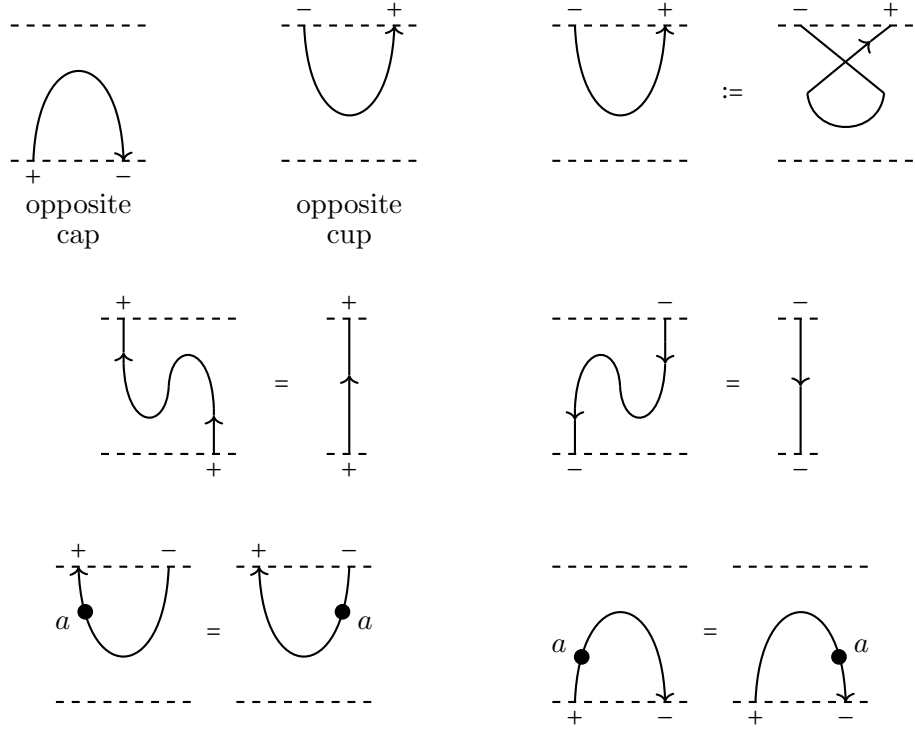
\begin{figure}
    \centering
    \begin{tikzpicture}[scale=0.6]
    \begin{scope}[shift={(0,0)}]
    %\draw[thin,yellow] (0,0) grid (4,4);
    \draw[thick,dashed] (0,3) -- (3,3);
    \draw[thick,dashed] (0,0) -- (3,0);
    
    \node at (0.5,-0.35) {$+$};
    \node at (2.5,-0.35) {$-$};
    
    \draw[thick,->] (0.5,0) .. controls (0.6,2.65) and (2.4,2.65) .. (2.5,0);
    \node at (1.5,-1) {opposite};
    \node at (1.5,-1.75) {cap};
   % \node at (1.5,-1.25) {cap};
    \end{scope}

    \begin{scope}[shift={(6,0)}]
    %\draw[thin,yellow] (0,0) grid (4,4);
    \draw[thick,dashed] (0,3) -- (3,3);
    \draw[thick,dashed] (0,0) -- (3,0);
    
    \node at (0.5,3.35) {$-$};
    \node at (2.5,3.35) {$+$};
    \draw[thick,->] (0.5,3) .. controls (0.6,0.35) and (2.4,0.35) .. (2.5,3);
    \node at (1.5,-1) {opposite};
    \node at (1.5,-1.75) {cup};
  %  \node at (1.5,-1.25) {cup};
    \end{scope}

    \begin{scope}[shift={(12,0)}]
    %\draw[thin,yellow] (0,0) grid (4,4);
    \draw[thick,dashed] (0,3) -- (3,3);
    \draw[thick,dashed] (0,0) -- (3,0);
    
    \node at (0.5,3.35) {$-$};
    \node at (2.5,3.35) {$+$};
    \draw[thick,->] (0.5,3) .. controls (0.6,0.35) and (2.4,0.35) .. (2.5,3);
  %  \node at (1.5,-1) {opposite};
  %  \node at (1.5,-1.75) {cup};
    
    \node at (4,1.5) {$:=$};
    \end{scope}

    \begin{scope}[shift={(17,0)}]
    %\draw[thin,yellow] (0,0) grid (4,4);
    \draw[thick,dashed] (0,3) -- (3,3);
    \draw[thick,dashed] (0,0) -- (3,0);
    
    \node at (0.5,3.35) {$-$};
    \node at (2.5,3.35) {$+$};
    \draw[thick] (0.65,1.5) .. controls (0.75,0.5) and (2.25,0.5) .. (2.35,1.5);
    \draw[thick,->] (0.65,1.5) -- (2.05,2.65);
    \draw[thick] (2.05,2.65) -- (2.5,3);
    \draw[thick] (2.35,1.5) -- (0.5,3);
    \end{scope} 
    
    \begin{scope}[shift={(2,-6.5)}]
    %\draw[thin,yellow] (0,0) grid (4,4);
    \draw[thick,dashed] (0,3) -- (3,3);
    \draw[thick,dashed] (0,0) -- (3,0);
    
    \node at (0.5, 3.35) {$+$};
    \node at (2.5,-0.35) {$+$};
    
    \draw[thick] (0.5,3) -- (0.5,2.25);
    \draw[thick,<-] (0.5,2.25) -- (0.5,2.0);    
    \draw[thick,<-] (2.5,0.85) -- (2.5,0);
    \draw[thick] (2.5,1.0) -- (2.5,0.85);
    
    \draw[thick] (0.5,2) .. controls (0.55,0.5) and (1.45,0.5) .. (1.5,1.5);
    \draw[thick] (1.5,1.5) .. controls (1.55,2.5) and (2.45,2.5) .. (2.5,1);    
    \node at (4,1.5) {$=$};
    \end{scope}
    
    \begin{scope}[shift={(7,-6.5)}]
    %\draw[thin,yellow] (0,0) grid (4,4);
    \draw[thick,dashed] (0,3) -- (1,3);
    \draw[thick,dashed] (0,0) -- (1,0);
    \node at (0.5, 3.35) {$+$};
    \node at (0.5,-0.35) {$+$};
    \draw[thick,<-] (0.5,1.75) -- (0.5,0);
    \draw[thick] (0.5,3) -- (0.5,1.75);
    \end{scope}
    
    \begin{scope}[shift={(12,-6.5)}]
%    \draw[thin,yellow] (0,0) grid (4,4);
    \draw[thick,dashed] (0,3) -- (3,3);
    \draw[thick,dashed] (0,0) -- (3,0);
    
    \node at (2.5, 3.35) {$-$};
    \node at (0.5,-0.35) {$-$};
    
    \draw[thick] (2.5,3) -- (2.5,2.5);
    \draw[thick,->] (2.5,2.5) -- (2.5,2.0);
    \draw[thick] (0.5,0.75) -- (0.5,0);
    \draw[thick,->] (0.5,1.0) -- (0.5,0.75); 

    \draw[thick] (0.5,1) .. controls (0.55,2.5) and (1.45,2.5) .. (1.5,1.5);
    \draw[thick] (1.5,1.5) .. controls (1.55,0.5) and (2.45,0.5) .. (2.5,2);   
    \node at (4,1.5) {$=$};
    \end{scope}
    
    \begin{scope}[shift={(17,-6.5)}]
    %\draw[thin,yellow] (0,0) grid (4,4);
    \draw[thick,dashed] (0,3) -- (1,3);
    \draw[thick,dashed] (0,0) -- (1,0);
    \node at (0.5, 3.35) {$-$};
    \node at (0.5,-0.35) {$-$};
    \draw[thick] (0.5,1.5) -- (0.5,0);
    \draw[thick,->] (0.5,3) -- (0.5,1.5);
    \end{scope} 
    
\begin{scope}[shift={(1,-12)}]
%\draw[thin,yellow] (0,0) grid (3,3);
\draw[thick,dashed] (0,3) -- (3,3);
\draw[thick,dashed] (0,0) -- (3,0);
\node at (0.5,3.35) {$+$};
\node at (2.5,3.35) {$-$};
\draw[thick,<-] (0.5,3) .. controls (0.6,0.35) and (2.4,0.35) .. (2.5,3);

\draw[thick,fill] (0.80,2) arc (0:360:1.5mm);
\node at (0.15,1.75) {$a$};
\node at (3.5,1.5) {$=$};
\end{scope}

\begin{scope}[shift={(5,-12)}]
%\draw[thin,yellow] (0,0) grid (3,3);
\draw[thick,dashed] (0,3) -- (3,3);
\draw[thick,dashed] (0,0) -- (3,0);
\node at (0.5,3.35) {$+$};
\node at (2.5,3.35) {$-$};
\draw[thick,<-] (0.5,3) .. controls (0.6,0.35) and (2.4,0.35) .. (2.5,3);

\draw[thick,fill] (2.50,2) arc (0:360:1.5mm);
\node at (2.85,1.75) {$a$};
\end{scope}

\begin{scope}[shift={(12,-12)}]
%\draw[thin,yellow] (0,0) grid (3,3);
\draw[thick,dashed] (0,3) -- (3,3);
\draw[thick,dashed] (0,0) -- (3,0);
\node at (0.5,-0.35) {$+$};
\node at (2.5,-0.35) {$-$};
\draw[thick,->] (0.5,0) .. controls (0.6,2.65) and (2.4,2.65) .. (2.5,0);

\draw[thick,fill] (0.80,1) arc (0:360:1.5mm);
\node at (0.15,1.25) {$a$};
\node at (3.5,1.5) {$=$};
\end{scope}

\begin{scope}[shift={(16,-12)}]
%\draw[thin,yellow] (0,0) grid (3,3);
\draw[thick,dashed] (0,3) -- (3,3);
\draw[thick,dashed] (0,0) -- (3,0);
\node at (0.5,-0.35) {$+$};
\node at (2.5,-0.35) {$-$};
\draw[thick,->] (0.5,0) .. controls (0.6,2.65) and (2.4,2.65) .. (2.5,0);

\draw[thick,fill] (2.50,1) arc (0:360:1.5mm);
\node at (2.85,1.25) {$a$};
\end{scope}

    \end{tikzpicture}
    \caption{Top left: the cap and cup cobordisms with the opposite orientations. The cap and cup cobordisms for the opposite orientation are obtained by composing the original cup and cap in Figure~\ref{figure-1.1a} with the transposition cobordisms, see top right for the cup cobordism for the opposite orientation. Middle row shows isotopy relations on cup and cap cobordisms and the bottom row--compatibility with the dot maps (isotopies of dots across local maxima and minima). }
    \label{figure-1.3}
\end{figure}

\begin{figure}
    \centering
\begin{tikzpicture}[scale=0.6]
\begin{scope}[shift={(0,0)}]
%\draw[thin,yellow] (0,0) grid (4,4);
\node at (2,3) {$\omega=a_1\cdots a_n$};
\draw[thick,->] (0,2) -- (2.5,2);
\draw[thick] (2.5,2) -- (4,2);

\draw[thick,fill] (0.65,2) arc (0:360:1.5mm);
\node at (0.5,1.25) {$a_1$};
\draw[thick,fill] (1.65,2) arc (0:360:1.5mm);
\node at (1.5,1.25) {$a_2$};
\node at (2.45,1.25) {$\cdots$};
\draw[thick,fill] (3.65,2) arc (0:360:1.5mm);
\node at (3.50,1.25) {$a_n$};
\node at (4.75,2) {$=$};
\node at (6.65,2) {$\alpha_{\I}(Q_{\init}\omega)$};
\node at (8.55,2) {$=$};
\node at (10.5,2) {$\alpha_{\I,(Q)}(\omega)$};
\end{scope}

\begin{scope}[shift={(14.5,0)}]
%\draw[thin,yellow] (0,0) grid (4,4);
\draw[thick,->] (0,2) -- (1,2);
\draw[thick] (1,2) -- (2,2);
\node at (2.75,2) {$=$};
\node at (4.5,2) {$\alpha_{\I}(Q_{\init})$};
\end{scope}

\end{tikzpicture}
    \caption{Left: a decorated interval and its evaluation. Right: an interval with no defects.}
    \label{figure-B1}
\end{figure}

\vspace{0.1in} 

An interval without defects evaluates to $\alphai(Q_{\init})=\alphaiQ(\emptyset)\in \Bool$, where $\emptyset\in \Sigma^{\ast}$ is the empty word. We will write 
\begin{equation}\alphaiQ(\omega):=
\alphai(Q_{\init}\omega),   \hspace{0.75cm}
\alphaiQ:\Sigma^{\ast}\lra \Bool
\end{equation}
for the interval evaluation of words $\omega\in \Sigma^{\ast}$. Evaluation $\alphaiQ(\omega)=1$ if and only if $\omega$ is in the language $L_\I$ accepted by the automaton $(Q)$. 

Our notations contain several versions of the empty set: 
\begin{itemize}
\item $\emptyset\in \Sigma^{\ast}$ is the empty word and the identity of the monoid $\Sigma^{\ast}$. 
\item $\emptyset_0$ is the empty $0$-manifold and the unit (identity) object $\one$ of various monoidal categories of  1-cobordisms. 
\item $\emptyset_1$ is the empty $1$-manifold, which is the identity morphism of the identity object $\one$ of the category of 1-cobordisms. 
\end{itemize}

{\it Circular and strongly circular languages.}
Denote by $\alphacQ(\omega)$ the evaluation of an oriented circle with the circular word $\omega$ written on it.

We view a circular word as an equivalence class of words in $\Sigma^{\ast}$ modulo the equivalence relation $\omega_1\omega_2\sim\omega_2\omega_1$ for $\omega_1,\omega_2\in\Sigma^{\ast}$ and denote the equivalence classes by $\Sigma^{\circ}:=\Sigma^{\ast}/\sim$. 
Evaluation  $\alphacQ(\omega)$ does not depend on the presentation of a $\omega$-decorated circle 
as a concatenation of basis morphisms, for letters in $\omega$. The corresponding evaluation map 
\[
\alphacQ \ : \ \Sigma^{\circ}\lra \Bool 
\]
goes from the set of circular words to $\Bool$. 

Put a circle with a defect circular word $\omega=a_1\cdots a_n$ in a standard position as shown in Figure~\ref{figure-B2}.

\vspace{0.1in}

\begin{figure}
    \centering
\begin{tikzpicture}[scale=0.6]
\begin{scope}[shift={(0,0)}]
%\draw[thin,yellow] (0,0) grid (4,4);
\draw[thick,dashed] (0,4) -- (3,4);
\draw[thick,dashed] (0,0) -- (3,0);
\draw[thick,<-] (2.5,2) arc (0:360:1);

\draw[thick,fill] (1.65,3) arc (0:360:1.5mm);
\node at (1.5,3.45) {$a_n$};
\draw[thick,fill] (0.90,1.35) arc (0:360:1.5mm);
\node at (0.25,2.5) {\rotatebox[origin=c]{-25}{$\vdots$}};
\node at (0.25,1) {$a_2$};
\draw[thick,fill] (1.65,1) arc (0:360:1.5mm);
\node at (1.45,0.50) {$a_1$};

\node at (3.75,2) {$=$};
\end{scope}

\begin{scope}[shift={(4.95,0)}]
%\draw[thin,yellow] (0,0) grid (4,4);
\draw[thick,dashed] (0,4) -- (3,4);
\draw[thick,dashed] (0,3) -- (3,3);
\draw[thick,dashed] (0,1) -- (3,1);
\draw[thick,dashed] (0,0) -- (3,0);

\draw[thick] (0.5,3) .. controls (0.6,4) and (2.4,4) .. (2.5,3);
\draw[thick] (2.5,1) -- (2.5,2);
\draw[thick,<-] (2.5,2) -- (2.5,3);
\draw[thick] (0.5,1) -- (0.5,3);
\draw[thick] (0.5,1) .. controls (0.6,0) and (2.4,0) .. (2.5,1);

\draw[thick,fill] (0.65,2.70) arc (0:360:1.5mm);
\node at (-0.25,2.75) {$a_n$};
\node at (-0.25,2.25) {$\vdots$};
\draw[thick,fill] (0.65,1.75) arc (0:360:1.5mm);
\node at (-0.25,1.75) {$a_2$};
\draw[thick,fill] (0.65,1.25) arc (0:360:1.5mm);
\node at (-0.25,1.25) {$a_1$};

\node at (4,3.5) {$\ev$};
\node at (4,0.5) {$\coev$};

\node at (5.10,2) {$=$};
\node at (7,1.75) {$\displaystyle{\sum_{q\in Q}} q^*(q\omega)$};
\end{scope}

\end{tikzpicture}
    \caption{Evaluation of an $\omega$-decorated circle, $\omega=a_1\cdots a_n$.}
    \label{figure-B2}
\end{figure}
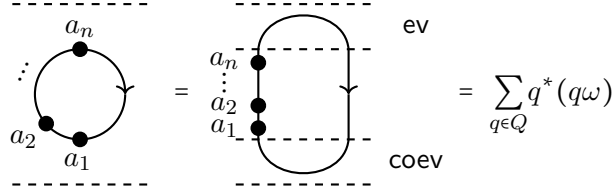

The evaluation of $\omega$ is then 
\begin{equation}\label{eq_circ_def}
\alphacQ(\omega) \ = \ \ev \circ (m_\omega\otimes \id_+)\circ \coev \ = \ \sum_{q\in Q} q^{\ast}(q\omega).  
\end{equation} 
A word $\omega$ evaluates to $1$ via $\alpha_{\circ,(Q)}$ if and only if for some state $q\in Q$ there is a path $\omega$ in the decorated graph $(Q)$ that starts and ends at $q$. Evaluation $\alphacQ$ defines a circular language $\LcQ\subset \Sigma^{\circ}$.

We say that a language $L'\subset \Sigma^{\ast}$ is 
\begin{itemize}
\item \emph{circular} if $\omega_1\omega_2\in L'$ if and only if $\omega_2\omega_1\in L'$ for any $\omega_1\omega_2\in \Sigma^{\ast}$, 
\item \emph{strongly circular} if  
it is circular  and $\omega\in L'$ implies that $\omega^n\in L'$ for all $n\ge 0$, 
\item \emph{cyclic} if it is circular and $\omega\in L'\Leftrightarrow \omega^n\in L'$ for any $n\ge 2,$ see~\cite{BR1,Cart97},
\item a \emph{trace} or  \emph{loop} language if it consists of words that are loops in some automaton $(Q)$. 
\end{itemize}
Setting $n=0$ above, we see that a strongly circular language is either the empty language $\emptyset_L$ or it contains the empty word, since $\emptyset=\omega^0$ for any $\omega\in \Sigma^{\ast}$. 
The language $\{a^n \}$ is an example of a circular but not a strongly circular language. 
The notion of a \emph{cyclic} language is similar but different from that of a strongly circular language. A trace language is strongly circular, see below. 

Likewise, an evaluation $\alpha:\Sigma^{\ast}\lra \Bool$ is called 
\begin{itemize}
\item \emph{circular} if $\alpha(\omega_1\omega_2)=\alpha(\omega_2\omega_1), $ for all $\omega_1\omega_2\in \Sigma^{\ast}$, 
\item \emph{strongly circular} if it is circular and 
 $\alpha(\omega^n)=\alpha(\omega)$ for any word $\omega\in \Sigma^{\ast}$ and $n\ge 0$. 
\end{itemize}

\begin{prop} For any automaton $(Q)$, the trace language $\LcQ$ is a strongly circular regular language. It depends only on the transition function in $(Q)$ and not on the sets of initial and accepting states $Q_{\init},Q_{\t}$. 
\end{prop} 

\begin{proof} It is immediate that circular evaluation  $\alpha_{\circ,Q}$ and the associated circular language $L_{\circ,Q}$ 
is described by a finite system via \eqref{eq_circ_def} and is a regular circular language. In more details, the language $L_{\circ,(Q)}$ picks out circular words for which there is a cycle in $(Q)$. For each state $q\in Q$ we can form the automaton $(Q)_q$ with the states and transition function given by $(Q)$ and $q$ being the only initial and accepting state. Then $L_{\circ,(Q)}$ is the language of the automaton $\sqcup_{q\in Q}(Q)_q$, the disjoint union of automata $(Q)_q$ over all states $q$ in $Q$. 

The empty word is in $\LcQ$  since 
\[  \alphacQ(\emptyset) \  = \ \sum_{q\in Q} q^{\ast}(q) \ = \ \sum_{q\in Q} 1 \ = \ 1\in \Bool. 
\] 
A word $\omega\in \LcQ$ if and only if  $q^{\ast}(q\omega)=1$ for some state $q$, which means that there is a path $\omega$ from $q$ to itself (in general, there may be several paths $\omega$ starting at $q$; the circular evaluation of $\omega$ is $1$ if and only if there is a path that comes back to $q$). 

The $n$-th power of this path will go from $q$ to $q$ as well, so that $\omega^n\in \LcQ$ for any $n\ge 0$, and the language $\LcQ$ is strongly circular. 
\end{proof} 

We see that an automaton $(Q)$ defines a pair of evaluations 
\begin{equation}\alpha_{(Q)}:=(\alphaiQ,\alphacQ),
\end{equation}
the second of which is strongly circular. We may call these \emph{the interval} and \emph{the circular} or \emph{the trace} evaluations of $(Q)$, respectively. These evaluations give rise to a pair of regular languages 
\begin{equation}
L_{(Q)}:=(L_{\I,(Q)},L_{\circ,(Q)}),
\end{equation}
with $L_{\circ,(Q)}$ strongly circular. We can call these languages \emph{the interval} and \emph{the trace} languages of $(Q)$, respectively.
The language $L_{\circ,(Q)}$ may also be called  \emph{the loop language} or \emph{the circular language} of the automaton $(Q)$. 

\vspace{0.07in} 

Note that for the empty automaton $(\emptyset)$ with no states both languages $L_{\I,(\emptyset)}$ and $L_{\circ,(\emptyset)}$ are empty (contain no words), justifying our inclusion of the empty language into the set of strongly circular languages. For any nonempty automaton $(Q)$ its trace language $L_{\circ,(Q)}$ contains the empty word $\emptyset_0$, while its (interval) language $L_{\I,(Q)}$ contains 
the empty word if and only if $Q_{\init}\cap Q_{\t}$ is nonempty.

\vspace{0.1in}

{\it Decomposition of the identity.}
Any TQFT allows for a so-called decomposition of the identity. In the TQFT for the automaton $(Q)$, one can introduce endpoints labelled by $q$ and $q^{\ast}$, over all states $q\in Q$, depending on 
the orientation of the interval near the endpoint, see Figure~\ref{figure-B3}. One can then decompose an arc as the sum over pairs of half-intervals labelled $q$ and $q^{\ast}$, over all $q\in Q$, see Figure~\ref{figure-B3} on the right.  Another skein relation is shown in that figure as well.

\begin{figure}
    \centering
\begin{tikzpicture}[scale=0.6]
\begin{scope}[shift={(0,0)}]
%\draw[thin,yellow] (0,0) grid (4,4);
\node at (0.5,3.35) {$+$};
\draw[thick,dashed] (0,3) -- (1,3);
\draw[thick,dashed] (0,0) -- (1,0);
\draw[thick,<-] (0.5,3) -- (0.5,1.75);
\node at (0.6,1.25) {$q^*$};
\end{scope}

\begin{scope}[shift={(4,0)}]
%\draw[thin,yellow] (0,0) grid (4,4);
\node at (0.5,3.35) {$+$};
\draw[thick,dashed] (0,3) -- (1,3);
\draw[thick,dashed] (0,0) -- (1,0);
\draw[thick] (0.5,3) -- (0.5,2);
\draw[thick,<-] (0.5,2) -- (0.5,0);
\node at (0.5,-0.35) {$+$};
\node at (1.5,1.5) {$=$};
\node at (2.5,1.25) {$\displaystyle{\sum_{q\in Q}}$};

\begin{scope}[shift={(3.5,0)}]
%\draw[thin,yellow] (0,0) grid (4,4);
\node at (0.5,3.35) {$+$};
\draw[thick,dashed] (0,3) -- (1,3);
\draw[thick,dashed] (0,0) -- (1,0);
\draw[thick] (0.5,3) -- (0.5,2);
\draw[thick,<-] (0.5,3) -- (0.5,2);
\node at (1.00,2.1) {$q$};

\draw[thick,<-] (0.5,1) -- (0.5,0);
\node at (1.15,0.9) {$q^{\ast}$};
\node at (0.5,-0.35) {$+$};
\end{scope}
\end{scope}

\begin{scope}[shift={(12,0)}]
%\draw[thin,yellow] (0,0) grid (4,4);
\node at (0.5,3.35) {$+$};
\draw[thick,dashed] (0,3) -- (1,3);
\draw[thick,dashed] (0,0) -- (1,0);

\draw[thick,fill] (0.65,2.25) arc (0:360:1.5mm);
\node at (1,2.25) {$a$};

\draw[thick,<-] (0.5,3) -- (0.5,1.5);
\node at (0.5,1.10) {$q$};
\node at (1.75,1.5) {$=$};
\node at (3.25,1.25) {$\displaystyle{\sum_{q'\in \delta(q,a)}}$};

\begin{scope}[shift={(4.45,0)}]
%\draw[thin,yellow] (0,0) grid (4,4);
\node at (0.5,3.35) {$+$};
\draw[thick,dashed] (0,3) -- (1,3);
\draw[thick,dashed] (0,0) -- (1,0);

\draw[thick,<-] (0.5,3) -- (0.5,1.5);
\node at (0.55,1.10) {$q'$};
\end{scope}

\end{scope}

\end{tikzpicture}
    \caption{Left: an endpoint decorated by $q$. Middle: decomposition of the identity endomorphism of $+$. Right: action of $a$ on an endpoint labelled $q$.}
    \label{figure-B3}
\end{figure}
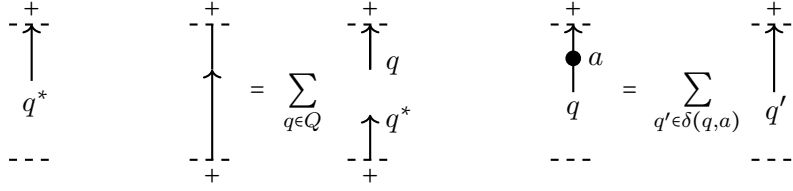

%%%%%%%%%%%%%%%%%
% cyclic language dependence 
%%%%%%%%%%%%%%%%%

\subsection{Trace languages of automata with a given interval language}\label{subsec_dependence}

Let us fix a regular language $L$ and consider an automaton $(Q)$ with the language or \emph{interval language} $L$:  \[L=L_{\I,(Q)}.
\] 
To $(Q)$ there is also associated a strongly circular language $L_{\circ,(Q)}$, the trace language of $(Q)$. We explain here that there is a large variety of possible trace languages for automata with the fixed interval language $L$. 

\vspace{0.1in} 

Pick an automaton $(Q')$ with no initial or accepting states, so that $L_{\I,(Q')}=\emptyset$ is the empty language. The trace language $L_{\circ,(Q')}$ is a regular strongly circular language. The disjoint union automaton $(Q)\sqcup (Q')$ has the same interval language as $(Q)$ but its trace language is the sum 
\[
L_{\circ,(Q)\sqcup (Q')} = L_{\circ,(Q)} + L_{\circ,(Q')}
\]
of the trace languages for the two automata. Here, we view languages as elements of $\mathcal{P}(\Sigma^{\ast})$, the powerset of the set of words $\Sigma^{\ast}$, which is naturally a $\Bool$-module  under the union of sets. Thus, the sum of languages is defined as the union of languages. 

We see that  the trace language $L_{\circ,(Q)}$ can be beefed up by adding to it the trace language of any nondeterministic finite automaton:  

\begin{prop}\label{prop_add_lang} Given a regular language $L$, if some automaton for $L$ has the trace language $L_{\circ}$, then all languages of the form 
\[
 L_{\circ} + L',
\]
where $L'$ is the trace language of some automaton, are the trace languages of automata with the interval language $L$. 
\end{prop} 

Note that the sum of two trace languages (respectively, of two strongly circular languages) is a trace language (respectively, a strongly circular language). 

\vspace{0.07in} 

Let $(Q')$ be an automaton as above, with no initial or accepting states, and $Q''\subset Q'$ a subset of the states of $Q'$. Define the language $L''$ to consist of words $\omega$ such that there is a circular path $\omega$ in $(Q')$ that passes through a vertex of $Q''$. The language $L''$ is strongly circular. Figure~\ref{figure-H1} shows an example, with a 2-state automaton $(Q')$. 

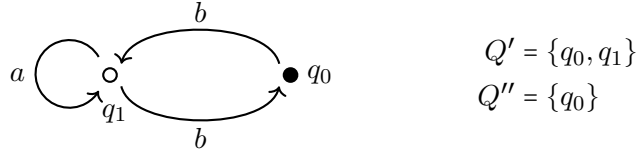
\begin{figure}
    \centering
\begin{tikzpicture}[scale=0.6]
\begin{scope}[shift={(0,0)}]
%\draw[thin,yellow] (0,0) grid (4,4);

\draw[thick] (0.15,2) arc (0:360:1.5mm);
\draw[thick,fill] (4.15,2) arc (0:360:1.5mm);

\draw[thick,<-] (0.25,2.25) .. controls (0.5,3.25) and (3.5,3.25) .. (3.75,2.25);

\draw[thick,->] (0.25,1.75) .. controls (0.5,0.75) and (3.5,0.75) .. (3.75,1.75);

\draw[thick,->] (-0.25,2.40) arc (30:330:0.75);

\node at (2,3.40) {$b$};
\node at (2,0.60) {$b$};

\node at (-2.05,2) {$a$};

\node at (0.1,1.15) {$q_1$};

\node at (4.65,2) {$q_0$};

\node at (10,2.5) {$Q'=\{ q_0,q_1\}$};
\node at (9.5,1.5) {$Q'' = \{ q_0\}$};

\end{scope}

\end{tikzpicture}
    \caption{Two-state automaton $(Q')$ with $Q''=\{q_0\}\subset Q'=\{ q_0,q_1\}$ and language $L''$ in \eqref{eq_language_L2} of circular words that pass through $q_0$.}
    \label{figure-H1}
\end{figure}

\vspace{0.1in}

Take $Q''=\{q_0\} \subset Q'=\{q_0,q_1\}$ in that example. The language 
\begin{equation}\label{eq_language_L2} 
L'' \ = \ 
(ba^{\ast}b)^{\ast} + (a^{\ast}+b^2)^{\ast}b^2 (a^{\ast}+b^2)^{\ast}
\end{equation} 
of circular paths that pass through $q_0$ is strongly circular. 

Language $L''$ is not the trace language of any automaton. Indeed, suppose it is the trace language of an automaton $(Q_1)$. Notice that words in $L''$ contain subwords $a^n$ for all $n$. This implies that there exists $m\ge 1$ and a state $q\in Q_1$ with a circular path $a^m$ from $q$ to $q$. Then the trace language of $(Q_1)$ contains $a^m$. This is a contradiction with \eqref{eq_language_L2} or with Figure~\ref{figure-H1}, since any nonempty word in $L''$ contains the letter $b$.  

\begin{corollary}
    Not every strongly circular language is the trace language of an automaton. 
\end{corollary}

It is a natural question whether any strongly circular language is the language $L''$ associated to an automaton $(Q')$ and a subset $Q''\subset Q'$ of its states. Language $L''$ consists of words realizable as circular paths in $(Q')$ that go through a state in $Q''$.

\begin{prop} Any regular strongly circular one-letter language is the trace language of some automaton. 
\end{prop} 

\begin{proof} Let $L\subset a^{\ast}$ be a  regular strongly circular one-letter language. Then $L$ is eventually periodic  \cite{matos1994periodic}, so that for some $N$ and $k\ge 1$ we have $a^m \in L \leftrightarrow a^{m+k}\in L, m\ge N$. Let $j_1,\dots, j_r$ be exponents of words in $L$ that are less than $N$. Since $L$ is strongly circular, with any word it contains all its powers. This implies that $k=mn$ for some $n,m$ such that $L\cap a^Na^{\ast}= a^N(a^n)^{\ast}$. 

We can now realize $L$ as the trace language of the automaton $(Q)$ which is the disjoint union of oriented loop automata of lengths $j_1,\dots, j_r$ and a \emph{flower} automaton which is the one-vertex union of oriented loops of lengths $N,N+n,\dots, N+(m-1)n$. 
\end{proof}

\begin{remark} Let us also refer the reader to the related notion of a \emph{strongly cyclic} language in~\cite{BCR96,Cart97}.  
\end{remark}

{\it Coverings of automata.}
An automaton $(Q)$ can be viewed as a decorated oriented graph, possibly with loops and multiple edges. Viewing the underlying graph as a topological space $Y=Y_{(Q)}$, pick a finite locally-trivial covering $p:Z\lra Y$ (in particular, $p$ is surjective). Topological space $Z$ can be viewed as a graph. We lift all decorations from $Y$ to $Z$ to turn it into an automaton. Namely, let $Q':= p^{-1}(Q)$ be the set of vertices of $Z$. Define sets of initial and accepting vertices of $(Q')$ by $Q'_{\init}:=p^{-1}(Q_{\init})$, $Q'_{\t}:= p^{-1}(Q_{\t})$, i.e., sets of initial and accepting vertices of $(Q')$ are the inverse image under $p$ of sets of initial and accepting vertices of $(Q)$.  Edges of $Z$ are oriented to match orientation with edges of $Y$, so that $p$ applied to any edge preserves its orientation. Labels on edges of $Z$ must match those of $Y$ under the map $p$ as well. 

Two examples of automata and their covering automata are shown in Figures~\ref{figure-F1} and~\ref{figure-F2}. In both examples graphs $Y$ underlying automata $(Q)$ are connected and coverings have degree three. 
 
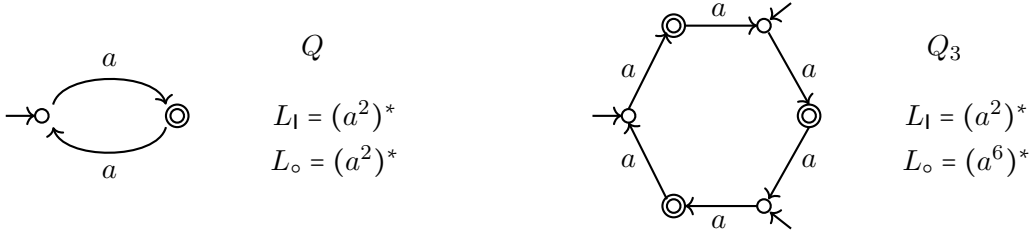
\begin{figure}
    \centering
\begin{tikzpicture}[scale=0.6]
\begin{scope}[shift={(0,0)}]
%\draw[thin,yellow] (0,0) grid (4,4);

\draw[thick,->] (-0.80,2) -- (-0.15,2);

\draw[thick] (0.15,2) arc (0:360:1.5mm);
\draw[thick] (3.15,2) arc (0:360:1.5mm);
\draw[thick] (3.25,2) arc (0:360:2.5mm);

\draw[thick,->] (0.25,2.25) .. controls (0.45,3) and (2.55,3) .. (2.75,2.25);
\draw[thick,<-] (0.25,1.75) .. controls (0.45,1) and (2.55,1) .. (2.75,1.75);

\node at (1.5,3.25) {$a$};
\node at (1.5,0.75) {$a$};

\node at (6,3.5) {$Q$};

\node at (6.5,2) {$L_{\I} = (a^2)^*$};
\node at (6.5,1) {$L_{\circ} = (a^2)^*$};
\end{scope}

\begin{scope}[shift={(13,0)}]
%\draw[thin,yellow] (0,0) grid (4,4);

\draw[thick,->] (-0.80,2) -- (-0.15,2);
\draw[thick,->] (3.6,4.5) -- (3.13,4.13);
\draw[thick,->] (3.6,-0.5) -- (3.13,-0.13);

\draw[thick] (0.15,2) arc (0:360:1.5mm);

\draw[thick] (1.15,4) arc (0:360:1.5mm);
\draw[thick] (1.25,4) arc (0:360:2.5mm);

\draw[thick] (3.15,4) arc (0:360:1.5mm);

\draw[thick] (4.15,2) arc (0:360:1.5mm);
\draw[thick] (4.25,2) arc (0:360:2.5mm);

\draw[thick] (3.15,0) arc (0:360:1.5mm);

\draw[thick] (1.15,0) arc (0:360:1.5mm);
\draw[thick] (1.25,0) arc (0:360:2.5mm);

\draw[thick,->] (0,2.15) -- (0.85,3.85);
\draw[thick,->] (1.25,4) -- (2.85,4);

\draw[thick,->] (3.10,3.85) -- (4,2.25);
\draw[thick,<-] (3.10,0.15) -- (4,1.75);

\draw[thick,<-] (1.25,0) -- (2.85,0);
\draw[thick,<-] (0,1.85) -- (0.85,0.15);

\node at (0,3) {$a$};
\node at (2,4.35) {$a$};
\node at (4,3) {$a$};
\node at (4,1) {$a$};
\node at (2,-0.35) {$a$};
\node at (0,1) {$a$};

\node at (7,3.5) {$Q_3$};

\node at (7.5,2) {$L_{\I}=(a^2)^*$};
\node at (7.5,1) {$L_{\circ}=(a^6)^*$};

\end{scope}

\end{tikzpicture}
    \caption{An automaton $(Q)$ for the language $(a^2)^{\ast}$ and its covering automaton $(Q_3)$. The two automata share the interval language but have different trace languages.}
    \label{figure-F1}
\end{figure}

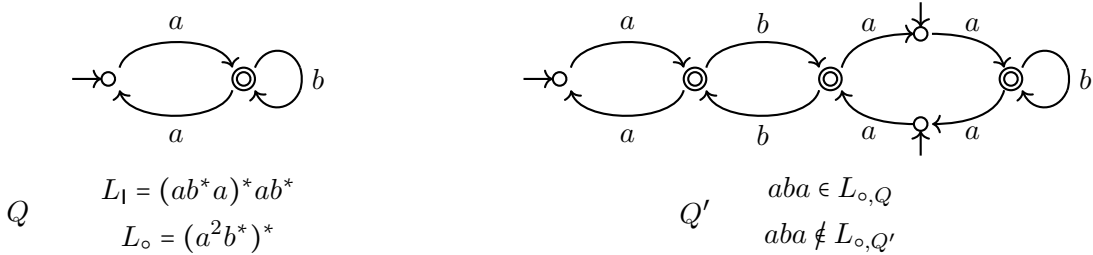
\begin{figure}
    \centering
\begin{tikzpicture}[scale=0.6]
\begin{scope}[shift={(0,0)}]
%\draw[thin,yellow] (0,0) grid (4,4);
\draw[thick,->] (-0.80,2) -- (-0.15,2);

\draw[thick] (0.15,2) arc (0:360:1.5mm);
\draw[thick] (3.15,2) arc (0:360:1.5mm);
\draw[thick] (3.25,2) arc (0:360:2.5mm);

\draw[thick,->] (0.25,2.25) .. controls (0.45,3) and (2.55,3) .. (2.75,2.25);
\draw[thick,<-] (0.25,1.75) .. controls (0.45,1) and (2.55,1) .. (2.75,1.75);

\draw[thick] (3.25,2.25) .. controls (3.45,2.75) and (4.05,2.75) .. (4.25,2.25);
\draw[thick] (4.25,2.25) .. controls (4.30,2.10) and (4.30,1.90) .. (4.25,1.75);
\draw[thick,<-] (3.25,1.75) .. controls (3.45,1.25) and (4.05,1.25) .. (4.25,1.75);
\node at (4.65,2) {$b$};

\node at (1.5,3.25) {$a$};
\node at (1.5,0.75) {$a$};

\node at (-2,-1) {$Q$};

\node at (2.0,-0.5) {$L_{\I} = (ab^*a)^*ab^*$};
\node at (2.0,-1.5) {$L_{\circ} = (a^2b^*)^*$};
\end{scope}

\begin{scope}[shift={(10,0)}]
%\draw[thin,yellow] (0,0) grid (10,4);
\draw[thick,->] (-0.80,2) -- (-0.15,2);

\draw[thick] (0.15,2) arc (0:360:1.5mm);
\draw[thick] (3.15,2) arc (0:360:1.5mm);
\draw[thick] (3.25,2) arc (0:360:2.5mm);

\draw[thick,->] (0.25,2.25) .. controls (0.45,3) and (2.55,3) .. (2.75,2.25);
\draw[thick,<-] (0.25,1.75) .. controls (0.45,1) and (2.55,1) .. (2.75,1.75);

\node at (1.5,3.25) {$a$};
\node at (1.5,0.75) {$a$};

\draw[thick,->] (3.25,2.25) .. controls (3.45,3) and (5.55,3) .. (5.75,2.25);
\draw[thick,<-] (3.25,1.75) .. controls (3.45,1) and (5.55,1) .. (5.75,1.75);

\node at (4.5,3.25) {$b$};
\node at (4.5,0.75) {$b$};

\draw[thick] (6.15,2) arc (0:360:1.5mm);
\draw[thick] (6.25,2) arc (0:360:2.5mm);

\draw[thick,->] (6.25,2.25) .. controls (6.45,3) and (7.65,3) ..  (7.85,3);
\draw[thick,<-] (6.25,1.75) .. controls (6.45,1) and (7.65,1) ..  (7.85,1);

\node at (6.85,3.20) {$a$};
\node at (6.85,0.80) {$a$};

\draw[thick] (8.15,3) arc (0:360:1.5mm);
\draw[thick] (8.15,1) arc (0:360:1.5mm);

\draw[thick,->] (8,3.70) -- (8,3.15);
\draw[thick,->] (8,0.30) -- (8,0.85);

\draw[thick,->] (8.25,3) .. controls (8.45,3) and (9.65,3) ..  (9.85,2.25);
\draw[thick,<-] (8.25,1) .. controls (8.45,1) and (9.65,1) ..  (9.85,1.75);

\node at (9.15,3.20) {$a$};
\node at (9.15,0.80) {$a$};

\draw[thick] (10.15,2) arc (0:360:1.5mm);
\draw[thick] (10.25,2) arc (0:360:2.5mm);

\draw[thick] (10.25,2.25) .. controls (10.45,2.75) and (11.05,2.75) .. (11.25,2.25);
\draw[thick] (11.25,2.25) .. controls (11.30,2.10) and (11.30,1.90) .. (11.25,1.75);
\draw[thick,<-] (10.25,1.75) .. controls (10.45,1.25) and (11.05,1.25) .. (11.25,1.75);
\node at (11.65,2) {$b$};

\node at (3,-1) {$Q'$};

\node at (6.0,-0.5) {$aba\in L_{\circ, Q}$};
\node at (6.0,-1.5) {$aba\not\in L_{\circ,Q'}$};
\end{scope}

\end{tikzpicture}
    \caption{An automaton $(Q)$ for the language  $L_{\I} = (ab^*a)^*ab^*$ with the trace language  $L_{\circ} = (a^2b^*)^*$ and its covering automaton $(Q')$ with the same interval language $L_{\I}$ but a different trace language.}
    \label{figure-F2}
\end{figure}

\begin{prop} An automaton $(Q)$ and its covering automaton $(Q')$ have the same interval language, while the trace language of $(Q')$ is a subset of the trace language of $(Q)$: 
\[
 L_{\I,(Q')} \ = \ L_{\I,(Q)}, \ \ \  L_{\I,(Q')} \ \subset \ L_{\I,(Q)}.
\]
\end{prop} 
\begin{proof}
Note that any accepting path in $(Q')$ with a word $\omega$ projects to an accepting path in $(Q)$ carrying the same word, and any lifting of a path with a word $\omega$ in $(Q)$ gives a path in $(Q')$ with the same word, implying the equality of interval languages. A circular path in $(Q')$ projects to a circular path in $(Q)$, preserving the word, while a circular path in $(Q)$ does not always lift to a circular path in $(Q')$. 
\end{proof} 

\begin{example} Graph $Y$ underlying the automaton $(Q)$ in Figure~\ref{figure-F1} is a cycle and its connected coverings are parameterized by the degree $n$ of the covering. Denote by $(Q_n)$ the corresponding automaton. Then $(Q_1)=(Q)$ and $(Q_3)$ is also shown in Figure~\ref{figure-F1}. The interval and circular languages for $(Q)$ and $(Q_n)$ are 
\[
 L_{\I,(Q)} = L_{\I,(Q_n)} = (a^2)^*, \ \ 
 L_{\circ,(Q)} = (a^2)^{\ast}, \ \ L_{\circ,(Q_n)} = (a^{2n})^{\ast}. 
\]
In particular, as $n$ becomes large, the only short length word in the circular language for $(Q_n)$ is the empty word $\emptyset$. 
\end{example} 

Given an automaton $(Q)$ with the interval language $L$, assume that $(Q)$ has at least one oriented cycle, so that $L_{\circ,(Q)}$ contains a nonempty word. Arrange states of $(Q)$ around a circle and draw arrows $q\stackrel{a}{\lra}q'$, $a\in \Sigma$ so that they all go clockwise around the circle, at most one full rotation each. In particular, an arrow $q\stackrel{a}{\lra}q$ from a state to itself will make a full rotation around the circle. An example of such arrangement for the automaton $(Q)$ in Figure~\ref{figure-F2} is shown in Figure~\ref{figure-F3}. 

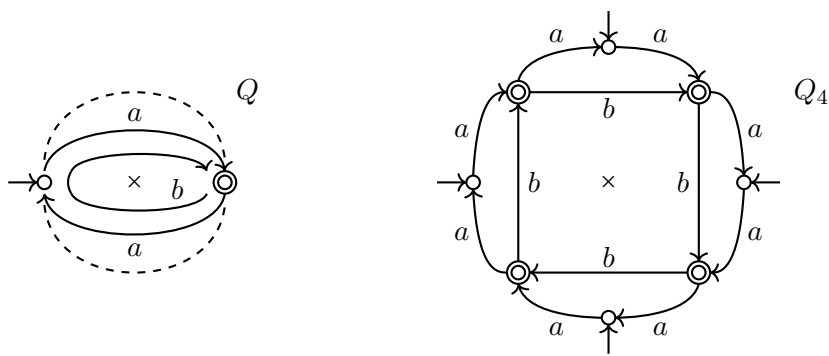
\begin{figure}
    \centering
\begin{tikzpicture}[scale=0.6]
\begin{scope}[shift={(0,0)}]
%\draw[thin,yellow] (0,0) grid (4,4);
\draw[thick,->] (-0.80,2) -- (-0.15,2);

\draw[thick] (0.15,2) arc (0:360:1.5mm);
\draw[thick] (4.15,2) arc (0:360:1.5mm);
\draw[thick] (4.25,2) arc (0:360:2.5mm);

\draw[thick,->] (0,2.25) .. controls (0.2,3.45) and (3.8,3.45) .. (4,2.25);
\draw[thick,<-] (0,1.75) .. controls (0.2,0.55) and (3.8,0.55) .. (4,1.75);

\begin{scope}[shift={(-0.15,0)}]
\draw[thick,<-] (3.75,2.25) .. controls (3.55,2.75) and (0.95,2.75) .. (0.75,2.25);
\draw[thick] (0.75,2.25) .. controls (0.65,2.10) and (0.65,1.90) .. (0.75,1.75);
\draw[thick] (3.75,1.75) .. controls (3.55,1.25) and (0.95,1.25) .. (0.75,1.75);
\node at (3.10,1.85) {$b$};
\node at (2.15,2) {$\times$};
\end{scope}
 
\node at (2.0,3.50) {$a$};
\node at (2.0,0.50) {$a$};

\node at (4.5,4) {$Q$};

\draw[thick,dashed] (0,2.5) .. controls (0.2,4.5) and (3.8,4.5) .. (4,2.5);

\draw[thick,dashed] (0,1.5) .. controls (0.2,-0.5) and (3.8,-0.5) .. (4,1.5);
\end{scope}

\begin{scope}[shift={(10.5,0)}]
%\draw[thin,yellow] (0,0) grid (5,5);

\draw[thick] (0.15,0) arc (0:360:1.5mm);
\draw[thick] (0.25,0) arc (0:360:2.5mm);

\draw[thick] (4.15,0) arc (0:360:1.5mm);
\draw[thick] (4.25,0) arc (0:360:2.5mm);

\draw[thick] (0.15,4) arc (0:360:1.5mm);
\draw[thick] (0.25,4) arc (0:360:2.5mm);

\draw[thick] (4.15,4) arc (0:360:1.5mm);
\draw[thick] (4.25,4) arc (0:360:2.5mm);

\draw[thick,->] (0,0.25) -- (0,3.75);
\draw[thick,<-] (0.25,0) -- (3.75,0);
\draw[thick,->] (0.25,4) -- (3.75,4);
\draw[thick,<-] (4,0.25) -- (4,3.75);

\node at (0.35,2) {$b$};
\node at (2,3.65) {$b$};
\node at (3.65,2) {$b$};
\node at (2,0.35) {$b$};

\node at (2,2) {$\times$};

\draw[thick] (-0.85,2) arc (0:360:1.5mm);
\draw[thick] (5.15,2) arc (0:360:1.5mm);
\draw[thick] (2.15,5) arc (0:360:1.5mm);
\draw[thick] (2.15,-1) arc (0:360:1.5mm);

\draw[thick,->] (-1.80,2) -- (-1.15,2);
\draw[thick,<-] (5.15,2) -- (5.80,2);

\draw[thick,<-] (2,5.15) -- (2,5.80);
\draw[thick,<-] (2,-1.15) -- (2,-1.80);

\node at (6.5,4) {$Q_4$};

\draw[thick,->] (-1,2.15) .. controls (-0.9,4) and (-0.45,4) .. (-0.25,4);
\draw[thick,<-] (-1,1.85) .. controls (-0.9,0) and (-0.45,0) .. (-0.25,0);

\draw[thick,<-] (5,2.15) .. controls (4.9,4) and (4.45,4) .. (4.25,4);
\draw[thick,->] (5,1.85) .. controls (4.9,0) and (4.45,0) .. (4.25,0);

\draw[thick,->] (0,4.25) .. controls (0.2,5) and (1.75,5) .. (1.85,5);
\draw[thick,<-] (0,-0.25) .. controls (0.2,-1) and (1.75,-1) .. (1.85,-1);

\draw[thick,<-] (4,4.25) .. controls (3.8,5) and (2.25,5) .. (2.15,5);
\draw[thick,->] (4,-0.25) .. controls (3.8,-1) and (2.25,-1) .. (2.15,-1);

\node at (0.85,5.25) {$a$};
\node at (3.15,5.25) {$a$};

\node at (0.85,-1.25) {$a$};
\node at (3.15,-1.25) {$a$};

\node at (-1.25,3.15) {$a$};
\node at (-1.25,0.85) {$a$};

\node at (5.25,3.15) {$a$};
\node at (5.25,0.85) {$a$};
\end{scope}

\end{tikzpicture}    
\caption{Left: automaton $(Q)$ with states arranged along a circle and edges going clockwise around the circle. Right: a 4-fold cyclic cover automaton $(Q_4)$ of that arrangement.}
    \label{figure-F3}
\end{figure}
 
\vspace{0.1in} 

Now for each $n\ge 1$ we can form the ``cyclic'' cover $(Q_n)$ of $(Q)$ by taking the cyclic $n$-cover of the circle and extending to a cover of the automaton $(Q)$. The resulting automaton $(Q_n)$ has as $n$ times as many states and edges as $(Q)$, with $(Q_1)=(Q)$. An example is shown in Figure~\ref{figure-F3} on the right. The following observation holds.

\begin{prop} Automata $(Q_n)$ all have the same interval language $L$. The trace language $L_{\circ,(Q_n)}$ for the automaton $(Q_n)$ does not contain any words of length less than $n$ other than the empty word $\emptyset$. The trace language $L_{\circ,(Q_n)}$  is infinite for each $n\ge 1$. 
\end{prop}

From automata $(Q_n)$ we obtain a family of TQFTs with defects with the same interval evaluation $\alpha_{\I}$ but circular evaluations given by languages $L_{\circ,(Q_n)}$ that shrink as $n$ increases, in the sense that 
$L_{\circ,(Q_{nm})}\subset L_{\circ,(Q_n)}$ and $L_{\circ,(Q_n)}$ does not contain words of length strictly between $0$ and $n$.

\begin{remark} A sample of coverings of the figure eight graph, in relation to subgroups of the free group $F_2$, can be found in A.~Hatcher's textbook~\cite[Section 1.3]{Hat02}. 
\end{remark}

{\it Weak coverings.} The covering automata construction can be generalized as follows. A  \emph{weak covering} $p:(Q')\lra (Q)$ of automata is a surjective map of underlying graphs $p:Y_{(Q')}\lra Y_{(Q)}$ such that 
\begin{itemize}
\item $p^{-1}(Q_{\init}) = Q'_{\init}, p^{-1}(Q_{\t}) = Q'_{\t}$, that is, $p$ preserves properties of a state to be initial and accepting, 
\item The label $a\in \Sigma$ of each edge of $(Q')$ is preserved by $p$, 
\item For each arrow $\gamma: q_1\stackrel{a}{\lra}q_2$ in $(Q)$, $a\in \Sigma$, and any $q_1'\in p^{-1}(q_1)$ there exists an arrow $\gamma':q_1' \stackrel{a}{\lra} q_2'$ which lifts $\gamma$, that is $q_2'\in p^{-1}(q_2)$ or, equivalently, $p(\gamma')=\gamma$. 
\end{itemize}
See Figure~\ref{figure-G1} for an example of a weak covering. 

\begin{figure}
    \centering
\begin{tikzpicture}[scale=0.6]
\begin{scope}[shift={(0,0)}]
%\draw[thin,yellow] (0,0) grid (4,4);
\draw[thick,->] (-0.80,2) -- (-0.15,2);

\draw[thick] (0.15,2) arc (0:360:1.5mm);
\draw[thick] (5.15,2) arc (0:360:1.5mm);
\draw[thick] (5.25,2) arc (0:360:2.5mm);

\draw[thick,->] (0.25,2.25) .. controls (0.45,2.8) and (4.55,2.8) .. (4.75,2.25);
\draw[thick,<-] (0.25,1.75) .. controls (0.45,1.2) and (4.55,1.2) .. (4.75,1.75);

\draw[thick] (5.25,2.25) .. controls (5.45,2.75) and (6.05,2.75) .. (6.25,2.25);
\draw[thick] (6.25,2.25) .. controls (6.30,2.10) and (6.30,1.90) .. (6.25,1.75);
\draw[thick,<-] (5.25,1.75) .. controls (5.45,1.25) and (6.05,1.25) .. (6.25,1.75);
\node at (6.65,2) {$b$};

\node at (2.5,3.05) {$a$};
\node at (2.5,0.95) {$a$};

\node at (-0.1,1.2) {$q_1$};
\node at (4.9,1.2) {$q_2$};
\end{scope}

\begin{scope}[shift={(0,5.0)}]
%\draw[thin,yellow] (0,-1) grid (6,4);

\draw[thick,->] (-0.80,4) -- (-0.15,4);
\draw[thick,->] (-0.80,2) -- (-0.15,2);
\draw[thick,->] (-0.80,0) -- (-0.15,0);

\draw[thick] (0.15,4) arc (0:360:1.5mm);
\draw[thick] (0.15,2) arc (0:360:1.5mm);
\draw[thick] (0.15,0) arc (0:360:1.5mm);

\draw[thick] (5.15,0.5) arc (0:360:1.5mm);
\draw[thick] (5.25,0.5) arc (0:360:2.5mm);

\draw[thick] (5.15,3.5) arc (0:360:1.5mm);
\draw[thick] (5.25,3.5) arc (0:360:2.5mm);

\draw[thick,->] (5.25,3.35) .. controls (6.60,3.25) and (6.60,0.75) .. (5.25,0.65);
\node at (6.55,2) {$b$};

\draw[thick,<-] (5.35,3.65) .. controls (7.80,3.25) and (7.80,0.75) .. (5.45,0.35);
\node at (7.50,2) {$b$};

\draw[thick] (5.1,0.25) .. controls (4.25,-0.35) and (5,-0.77) .. (5.11,-0.80);

\draw[thick] (5.10,-0.80) .. controls (5.20,-0.80) and (5.30,-0.85) .. (5.5,-0.80);

\draw[thick,<-] (5.20,0.25)  .. controls (6.5,0.05) and (6.0,-0.75) .. (5.5,-0.80);

\node at (6.30,-0.25) {$b$};

\node at (0,4.60) {$q_1'''$};
\node at (0,2.60) {$q_1''$};
\node at (0,0.60) {$q_1'$};

\node at (5.30,4.25) {$q_2''$};
\node at (5.2,1.3) {$q_2'$};

\draw[thick,->] (0.10,4.10) .. controls (0.40,5) and (4.60,5) .. (4.85,3.75);

\draw[thick,<-] (0,3.85) .. controls (0.25,2.75) and (4.5,2.75) .. (4.75,3.53);

\draw[thick,->] (0.10,2.1) .. controls (0.40,2.1) and (4.25,2.1) .. (4.8,3.30);

\draw[thick,->] (0.10,1.9) .. controls (0.35,1.9) and (4.60,1.9) .. (4.85,0.70);

\draw[thick] (0.15,0) .. controls (0.40,0) and (1.55,0) .. (2.0,0.30);
\draw[thick] (2.1,0.40) -- (4.2,1.27);
\draw[thick,->] (4.3,1.37) .. controls (4.5,1.5) and (4.80,1.5) .. (5.17,3.25);

\draw[thick,<-] (0.00,1.85) .. controls (0.25,0) and (4.50,0) .. (4.75,0.5);

\draw[thick,<-] (0.0,-0.15) .. controls (0.40,-1) and (4.5,-1) .. (4.83,0.30);

\node at (3,5.0) {$a$};
\node at (2.5,3.30) {$a$};
\node at (1.75,2.45) {$a$};
\node at (3.5,1.8) {$a$};
\node at (1,1) {$a$};
\node at (2.8,1.00) {$a$};
\node at (2.5,-0.5) {$a$};

\end{scope}

\end{tikzpicture}
    \caption{A weak covering, with three states mapping to $q_1$ and two states mapping to $q_2$.  }
    \label{figure-G1}
\end{figure}
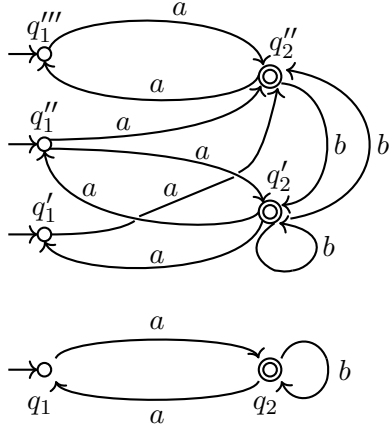

\begin{prop}  Given a weak covering $p:(Q')\lra (Q)$ of automata, the two automata share the interval language, while the trace language of $(Q')$ is a subset of that of $(Q)$: 
\[
 L_{\I,(Q')} \ = \ L_{\I,(Q)}, 
 \hspace{0.75cm} 
 L_{\I,(Q')} \ \subset \ L_{\I,(Q)}.
\]
\end{prop}

The above constructions show that there is a lot of variety in possible trace languages $L_{\circ}$ of automata $(Q)$ with a fixed interval language $L$. Taking coverings and weak coverings of $(Q)$ shrinks the trace language, while taking the disjoint union of $(Q)$ and an automaton $(Q')$ with the empty interval language, see Proposition~\ref{prop_add_lang}, enlarges the trace language. 

\begin{question} Given a regular language $L$, is there an efficient classification of strongly circular languages $L_{\circ}$ that are trace languages of automata with the interval language $L$?
\end{question}

A similar question may be posted with ``$\Tau$-automata'' replacing ``automata'' above, see Section~\ref{subsec_aut_top} for $\Tau$-automata.  

\vspace{0.1in} 

{\it Trimming an automaton.}
For an automaton $(Q)$, denote by $Q''\subset Q$ the subset such that $q\in Q''$ if and only if $q$ is either a state in a path from some initial state (state in $Q_{\init}$) to an accepting state or $q$ is a state in some oriented loop in the graph $(Q)$. Denote by $Q'\subset Q$ the set of states reachable from states in $Q''$. The $\Bool$-submodule $\Bool Q'$ of $\Bool Q$ is closed under the action of $\Sigma$ and we turn $Q'$ into the automaton $(Q')$ using that action of $\Sigma$, with the set of initial states -- the intersection $Q'\cap Q_{\init}$ and the set of accepting state -- the intersection $Q'\cap Q_{\t}$.

$\Bool$-submodule $\Bool(Q'\setminus Q'')$ is stable under $\Sigma$. The quotient of $\Bool Q'$ by this submodule produces a free $\Bool$-module $\Bool Q''$. Form the automaton $(Q'')$ on the set of states $Q''$, with the induced $\Sigma$-action, initial states $Q''\cap Q_{\in}$ and terminal states $Q''\cap Q_{\t}$.  

Thus, from the automaton $(Q)$ we first pass to the $\Bool[\Sigma^{\ast}]$-submodule $\Bool Q'$ and the associated automaton $(Q')$, then to the quotient $\Bool[\Sigma^{\ast}]$-module $\Bool Q''$ of $\Bool Q'$ and the associated automaton $(Q'')$.  
The following observation is clear. 

\begin{prop} Automata $(Q),(Q'),$ and $(Q'')$ share the same pair of interval and circular languages $(L_{\I,(Q)},L_{\circ,(Q)})$. 
\end{prop} 

The proposition says that the three TQFTs associated to the three automata evaluate the same on all closed morphisms in $\CobSI$ (by a closed morphism in a monoidal category we mean an endomorphism of the identity object $\one$). 

Passage from $(Q)$ to $(Q'')$ and from $\Bool Q$ to its subquotient $\Bool[\Sigma^{\ast}]$-module $\Bool Q''$ is analogous to passing from an automaton to the associated trim automaton.

%%%%%%%%%%%%%%%
% path integral
%%%%%%%%%%%%%%%

\subsection{Path integral interpretation of automata}
\label{subsec_path}

Consider a regular language $L$ and an automaton $(Q)$ that describes it. In the graph of the automaton oriented edges are labelled by letters (elements of $\Sigma$), while in the category $\CobSI$ it is vertices (defects) inside a cobordism that are labelled by elements of $\Sigma$. 

Let us now pass to the Poincar\'e dual decomposition of our cobordisms. Suppose given a morphism  $u\in\CobSI$ from a sign sequence $\varepsilon$ to $\varepsilon'$, thus a cobordism decorated as earlier. Consider the Poincar\'e dual of the decomposition of $u$ into defects and intervals between defects. Now each such interval becomes a vertex and a defect becomes an interval, see Figure~\ref{figure-A5} for an example. The orientation of intervals and circles is preserved. 

\vspace{0.1in} 

\begin{figure}
    \centering
\begin{tikzpicture}[scale=0.6]
\begin{scope}[shift={(0,0)}]
%\draw[thin,yellow] (0,0) grid (8,4);
\draw[thick,dashed] (0,4) -- (8,4);
\draw[thick,dashed] (0,0) -- (8,0);
\node at (0.5,4.35) {$+$};
\node at (1.5,4.35) {$-$};
\node at (2.5,4.35) {$+$};
\node at (3.5,4.35) {$-$};

\node at (0.5,-0.35) {$+$};
\node at (1.5,-0.35) {$+$};
\node at (3.5,-0.35) {$+$};

\draw[thick] (0.50,4) .. controls (0.40,3.5) and (0.75,2.50) .. (0.75,1.97);
\draw[thick,<-] (0.75,2) .. controls (0.75,1.50) and (0.75,0.50) .. (0.50,0);
\draw[thick,fill] (0.83,2.5) arc (0:360:1.5mm);
\node at (0.15,2.5) {$a$};

\draw[thick,fill] (0.85,1) arc (0:360:1.5mm);
\node at (0.15,1) {$b$};

\draw[thick,->] (1.5,4) .. controls (1.6,2.5) and (2.4,2.5) .. (2.5,4);
\draw[thick,fill] (2.15,2.85) arc (0:360:1.5mm);
\node at (2,2.35) {$c$};

\draw[thick,<-] (3,3) .. controls (3.15,3.25) and (3.75,3.75) .. (3.5,4);

\draw[thick,->] (1.5,0) .. controls (1.5,0.25) and (1.5,1.25) .. (2.5,1.25);
\draw[thick,fill] (1.95,1) arc (0:360:1.5mm);
\node at (1.40,1.35) {$b$};

\draw[thick,->] (3.5,0) .. controls (3.75,0.25) and (3,1.5) .. (3,2);

\draw[thick,<-] (4,2) .. controls (4.25,2.5) and (5,2) .. (6,3);
\draw[thick,fill] (4.75,2.30) arc (0:360:1.5mm);
\node at (4.50,2.75) {$c$};
\draw[thick,fill] (5.65,2.60) arc (0:360:1.5mm);
\node at (5.35,3.05) {$a$};

\draw[thick,<-] (6,1) arc (0:360:0.75);
\draw[thick,fill] (4.65,1) arc (0:360:1.5mm);
\node at (4,1.1) {$b$};

\draw[thick,<-] (8,3) arc (0:360:0.75);

\draw[thick,<-] (6.5,2) .. controls (6.75,1) and (7.25,1) .. (7.5,1);

\draw[thick,->] (9.5,1.5) -- (11.5,1.5);
\node at (10.5,2.75) {Poincar\'e};
\node at (10.5,2.00) {dual};
\end{scope}

\begin{scope}[shift={(13,0)}]
%\draw[thin,yellow] (0,0) grid (8,4);
\draw[thick,dashed] (0,4) -- (8,4);
\draw[thick,dashed] (0,0) -- (8,0);
\node at (0.5,4.5) {$+$};
\node at (1.5,4.5) {$-$};
\node at (2.5,4.5) {$+$};
\node at (3.5,4.5) {$-$};

\node at (0.5,-0.5) {$+$};
\node at (1.5,-0.5) {$+$};
\node at (3.5,-0.5) {$+$};

\draw[thick] (0.50,4) .. controls (0.40,3.5) and (0.75,2.50) .. (0.75,1.47);
\draw[thick,<-] (0.75,1.50) .. controls (0.75,1.25) and (0.75,0.50) .. (0.50,0);

\draw[thick,fill,white] (0.65,4) arc (0:360:0.15cm);
\draw[thick] (0.65,4) arc (0:360:0.15cm);

\draw[thick,fill,white] (0.65,0) arc (0:360:0.15cm);
\draw[thick] (0.65,0) arc (0:360:0.15cm);
 
\node at (0.25,2.8) {$a$};
\draw[thick,fill,white] (0.90,1.85) arc (0:360:0.15cm);
\draw[thick] (0.90,1.85) arc (0:360:0.15cm);
\node at (0.30,0.8) {$b$};

\draw[thick,->] (1.5,4) .. controls (1.6,2.5) and (2.4,2.5) .. (2.5,3.85);
\draw[thick,fill,white] (1.65,4) arc (0:360:0.15cm);
\draw[thick] (1.65,4) arc (0:360:0.15cm);
\draw[thick,fill,white] (2.65,4) arc (0:360:0.15cm);
\draw[thick] (2.65,4) arc (0:360:0.15cm);
\node at (2,2.45) {$c$};

\draw[thick,fill,white] (3.65,4) arc (0:360:0.15cm);
\draw[thick] (3.65,4) arc (0:360:0.15cm);

\draw[thick,->] (1.5,0) .. controls (1.5,0.25) and (1.5,1.25) .. (2.5,1.25);
\draw[thick,fill,white] (1.65,0) arc (0:360:0.15cm);
\draw[thick] (1.65,0) arc (0:360:0.15cm);
\draw[thick,fill,white] (2.80,1.25) arc (0:360:0.15cm);
\draw[thick] (2.80,1.25) arc (0:360:0.15cm);
\node at (2,0.70) {$b$};

\draw[thick,fill,white] (3.65,0) arc (0:360:0.15cm);
\draw[thick] (3.65,0) arc (0:360:0.15cm);

\draw[thick,<-] (4.0,2.1) .. controls (4.25,2.5) and (5,2) .. (6,3);
\draw[thick,fill,white] (6.15,3) arc (0:360:0.15cm);
\draw[thick] (6.15,3) arc (0:360:0.15cm); 
\draw[thick,fill,white] (4.05,2) arc (0:360:0.15cm);
\draw[thick] (4.05,2) arc (0:360:0.15cm);
\node at (4.50,2.70) {$c$};
\draw[thick,fill,white] (5.15,2.40) arc (0:360:0.15cm);
\draw[thick] (5.15,2.40) arc (0:360:0.15cm);
\node at (5.35,3.00) {$a$};

\draw[thick,->] (4.5,1) arc (180:-180:0.75);
\node at (4.2,1.30) {$b$};
\draw[thick,fill,white] (6.15,1) arc (0:360:0.15cm);
\draw[thick] (6.15,1) arc (0:360:0.15cm);

\draw[thick,<-] (8,3) arc (0:360:0.75);

\draw[thick] (7.5,1) arc (0:360:0.15cm);
\end{scope}

\end{tikzpicture}
    \caption{A morphism in $\Cob_{\Sigma,\I}$ and its Poincar\'e dual presentation.}
    \label{figure-A5}
\end{figure}
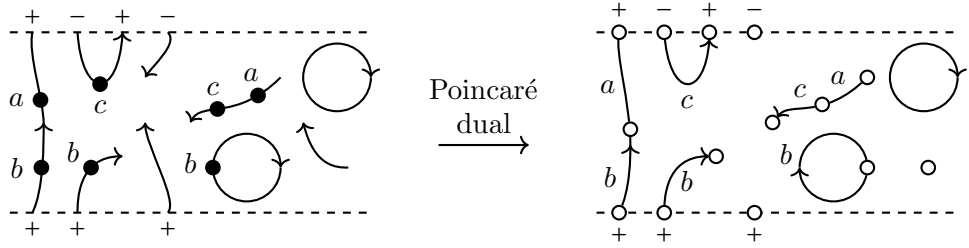

Let us understand this transformation, for a particular morphisms on the left of Figure~\ref{figure-A5}. A vertical interval (a morphism from $+$ to $+$) has defects $a,b$. It turns into an interval with three vertices (two at the endpoints) and two edges, labelled $a$ and $b$. An arc at the top,  which is a morphism from $\emptyset_0$ to $-+$ with a defect $c$, becomes an arc with two vertices and an interval labelled $c$. In particular, all boundary points become vertices and each edge inside the cobordism between two defects becomes a vertex as well. 

A half-interval with no defects on it (there are two such on the left of Figure~\ref{figure-A5}) becomes a single boundary vertex which carries a sign ($+$ or $-$). A half-interval with one or more defects on it becomes an  interval with two or more vertices and labels of defects becoming labels of edges (the example in the figure is for one defect $b$ on a half-interval). A floating interval with $k>0$ defects becomes an interval with $k+1$ vertices. A floating interval with no defects turns into a floating point. 

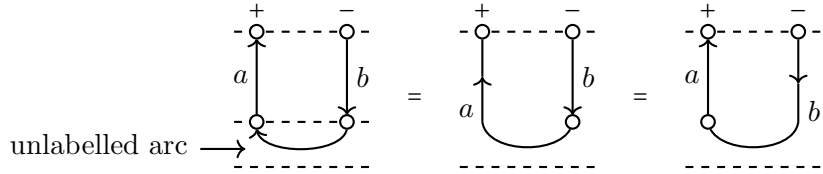
\begin{figure}
    \centering
\begin{tikzpicture}[scale=0.6]
\begin{scope}[shift={(0,0)}]
%\draw[thin,yellow] (0,0) grid (4,4);

\node at (0.5,3.45) {$+$};
\node at (2.5,3.45) {$-$};

\draw[thick,dashed] (0,3) -- (3,3);
\draw[thick,dashed] (0,1) -- (3,1);
\draw[thick,dashed] (0,0) -- (3,0);
\draw[thick,<-] (0.5,2.85) -- (0.5,1);
\node at (0.15,2) {$a$};

\draw[thick,->] (2.5,3) -- (2.5,1.15);
\node at (2.85,2) {$b$};

\draw[thick,<-] (0.5,0.85) .. controls (0.6,0.25) and (2.4,0.25) .. (2.5,0.85);

\draw[thick,fill,white] (0.65,3) arc (0:360:1.5mm);
\draw[thick] (0.65,3) arc (0:360:1.5mm);

\draw[thick,fill,white] (2.65,3) arc (0:360:1.5mm);
\draw[thick] (2.65,3) arc (0:360:1.5mm);

\draw[thick,fill,white] (0.65,1) arc (0:360:1.5mm);
\draw[thick] (0.65,1) arc (0:360:1.5mm);

\draw[thick,fill,white] (2.65,1) arc (0:360:1.5mm);
\draw[thick] (2.65,1) arc (0:360:1.5mm);

\node at (-3,0.5) {unlabelled arc};
\draw[thick,->] (-0.75,0.40) -- (0.25,0.40);
\node at (4,1.5) {$=$};
\end{scope}

\begin{scope}[shift={(5,0)}]
%\draw[thin,yellow] (0,0) grid (4,4);

\node at (0.5,3.45) {$+$};
\node at (2.5,3.45) {$-$};

\draw[thick,dashed] (0,3) -- (3,3);
\draw[thick,dashed] (0,0) -- (3,0);

\draw[thick] (0.5,2.85) -- (0.5,2);
\draw[thick,<-] (0.5,2) -- (0.5,1);
\node at (0.15,1.25) {$a$};

\draw[thick,->] (2.5,3) -- (2.5,1.15);
\node at (2.85,2) {$b$};

\draw[thick] (0.5,1) .. controls (0.6,0.25) and (2.4,0.25) .. (2.5,1);

\draw[thick,fill,white] (0.65,3) arc (0:360:1.5mm);
\draw[thick] (0.65,3) arc (0:360:1.5mm);

\draw[thick,fill,white] (2.65,3) arc (0:360:1.5mm);
\draw[thick] (2.65,3) arc (0:360:1.5mm);

\draw[thick,fill,white] (2.65,1) arc (0:360:1.5mm);
\draw[thick] (2.65,1) arc (0:360:1.5mm);
\node at (4,1.5) {$=$};
\end{scope}

\begin{scope}[shift={(10,0)}]
%\draw[thin,yellow] (0,0) grid (4,4);

\node at (0.5,3.45) {$+$};
\node at (2.5,3.45) {$-$};

\draw[thick,dashed] (0,3) -- (3,3);
\draw[thick,dashed] (0,0) -- (3,0);
\draw[thick,<-] (0.5,2.85) -- (0.5,1);
\node at (0.15,2) {$a$};

\draw[thick,->] (2.5,3) -- (2.5,1.85);
\draw[thick] (2.5,1.85) -- (2.5,1);
\node at (2.85,1.25) {$b$};

\draw[thick] (0.5,1) .. controls (0.6,0.25) and (2.4,0.25) .. (2.5,1);

\draw[thick,fill,white] (0.65,3) arc (0:360:1.5mm);
\draw[thick] (0.65,3) arc (0:360:1.5mm);

\draw[thick,fill,white] (2.65,3) arc (0:360:1.5mm);
\draw[thick] (2.65,3) arc (0:360:1.5mm);

\draw[thick,fill,white] (0.65,1) arc (0:360:1.5mm);
\draw[thick] (0.65,1) arc (0:360:1.5mm);
\end{scope}

\end{tikzpicture}
    \caption{An unlabelled arc is composed with two labelled arcs, and then one of the two vertices of the unlabelled arc can be erased. One may obtain similar diagrams for the opposite orientation.}
    \label{figure-A6}
\end{figure}

An arc or a circle with no defects remain as they are. This creates a minor inconvenience -- one should think of such an arc as unlabeled, but when composed with a labelled interval, the label of the latter can be extended to the arc, see Figure~\ref{figure-5}.

\begin{remark} \label{remark_pi} A related construction  starts with a category $\mcC_{\Sigma}$ with a single object $W$ with generating morphisms $a:W\lra W$ for each $a\in \Sigma$, with no relations on these morphisms, so that $\End_{\mcC_{\Sigma}}(W)\cong \Sigma^{\ast}$. One then passes to the rigid monoidal completion $\widetilde{\mcC_{\Sigma}}$, a category with objects -- sequences of signed objects of $\mcC_{\Sigma}$ and morphisms -- oriented one-manifolds decorated by morphisms in $\mcC_{\Sigma}$. In this category, endomorphisms of the unit object $\one$ (the empty sequence) are finite unions of loops in $\mcC_{\Sigma}$, that is, pairs $(Y,\gamma)$, where $Y\in \Ob(\mcC_{\Sigma})$ and $\gamma$ is an endomorphism of $Y$, modulo the equivalence relation: for any morphisms $\gamma_1:Y_1\lra Y_2$, $\gamma_2:Y_2\lra Y_1$ the pairs $(Y_1,\gamma_2\gamma_1)$ and $(Y_2,\gamma_1\gamma_2)$ are equivalent. 

An unlabelled arc in this category,  see Figure~\ref{figure-A6}, can be thought of as an arc labelled by the identity morphism of the object assigned to its boundary points. Any identity morphism in $\widetilde{\mcC_{\Sigma}}$ will be given by a union of such unlabelled arcs, going vertically (rather than sideways, as in Figure~\ref{figure-A6}). 
Such a rigid symmetric monoidal completion category $\widetilde{\mcC}$ can be defined for any small category $\mcC$. 
\end{remark} 

The Poincar\'e dual setup matches the graph description of finite state automata, for now an oriented interval between two vertices in a floating component of a cobordism is labelled by an element of $\Sigma$, similar to labeling of oriented intervals in the graph of the automaton by letters in $\Sigma$. 

A word $\omega\in \Sigma^{\ast}$ defines an oriented floating interval $I(\omega)$, see Figure~\ref{figure-A7} on the left, the Poincar\'e dual of the one in Figure~\ref{figure-5}.  

\vspace{0.1in} 

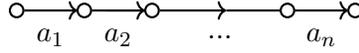
\begin{figure}
    \centering
\begin{tikzpicture}[scale=0.6]

\begin{scope}[shift={(0.1,0)}]
%\draw[thin,yellow] (0,0) grid (6,4);
%\draw[thick,dashed] (-0.5,3) -- (8,3);
%\draw[thick,dashed] (-0.5,0) -- (8,0);

\draw[thick,->] (0,1.5) -- (1.35,1.5);
\draw[thick,->] (1.35,1.5) -- (2.85,1.5);
\draw[thick,->] (2,1.5) -- (4.65,1.5);
\draw[thick,->] (4.65,1.5) -- (7.35,1.5);

\draw[thick,fill,white] (0.15,1.5) arc (0:360:1.5mm);
\draw[thick] (0.15,1.5) arc (0:360:1.5mm);

\node at (0.75,0.9) {$a_1$};

\draw[thick,fill,white] (1.65,1.5) arc (0:360:1.5mm);
\draw[thick] (1.65,1.5) arc (0:360:1.5mm);

\node at (2.25,0.9) {$a_2$};

\node at (4.5,0.85) {$\cdots$};

\draw[thick,fill,white] (3.15,1.5) arc (0:360:1.5mm);
\draw[thick] (3.15,1.5) arc (0:360:1.5mm);

\draw[thick,fill,white] (6.15,1.5) arc (0:360:1.5mm);
\draw[thick] (6.15,1.5) arc (0:360:1.5mm);

\node at (6.75,0.9) {$a_n$};

\draw[thick,fill,white] (7.65,1.5) arc (0:360:1.5mm);
\draw[thick] (7.65,1.5) arc (0:360:1.5mm);
\end{scope}

\end{tikzpicture}
    \caption{Decorated graph $I(\omega)$ associated to a word $\omega = a_1 \cdots a_n$. }
    \label{figure-A7}
\end{figure}

Word $\omega$ is in the language $L$ if and only if there exists a map 
$\psi: I(\omega)\lra (Q)$  from the graph of the interval to the graph of the automaton  that 
\begin{itemize}
\item Takes vertices to vertices and edges to edges, preserving orientation of edges, 
\item Labels of edges are preserved as well, 
\item Takes the initial vertex of the interval to one of the initial vertices of $(Q)$, 
\item Takes the terminal vertex of the interval to one of the accepting vertices of $(Q)$. 
\end{itemize}

More generally, we can consider all maps $\tau\in \Hom(I(\omega),(Q))$ of oriented graphs that satisfy the first condition and evaluate $\tau$ to $1\in \Bool$ if it additionally satisfies the next three conditions and to $0$ otherwise. Denote the evaluation by $\brak{\tau}$. Recall the evaluation $\alpha_{\I}$ that evaluates words in $L$ to $1$ and not in $L$ to $0$. Note that a sum of elements of $\Bool$ is $1$ if at least one of the terms is $1$, otherwise it is $0$. 

We have
\begin{equation}
  \alpha(\omega) \ = \ \sum_{\tau \in \mathsf{Map}(I(\omega),(Q))} \brak{\tau}.
\end{equation} 
In other words, $\alpha$ can be written as the sum of evaluations, over all maps from the oriented chain graph $I(\omega)$ to the graph of $(Q)$. 

One can loosely interpret this expression as a path integral interpretation of the evaluation $\alpha$, determining whether a word $\omega$ is in the language $L$. We sum over all maps from a graph which is a chain to $(Q)$, and assign $1$ to the map if the labels of all edges match, and boundary vertices are mapped to $Q_{\init}$ and $Q_{\t}$, respectively, see Figure~\ref{figure-C1}. This evaluation $\brak{\tau}$ can be written as the product of local evaluations, one for each edge of the graph $I(\omega)$, and for each of the two boundary vertices of $I(\omega)$. 

A degenerate interval $I(\emptyset_0)$, for the empty word $\emptyset_0$, is a single vertex in the Poincar\'e dual presentation. We sum over all maps to $(Q)$; in this case, over all states of $Q$, and evaluate a map to $1$ if the state if both an initial and an accepting state of $(Q)$. The empty word is in $L$ if such a state exists. 

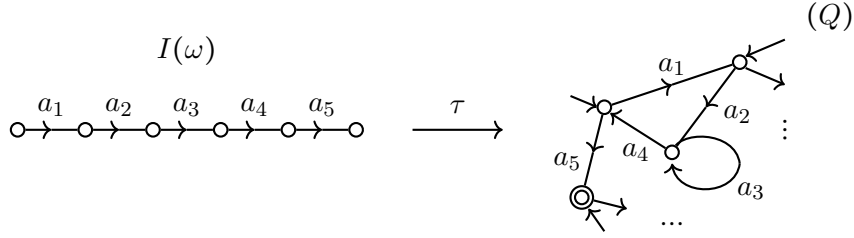
\begin{figure}
    \centering
\begin{tikzpicture}[scale=0.6]
\begin{scope}[shift={(0,0)}]
%\draw[thin,yellow] (0,0) grid (8,4);
\draw[thick,->] (0,2) -- (0.75,2);
\draw[thick,->] (0.75,2) -- (2.25,2);
\draw[thick,->] (2.25,2) -- (3.75,2);
\draw[thick,->] (3.75,2) -- (5.25,2);
\draw[thick,->] (5.25,2) -- (6.75,2);
\draw[thick] (6.75,2) -- (7.5,2);

\node at (0.75,2.5) {$a_1$};
\node at (2.25,2.5) {$a_2$};
\node at (3.75,2.5) {$a_3$};
\node at (5.25,2.5) {$a_4$};
\node at (6.75,2.5) {$a_5$};

\draw[thick,white,fill] (0.15,2) arc (0:360:1.5mm);
\draw[thick] (0.15,2) arc (0:360:1.5mm);

\draw[thick,white,fill] (1.65,2) arc (0:360:1.5mm);
\draw[thick] (1.65,2) arc (0:360:1.5mm);

\draw[thick,white,fill] (3.15,2) arc (0:360:1.5mm);
\draw[thick] (3.15,2) arc (0:360:1.5mm);

\draw[thick,white,fill] (4.65,2) arc (0:360:1.5mm);
\draw[thick] (4.65,2) arc (0:360:1.5mm);

\draw[thick,white,fill] (6.15,2) arc (0:360:1.5mm);
\draw[thick] (6.15,2) arc (0:360:1.5mm);

\draw[thick,white,fill] (7.65,2) arc (0:360:1.5mm);
\draw[thick] (7.65,2) arc (0:360:1.5mm);

\node at (3.75,3.75) {$I(\omega)$};

\node at (9.75,2.5) {$\tau$};
\draw[thick,->] (8.75,2) -- (10.75,2);

\end{scope}

\begin{scope}[shift={(12,0.5)}]
%\draw[thin,yellow] (0,0) grid (4,4);

\node at (6,4) {$(Q)$};
%\draw[thick,dashed] (0,4) -- (5,4);
%\draw[thick,dashed] (0,-1) -- (5,-1);

\draw[thick,->] (0.25,2.25) -- (0.85,2);
\draw[thick,->] (1,2) -- (2.5,2.5);
\draw[thick] (2.5,2.5) -- (4,3);

\node at (2.5,2.9) {$a_1$};
\node at (3.95,1.85) {$a_2$};
\node at (4.25,0.15) {$a_3$};
\node at (1.7,1.0) {$a_4$};
\node at (0.18,0.8) {$a_5$};

\draw[thick,->] (5,3.5) -- (4.13,3.13);
\draw[thick,->] (4.13,2.87) -- (5,2.5);
\draw[thick,->] (4,3) -- (3.25,2);
\draw[thick] (3.25,2) -- (2.5,1);
\draw[thick,<-] (1.13,1.87) -- (2.5,1);
\draw[thick,->] (1,2) -- (0.75,1);
\draw[thick] (0.75,1) -- (0.50,0);
\draw[thick,->] (0.5,0) -- (1.5,-0.25);
\draw[thick,<-] (0.65,-0.25) -- (1,-0.75);
\node at (5,1.5) {$\vdots$};
\node at (2.5,-0.65) {$\cdots$};

\draw[thick] (2.5,1) .. controls (2.6,1.5) and (3.9,1.5) .. (4,0.75);
\draw[thick,<-] (2.5,0.75) .. controls (2.6,0) and (3.9,0) .. (4,0.75);

\draw[thick,fill,white] (1.15,2) arc (0:360:1.5mm);
\draw[thick] (1.15,2) arc (0:360:1.5mm);

\draw[thick,fill,white] (4.15,3) arc (0:360:1.5mm);
\draw[thick] (4.15,3) arc (0:360:1.5mm);

\draw[thick,fill,white] (2.65,1) arc (0:360:1.5mm);
\draw[thick] (2.65,1) arc (0:360:1.5mm);

\draw[thick,fill,white] (0.75,0) arc (0:360:2.5mm);
\draw[thick,fill,white] (0.65,0) arc (0:360:1.5mm);
\draw[thick] (0.65,0) arc (0:360:1.5mm);
\draw[thick] (0.75,0) arc (0:360:2.5mm);

\end{scope}

\end{tikzpicture}
    \caption{A map of graph $I(\omega)$ to $(Q)$ that evaluates to $1$, for word length $|\omega|=5$. The leftmost vertex of $I(\omega)$ maps to one of the initial states of $(Q)$ and the rightmost  vertex maps to an accepting state (state in $Q_{t}$. }
    \label{figure-C1}
\end{figure}

\vspace{0.1in}

This interpretation extends to circular words. Recall that automaton $(Q)$ determines a circular language $L_{\circ}=L_{\circ,(Q)}$ and the corresponding circular evaluation $\alpha_{\circ}$, where a circular word $\omega\in L_{\circ}$ if there exists an $\omega$-path in $(Q)$, and then $\alpha_{\circ}(\omega)=1$. Denote by $\SS(\omega)$ the graph which is an oriented circle with word $\omega$ written along the edges. We have 
\begin{equation}\label{eq_path_circle}
  \alpha_{\circ}(\omega) \ = \ \sum_{\tau \in \mathsf{Map}(\SS(\omega),(Q))} \brak{\tau}.
\end{equation} 
Here, we are looking at all maps $\tau$ of the circle graph to the graph of $(Q)$ and evaluate a map to $1$ if and only if the labels of all edges match, see Figure~\ref{figure-C2}.

\vspace{0.1in} 

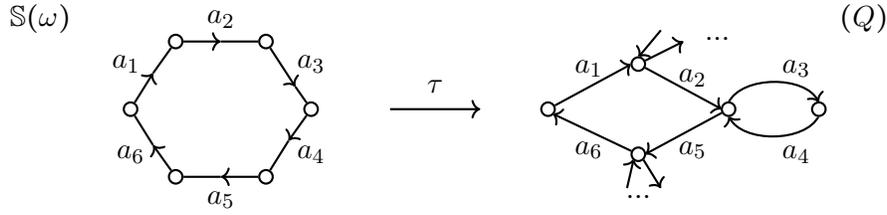
\begin{figure}
    \centering
\begin{tikzpicture}[scale=0.6]
\begin{scope}[shift={(0,0)}]
%\draw[thin,yellow] (0,0) grid (4,4);
\draw[thick,->] (1,3) -- (2,3);
\draw[thick] (2,3) -- (3,3);

\draw[thick] (1,0) -- (2,0);
\draw[thick,<-] (2,0) -- (3,0);

\draw[thick,->] (3,3) -- (3.65,2);
\draw[thick] (3.65,2) -- (4,1.5);

\draw[thick,->] (4,1.5) -- (3.5,0.8);
\draw[thick] (3.5,0.8) -- (3,0);

\draw[thick] (0,1.5) -- (0.5,0.7);
\draw[thick,<-] (0.5,0.7) -- (1,0);

\draw[thick] (1,3) -- (0.5,2.3);
\draw[thick,<-] (0.5,2.3) -- (0,1.5);

\draw[thick,fill,white] (0.15,1.5) arc (0:360:1.5mm);
\draw[thick] (0.15,1.5) arc (0:360:1.5mm);

\draw[thick,fill,white] (1.15,3) arc (0:360:1.5mm);
\draw[thick] (1.15,3) arc (0:360:1.5mm);

\draw[thick,fill,white] (1.15,0) arc (0:360:1.5mm);
\draw[thick] (1.15,0) arc (0:360:1.5mm);

\draw[thick,fill,white] (3.15,3) arc (0:360:1.5mm);
\draw[thick] (3.15,3) arc (0:360:1.5mm);

\draw[thick,fill,white] (3.15,0) arc (0:360:1.5mm);
\draw[thick] (3.15,0) arc (0:360:1.5mm);

\draw[thick,fill,white] (4.15,1.5) arc (0:360:1.5mm);
\draw[thick] (4.15,1.5) arc (0:360:1.5mm);

\node at (-2,3.5) {$\SS(\omega)$};

\node at (-0.1, 2.5) {$a_1$};
\node at (2, 3.5) {$a_2$};
\node at (4, 2.5) {$a_3$};
\node at (4, 0.5) {$a_4$};
\node at (2,-0.5) {$a_5$};
\node at (0, 0.5) {$a_6$};

\node at (6.75,2.0) {$\tau$};
\draw[thick,->] (5.75,1.5) -- (7.75,1.5);
\end{scope}

\begin{scope}[shift={(9.25,-0.50)}]
%\draw[thin,yellow] (0,0) grid (4,4);
%\draw[thick,dashed] (0,4.5) -- (6,4.5);
%\draw[thick,dashed] (0,-0.5) -- (6,-0.5);

\draw[thick,->] (0,2) -- (1.85,3);
\draw[thick,->] (2,1) -- (0.1,1.9);
\draw[thick,->] (2.5,3.75) -- (2,3.15);
\draw[thick,->] (2,3) -- (3,3.5);
\draw[thick,->] (2,3) -- (3.85,2);
\draw[thick,->] (4,2) -- (2.15,1);
\draw[thick,->] (1.75,0.25) -- (1.9,0.9);
\draw[thick,->] (2.1,0.9) -- (2.5,0.25);
\draw[thick,->] (4,2.15) .. controls (4.3,2.75) and (5.7,2.75) .. (6,2.15);
\draw[thick,<-] (4,1.85) .. controls (4.3,1.25) and (5.7,1.25) .. (6,1.85);

\node at (0.9,2.9) {$a_1$};
\node at (3.2,2.7) {$a_2$};
\node at (5.5,2.9) {$a_3$};
\node at (5.5,1.0) {$a_4$};
\node at (3.2,1.1) {$a_5$};
\node at (0.9,1.1) {$a_6$};

\node at (3.75,3.5) {$\cdots$};
\node at (2,0) {$\cdots$};

\draw[thick,fill,white] (0.15,2) arc (0:360:1.5mm);
\draw[thick] (0.15,2) arc (0:360:1.5mm);

\draw[thick,fill,white] (2.15,3) arc (0:360:1.5mm);
\draw[thick] (2.15,3) arc (0:360:1.5mm);

\draw[thick,fill,white] (2.15,1) arc (0:360:1.5mm);
\draw[thick] (2.15,1) arc (0:360:1.5mm);

\draw[thick,fill,white] (4.15,2) arc (0:360:1.5mm);
\draw[thick] (4.15,2) arc (0:360:1.5mm);

\draw[thick,fill,white] (6.15,2) arc (0:360:1.5mm);
\draw[thick] (6.15,2) arc (0:360:1.5mm);

\node at (7,4) {$(Q)$};
\end{scope}

\end{tikzpicture}
    \caption{Map of a circle graph to $(Q)$ that evaluates to $1$, word length $|\omega|=6$. }
    \label{figure-C2}
\end{figure}

\vspace{0.1in} 

Thus, we can think of both languages $L_{\I,(Q)}$ and $L_{\circ,(Q)}$ associated to an automaton $(Q)$ as computed via Boolean-valued path integrals or sums. To determine if $\omega$ is in $L_{\I,(Q)}$ we sum over maps of the interval graph $I(\omega)$ to the graph of $(Q)$. Whether $\omega$ is in $L_{\circ,(Q)}$ is determined by the sum over all maps of the circle graph $\SS(\omega)$ to the graph of $(Q)$. 

%%%%%%%%%%%%%%%%%%
% RELATION 
%%%%%%%%%%%%%%%%%%

\subsection{Relation to topological theories}
\label{subsec_relations}

The earlier paper~\cite{IK-top-automata} obtained a relation between Boolean topological theories and automata. There one starts with a regular language $L_{\I}$ and a circular language $L_{\circ}$ and builds state spaces $A(\varepsilon)$ for oriented $0$-manifolds given by sign sequences $\varepsilon$.  The state spaces $A(\varepsilon)$ are finite $\Bool$-modules, but they are not necessarily free or projective modules. The resulting theory is not a TQFT, in general: maps 
\[  A(\varepsilon)\otimes A(\varepsilon') \lra A(\varepsilon\sqcup \varepsilon')
\]
are not isomorphisms, in general, unlike in the construction of the present paper. 

In the present paper, one starts with an automaton  and defines a Boolean one-dimensional TQFT, with the state space $\Bool Q$ for the $+$ point. In particular, the state spaces are free $\Bool$-modules (see Section~\ref{subsec_aut_top} for a generalization to projective $\Bool$-modules where one replaces the discrete topological space of states of $(Q)$ by a finite topological space $X$).    
If the automaton $(Q)$ describes the language $L_{\I}$, the state space $A(+)$ can be obtained as the subquotient $\Bool$-module of $\Bool Q$. 
Thus, in~\cite{IK-top-automata} one starts with a pair of languages $(L_{\I},L_{\circ})$, while in the present paper both languages $L,L_{\circ}$ are determined by the automaton $(Q)$. 
If one picks the pair $(L_{\I},L_{\circ})$ associated to the automaton $(Q)$, the state space $A(\varepsilon)$ for a sign sequence $\epsilon$ in~\cite{IK-top-automata} is a subquotient of the free $\Bool$-module $\mcFQ(\varepsilon)$. This can be phrased more naturally, as the topological theory given by $(L_{\I},L_{\circ})$ being a subquotient theory of $\mcFQ$.

In particular, an automaton gives rise to a Boolean one-dimensional oriented TQFT (with inner endpoints and defects), while a pair of regular languages $(L_{\I},L_{\circ})$, with the second language circular, gives rise, in general, only to a one-dimensional oriented topological theory (also with inner endpoints and defects). 

Boolean TQFTs can only produce strongly circular trace languages $L_{\circ}$, see earlier and Proposition~\ref{prop_only_strongly} in the later Section~\ref{subsec_aut_top}, where the construction is extended to $\Tau$-automata and giving projective rather than free Boolean state spaces. In topological theories, one can use more general languages (circular rather than only strongly circular), still resulting in theories with finite state spaces but failing the TQFT axiom, see also the table in Section~\ref{subsec_table}. 

%%%%%%%%%%%%%%%%%
%
% 1-foams 
%
%%%%%%%%%%%%%%%%%

\section{Extending TQFTs to 1-foams}
\label{sec_extending} 

%%%%%%%%%%%%%%%%%%%
% Boolean 1D TQFT 
%%%%%%%%%%%%%%%%%%%%

\subsection{Boolean 1D TQFTs and finite topological spaces}\label{subsec_topological}
$\quad$ 
\vspace{0.07in} 

{\it Finite projective $\Bool$-modules and finite topological spaces.}
Assume $M$ is a finite $\Bool$-module. Then $M$ has a lattice structure, with 
\[
 x \wedge y  :=  \sum_{c\le x,c\le y} c, \hspace{0.75cm}  1 := \sum_{c\in M} c,  
\]
that is, $x\wedge y$ is the largest element less than or equal to both $x$ and $y$. Denote by $M^{\wedge}$ the set $M$ viewed as a lattice with \emph{join} $\vee$ and \emph{meet} $\wedge$ as above. 

\begin{prop} The following conditions on a finite $\Bool$-module $M$ are equivalent. 
\begin{enumerate}
\item $M^{\wedge}$ is a distributive lattice. 
\item $M$ is a retract of a free $\Bool$-module $\Bool^n$ for some $n$.  
\item $M$ is projective in the category of finite $\Bool$-modules, i.e., it has the lifting property for surjective semimodule homomorphisms. 
\item $M$ is isomorphic to the lattice of open sets $\mcUX$ of a finite topological space $X$. 
\end{enumerate}
\end{prop} 
$M$ is a \emph{retract} of a free $\Bool$-module if there are module maps $M\stackrel{\iota}{\lra}\Bool^n\stackrel{p}{\lra} M$ such that $p\circ \iota = \id_M$. 

We refer to~\cite[Section 3.2]{IK-top-automata} for a discussion of this proposition and more references on $\Bool$-modules. The tensor product $M\otimes N$ of two $\Bool$-modules has good behavior when one of $M,N$ is a projective $\Bool$-module. 

\vspace{0.07in} 

Let $X$ be a finite topological space. The set $\mcUX$ of open subsets of $X$ is naturally a finite distributive lattice, with join and meet operations $U\vee V= U\cup V, U\wedge V = U \cap V$, the empty set $\emptyset $ as the minimal element $0$ and  $X$ as the largest element $1$. 
Viewing $\mcUX$ as a finite $\Bool$-module, the addition is $U+V := U \cup V$.

Let $M$ be a finite projective $\Bool$-module. The proposition say, in particular, that $M\cong \mcUX$, for some finite topological space  $X$.

For $x\in X$ denote by $U_x$ the smallest open set that contains $x$. If $U_x=U_y$ for some $x,y\in X$, one of $x,y$ can be removed from $X$ without changing the lattice $\mcUX$, so we can assume $U_x\not= U_y$ for $x\not= y $ in $X$. Let us call an $X$ with this property a \emph{minimal} topological space and consider from now on only minimal $X$. 

A nonzero element $u$ of a $\Bool$-module $M$ is called \emph{irreducible} if $u=u_1+u_2$ implies that $u_1=u$ or $u_2=u$. Denote by $\irr(M)$ the set of irreducible elements of $M$. 

The set $\irr(\mcUX)$ consists of elements $U_x,x\in X$, 
\begin{equation}
\irr(\mcUX) \ = \ \{ U_x | x\in X\}. 
\end{equation} 
In particular, irreducibles in $\mcUX$ are in a bijection with points of $X$. Inclusion and projection maps 
\begin{equation}
\mcUX \stackrel{\iota}{\lra} \Bool X \stackrel{p}{\lra} \mcUX
\end{equation}
are given by 
\begin{equation}
\iota(U) = \sum_{x\in U} x, \hspace{0.75cm}  p(x) = U_x, 
\end{equation}
and $p \, \iota = \id_{\mcUX}$. 

\vspace{0.1in} 

{\it Duality.}
We continue to assume that $X$ is a minimal topological space, so that irreducibles in $\mcUX$ are in a bijection with points of $X$. Denote by $X^{\ast}$ the \emph{dual} topological space of $X$. It has the same underlying set of points and a set $V$ is open in $X^{\ast}$ if and only if it is closed in $X$. 
Thus, open sets of $X^{\ast}$ are complements of open sets of $X$. 
Irreducible elements $\irr(\mcUX^{\ast})$ are in a bijection with elements of $X$ and consist of minimal closed subsets $V_x$ of $X$ that contain $x$, one for each $x\in X$. 

There are evaluation and coevaluaton maps 
\begin{eqnarray}\label{eq_pairing_coev}
\mathsf{coev}_X & : &  \Bool \lra \mcUX \otimes \mcU(X^{\ast}), \hspace{0.75cm}
1 \longmapsto \sum_{x\in X} U_x\otimes V_x, \\ \label{eq_pairing_ev}
\mathsf{ev}_X & : & \mcU(X^{\ast}) \otimes \mcUX \lra \Bool, \hspace{0.75cm}  
V\otimes U \longmapsto \delta_{U\cap V}, 
\end{eqnarray} 
Here $\delta_{W}=1$ if set $W$ is nonempty and $\delta_{\emptyset}=0$. It is straightforward to check that these maps satisfy the deformation relations in Figure~\ref{figure-0.1} bottom row, see also \eqref{eq_two_isot_2}, where $M$ should be replaced by $\mcUX$ and $M^{\ast}$ by $\mcU(X^{\ast})$. 
Consequently, $\mcU(X^{\ast})$ is naturally isomorphic to the dual semimodule of $\mcUX$: 
\begin{equation}\label{eq_dualityX}
 \mcU(X^{\ast})  \ \cong \ \mcUX^{\ast}.
\end{equation}

\begin{corollary} A finite projective $\Bool$-module $M$ defines a symmetric monoidal functor 
\begin{equation}
\mcF_M  \ : \ \Cob \lra \Bool\dmod
\end{equation} 
taking objects $+$ and $-$ to $M$ and $M^{\ast}$ respectively, and cup and cap to $\coev_X$ and $\ev_X$ under a fixed isomorphism $M\cong \mcUX$ for a finite topological space $X$. 
\end{corollary} 
Of course, $X$ is determined by $M$ up to an isomorphism. 
For any sign sequence $\varepsilon$ the $\Bool$-module $\mcF(\varepsilon)$ is projective, isomorphic to the tensor product of projective modules $M$ and $M^{\ast}$.
A circle evaluates to $1\in \Bool$ under this functor, for any nonzero $M$. For an additional discussion of duality we refer to~\cite[Section 3.2]{IK-top-automata}. 

\vspace{0.1in} 

{\it Adding endpoints.}
Given $X$ as above, we can enhance category $\Cob$ to a  category denoted $\Cob^X_{\I}$ by allowing inner endpoints labelled by elements of $X$, see Figure~\ref{figure-D1}. A floating interval then carries two labels from $X$, one for each endpoint. 

\vspace{0.1in} 

\begin{figure}
    \centering
\begin{tikzpicture}[scale=0.6]
\begin{scope}[shift={(0,0)}]
%\draw[thin,yellow] (0,0) grid (3,3);
\node at (0.5,3.35) {$+$};
\draw[thick,dashed] (0,3) -- (1,3);
\draw[thick,dashed] (0,0) -- (1,0);

\draw[thick] (0.5,2.20) -- (0.5,3);
\draw[thick,->] (0.5,1.25) -- (0.5,2.25);
\node at (0.5,0.9) {$x$};

\node at (3,3) {$\mcUX$};
\draw[thick,->] (3,0.5) -- (3,2.5);
\node at (3,0) {$\Bool$};

\node at (4.5,3) {$\ni$};
\node at (4.25,-0.05) {$\ni$};

\node at (5.5,3) {$U_x$};
\draw[thick,|->] (5.5,0.5) -- (5.5,2.5);
\node at (5.5,-0.05) {$1$};
\end{scope}

\begin{scope}[shift={(9.5,0)}]
%\draw[thin,yellow] (0,0) grid (3,3);
\node at (0.5,3.35) {$-$};
\draw[thick,dashed] (0,3) -- (1,3);
\draw[thick,dashed] (0,0) -- (1,0);

\draw[thick,<-] (0.5,1.80) -- (0.5,3);
\draw[thick] (0.5,1.25) -- (0.5,1.85);
\node at (0.5,0.9) {$y$};
\node at (3,3) {$\mcU(X^{\ast})$};
\draw[thick,->] (3,0.5) -- (3,2.5);
\node at (3,0) {$\Bool$};

\node at (4.5,3) {$\ni$};
\node at (4.25,-0.05) {$\ni$};

\node at (5.5,3) {$V_x$};
\draw[thick,|->] (5.5,0.5) -- (5.5,2.5);
\node at (5.5,-0.05) {$1$};
\end{scope}

\begin{scope}[shift={(19,0)}]
%\draw[thin,yellow] (0,0) grid (3,3);
\draw[thick,dashed] (0,3) -- (1,3);
\draw[thick,dashed] (0,0) -- (1,0);

\node at (1.05,2.25) {$y$};
\draw[thick,->] (0.5,0.75) -- (0.5,2.25);
\node at (1.05,0.75) {$x$};
\node at (1.85,1.5) {$=$};

\node at (3.75,1.5) {$\delta_{U_x\cap V_y}$};

\end{scope}

\end{tikzpicture}
    \caption{Inner endpoints with in and out orientations labelled by points $x,y\in X$. In endpoint labelled $x$ denotes the element $U_x$ of $\mcUX$, while an out endpoint labelled $y$ denotes the functional on $\mcUX$ given by intersecting an open set in $X$ with the closed set $V_x$. A floating interval with $x,y$ endpoint labels evaluates to $1$ if and only if $U_x$ and $V_y$ have a nonempty intersection (notation $\delta_W$ denotes $1$ if the set $W$ is nonempty and $0$ if $W=\emptyset$).  }
    \label{figure-D1}
\end{figure}
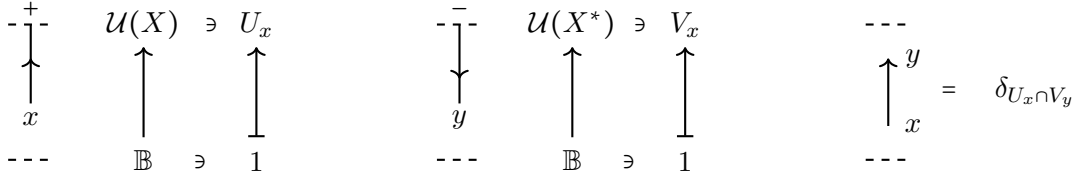

Functor $\mcF_{\mcUX}$ extends to a functor 
\begin{equation}
\mcF_{\mcUX} \ : \ \Cob^X_{\I}\lra \Bool\dmod
\end{equation} 
(keeping the same notation for the functor) 
that takes a half-interval with an endpoint labelled $x$ to $U_x\in \mcUX$ and, for the opposite orientation, taking $U$ to $\delta_{U,V_x}$, see Figure~\ref{figure-D1}. A floating interval with in, respectively out endpoint labelled $x,$ respectively $y$, evaluates to $\delta_{V_x,U_y}$, that is, to $1$ if and only if the smallest open set that contains $y$ intersects nontrivially the smallest closed set that contains $x$. 

It makes sense to simplify the notation and denote the functor $\mcF_{\mcUX}$ by $\mcF_X$, for short. 

%%%%%%%%%%%%%%%
%  Extending to 1-foams 
%%%%%%%%%%%%%%%

\subsection{Extending to one-foams}\label{subsec_foams}
$\quad$ 
\vspace{0.05in} 

{\it Multiplication on $\mcUX$ and foams.} Intersection of sets give rise to a $\Bool$-module map 
\begin{equation}\label{eq_multX}
\mcUX \otimes \mcUX \ \stackrel{m}{\lra} \ \mcUX,  \hspace{1cm}  U\cdot V := U \cap V. 
\end{equation} 
This map is well-defined due to distributivity of the intersection over union, $U\cdot (V_1+V_2)=U\cdot V_1+ U \cdot V_2$. It makes $\mcUX$ into an associative commutative unital semialgebra over $\Bool$. The unit element $1$ of $\mcUX$ is given by the biggest open set $X$, since $X\cdot U = U$ for $U\in \mcUX$. 

Consider the multiplication on the dual topological space $\mcU(X^{\ast})\otimes \mcU(X^{\ast})\lra \mcU(X^{\ast})$. Dualizing this multiplication via $\Hom_{\Bool\dmod}(\ast,\Bool)$ and isomorphism \eqref{eq_dualityX} gives a commutative coassociative comultiplication  
\begin{equation}
\mcUX \ \stackrel{\Delta}{\lra} \ \mcUX \otimes \mcUX 
\end{equation} 
with the counit map 
\begin{equation} 
\epsilon: \mcUX\lra \Bool, \ \ \epsilon(U)=1 \ \text{ if and only if } \ U\not= \emptyset. 
\end{equation}
The formula for comultiplication: 
\begin{equation}\label{eq_comult} 
\Delta(U_x) := U_x\otimes U_x, \ \  x\in X, \ \ 
\Delta(U) := \sum_{x| U_x\subset U}\Delta(U_x) = \sum_{x| U_x\subset U} U_x\otimes U_x,
\end{equation} 
where, for instance, the first sum is over $x$ such that $U_x\subset U$.
Irreducible elements $U_x$ of $\mcUX$ are sent by $\Delta$ to their tensor squares, and then the map is extended to all open sets $U$ by summing over applications of this map to all irreducibles in $U$.   
To check that $\Delta$ is indeed dual to multiplication $m$ in $\mcU(X^{\ast})$, pick open sets $U_1,U_2$ and a closed set $V$ in $X$. Then the pairing \eqref{eq_pairing_ev} computes 
\begin{equation} \label{eq_pairing} 
(V,U_1\cdot  U_2) =\delta_{ V\cap U_1 \cap U_2 } = \sum_{x|V_x\subset V} \delta_{V_x,U_1}\delta_{V_x,U_2}
=(\Delta(V),U_1\otimes U_2). 
\end{equation}
where $V_x$ is the smallest closed set in $X$ that contains $x$. We see that the multiplication \eqref{eq_multX} on $\mcUX$ is dual to the comultiplication on $\mcU(X^{\ast})$ with respect to the pairing \eqref{eq_pairing_ev}. 

We can now extend our graphical calculus by adding oriented 3-valent vertices of two kinds, see Figure~\ref{figure-D2}, to denote multiplication and comultiplication on $\mcUX$, and inner endpoints to denote the unit element $X\in \mcUX$ and the trace $\varepsilon$, which dualizes to the unit element in $\mcU(X^{\ast})$. 

\vspace{0.1in} 

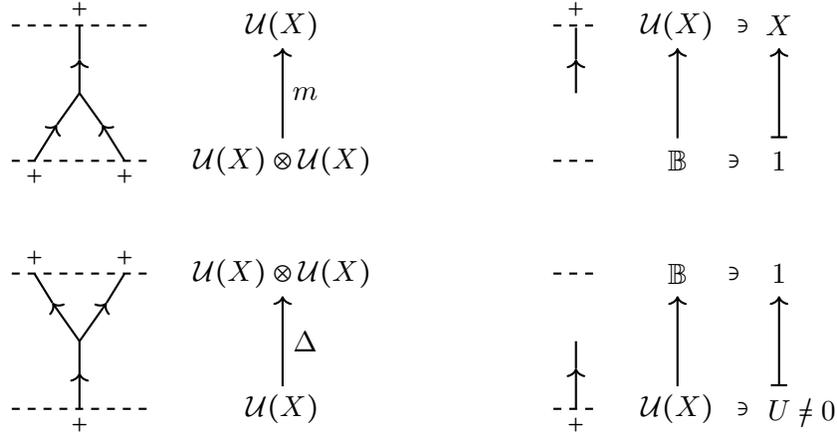
\begin{figure}
    \centering
\begin{tikzpicture}[scale=0.6]
\begin{scope}[shift={(0,0)}]
%\draw[thin,yellow] (0,0) grid (3,3);
\draw[thick,dashed] (0,3) -- (3,3);
\draw[thick,dashed] (0,0) -- (3,0);
\node at (0.5,-0.35) {$+$};
\node at (2.5,-0.35) {$+$};
\node at (1.5,3.35) {$+$};

\draw[thick] (1.5,2.25) -- (1.5,3);
\draw[thick,->] (1.5,1.5) -- (1.5,2.25);

\draw[thick,->] (0.5,0) -- (1,0.80);
\draw[thick] (1,0.80) -- (1.5,1.5);

\draw[thick,->] (2.5,0) -- (2,0.80);
\draw[thick] (2,0.80) -- (1.5,1.5);

\node at (6,3) {$\mcUX$};
\draw[thick,->] (6,0.5) -- (6,2.5);
\node at (6.5,1.5) {$m$};
\node at (6,0) {$\mcUX \otimes \mcUX$};
\end{scope}

\begin{scope}[shift={(12,0)}]
%\draw[thin,yellow] (0,0) grid (3,3);
\draw[thick,dashed] (0,3) -- (1,3);
\draw[thick,dashed] (0,0) -- (1,0);
\node at (0.5,3.35) {$+$};
\draw[thick,->] (0.5,1.5) -- (0.5,2.25);
\draw[thick] (0.5,2.25) -- (0.5,2.95);

\node at (2.75,3) {$\mcUX$};
\draw[thick,->] (2.75,0.5) -- (2.75,2.5);
\node at (2.75,0) {$\Bool$};

\node at (4.20,3) {$\ni$};
\node at (4.00,0) {$\ni$};

\node at (5,3) {$X$};
\draw[thick,|->] (5,0.5) -- (5,2.5);
\node at (5,0) {$1$};
\end{scope}

\begin{scope}[shift={(0,-5.5)}]
%\draw[thin,yellow] (0,0) grid (3,3);
\draw[thick,dashed] (0,3) -- (3,3);
\draw[thick,dashed] (0,0) -- (3,0);
\node at (0.5,3.35) {$+$};
\node at (2.5,3.35) {$+$};
\node at (1.5,-0.35) {$+$};

\draw[thick,<-] (1.5,0.80) -- (1.5,0);
\draw[thick] (1.5,1.5) -- (1.5,0.79);

\draw[thick] (0.5,3) -- (0.9,2.35);
\draw[thick,<-] (0.9,2.35) -- (1.5,1.5);

\draw[thick] (2.5,3) -- (2.1,2.35);
\draw[thick,<-] (2.1,2.35) -- (1.5,1.5);

\node at (6,0) {$\mcUX$};
\draw[thick,->] (6,0.5) -- (6,2.5);
\node at (6.5,1.5) {$\Delta$};
\node at (6,3) {$\mcUX \otimes \mcUX$};
\end{scope}

\begin{scope}[shift={(12,-5.5)}]
%\draw[thin,yellow] (0,0) grid (3,3);
\draw[thick,dashed] (0,3) -- (1,3);
\draw[thick,dashed] (0,0) -- (1,0);
\node at (0.5,-0.35) {$+$};
\draw[thick,->] (0.5,0) -- (0.5,0.90);
\draw[thick] (0.5,0.90) -- (0.5,1.5);

\node at (2.75,3) {$\Bool$};
\draw[thick,->] (2.75,0.5) -- (2.75,2.5);
\node at (2.75,0) {$\mcUX$};

\node at (4.00,3) {$\ni$};
\node at (4.20,0) {$\ni$};

\node at (5,3) {$1$};
\draw[thick,|->] (5,0.5) -- (5,2.5);
\node at (5.5,0) {$U\not= 0$};
\end{scope}

\end{tikzpicture}
    \caption{Multiplication, comultiplication, unit, and counit diagrams for $\mcUX$.  }
    \label{figure-D2}
\end{figure}

A three-valent vertex either have two edges oriented in and one edge oriented out or two edges oriented out and one edge oriented in, describing multiplication, respectively comultiplication on $\mcUX$, see Figure~\ref{figure-D2}.

Orientation reversal corresponds to switching between $X$ and $X^{\ast}$, and the relevant duality relations (relating multiplication $m_X$ in $\mcUX$ with comultiplication $\Delta_{X^{\ast}}$ in $\mcU(X^{\ast})$, and likewise for $\Delta_X$ and $m_{X^{\ast}}$) are shown in Figure~\ref{figure-D3}.

\vspace{0.1in}

\begin{figure}
    \centering
\begin{tikzpicture}[scale=0.6]
\begin{scope}[shift={(0,0)}]
%\draw[thin,yellow] (0,0) grid (3,3);
\draw[thick,dashed] (0,3) -- (3,3);
\draw[thick,dashed] (0,0) -- (3,0);
\node at (0.5,-0.35) {$+$};
\node at (1.5,-0.35) {$+$};
\node at (2.5,-0.35) {$-$};

\draw[thick,->] (0.5,0) -- (0.75,0.50);
\draw[thick] (0.75,0.50) -- (1,1);

\draw[thick] (1,1) -- (1.25,0.50);
\draw[thick,<-] (1.25,0.50) -- (1.5,0);

\draw[thick,->] (1,1) .. controls (1.1,2) and (2.15,2) .. (2.25,1);
\draw[thick] (2.25,1) .. controls (2.3,0.85) and (2.45,0.5) .. (2.5,0);

\node at (3.5,1.5) {$=$};
\end{scope}

\begin{scope}[shift={(4,0)}]
%\draw[thin,yellow] (0,0) grid (3,3);
\draw[thick,dashed] (0,3) -- (3,3);
\draw[thick,dashed] (0,0) -- (3,0);
\node at (0.5,-0.35) {$+$};
\node at (1.5,-0.35) {$+$};
\node at (2.5,-0.35) {$-$};

\draw[thick,->] (2.5,1) -- (2.5,0.50);
\draw[thick] (2.5,0.50) -- (2.5,0);

 % testing, modifying arcs below 
\draw[thick,->] (1.5,0) -- (1.75,1.2);
\draw[thick] (1.75,1.2) .. controls (1.85,1.7) and (2.4,1.5) .. (2.5,1);

%\draw[thick,->] (1.5,0) -- (1.75,1);
%\draw[thick] (1.75,1) .. controls (1.85,1.5) and (2.4,1.5) .. (2.5,1);

\draw[thick,->] (0.5,0) .. controls (0.6,1.85) and (1.15,1.85) .. (1.25,2);

\draw[thick] (1.2,1.95) .. controls (1.5,2.5) and (3.5,2.5) .. (2.5,1);
\end{scope}

\begin{scope}[shift={(10,0)}]
%\draw[thin,yellow] (0,0) grid (3,3);
\draw[thick,dashed] (0,3) -- (3,3);
\draw[thick,dashed] (0,0) -- (3,0);
\node at (0.5,3.35) {$+$};
\node at (1.5,3.35) {$+$};
\node at (2.5,3.35) {$-$};

\draw[thick] (0.5,3) -- (0.6,2.73);
\draw[thick,<-] (0.6,2.73) -- (1,2);

\draw[thick,->] (1,2) -- (1.4,2.73);
\draw[thick] (1.4,2.73) -- (1.5,3);

\draw[thick] (1,2) .. controls (1.1,1) and (2.15,1) .. (2.25,2);
\draw[thick,<-] (2.25,2) .. controls (2.3,2.15) and (2.45,2.5) .. (2.5,3);

\node at (3.5,1.5) {$=$};
\end{scope}

\begin{scope}[shift={(14,0)}]
%\draw[thin,yellow] (0,0) grid (3,3);
\draw[thick,dashed] (0,3) -- (3,3);
\draw[thick,dashed] (0,0) -- (3,0);
\node at (0.5,3.35) {$+$};
\node at (1.5,3.35) {$+$};
\node at (2.5,3.35) {$-$};

\draw[thick] (2.5,2) -- (2.5,2.50);
\draw[thick,<-] (2.5,2.50) -- (2.5,3);
 
\draw[thick] (1.5,3) -- (1.65,2.2);
\draw[thick,<-] (1.65,2.2) .. controls (1.85,1.5) and (2.4,1.5) .. (2.5,2);

\draw[thick] (0.5,3) -- (0.6,2);

\draw[thick,<-] (0.6,2) .. controls (0.7,0.5) and (3.5,0.5) .. (2.5,2);
\end{scope}

\end{tikzpicture}
    \caption{ The duality relation between  multiplication and comultiplication for $\mcUX$ and $\mcU(X^{\ast})$. }
    \label{figure-D3}
\end{figure}

Associative commutative unital $\Bool$-algebra relations on $\mcUX$ are shown in Figure~\ref{figure-D4}. The same relations with all orientations reversed hold as well, see Figure~\ref{figure-D4.1} where these relations are also rotated. These correspond to the coassociative cocommutative counital $\Bool$-coalgebra structure on $\mcUX$ or, equivalently, to the algebra structure on the dual module $\mcU(X^{\ast})$. Additionally, the relation in Figure~\ref{figure-D5} holds.  

\vspace{0.1in}

\begin{figure}
    \centering
\begin{tikzpicture}[scale=0.6]

\begin{scope}[shift={(0,0)}]
%\draw[thin,yellow] (0,0) grid (4,4);
\node at (0.25,-0.35) {$+$};
\node at (1.5,-0.35) {$+$};
\node at (2.75,-0.35) {$+$};
\node at (1.5,3.35) {$+$};
\draw[thick,dashed] (0,3) -- (3,3);
\draw[thick,dashed] (0,0) -- (3,0);

\draw[thick,->] (1.5,1.5) -- (1.5,2.25);
\draw[thick] (1.5,2.25) -- (1.5,3);

\draw[thick] (1.5,1.5) -- (2,0.9);
\draw[thick,<-] (1.97,0.95) -- (2.75,0);

\draw[thick,->] (0.25,0) -- (0.55,0.35);
\draw[thick,->] (0.5,0.3) -- (1.25,1.20);
\draw[thick] (1.2,1.15) -- (1.5,1.5);

\draw[thick] (0.85,0.75) -- (1.105,0.457);
\draw[thick,<-] (1.10,0.46) -- (1.5,0);

\node at (3.5,1.5) {$=$};
\end{scope}

\begin{scope}[shift={(4,0)}]
%\draw[thin,yellow] (0,0) grid (4,4);
\node at (0.25,-0.35) {$+$};
\node at (1.5,-0.35) {$+$};
\node at (2.75,-0.35) {$+$};
\node at (1.5,3.35) {$+$};
\draw[thick,dashed] (0,3) -- (3,3);
\draw[thick,dashed] (0,0) -- (3,0);

\draw[thick,->] (1.5,1.5) -- (1.5,2.25);
\draw[thick] (1.5,2.25) -- (1.5,3);

\draw[thick] (1.5,1.5) -- (1,0.9);
\draw[thick,<-] (1.03,0.95) -- (0.25,0);

\draw[thick,->] (2.75,0) -- (2.45,0.35);
\draw[thick,->] (2.5,0.3) -- (1.75,1.20);
\draw[thick] (1.8,1.15) -- (1.5,1.5);

\draw[thick] (2.15,0.75) -- (1.895,0.457);
\draw[thick,<-] (1.90,0.46) -- (1.5,0);
\end{scope}

\begin{scope}[shift={(11,0)}]
%\draw[thin,yellow] (0,0) grid (4,4);
\node at (0.5,-0.35) {$+$};
\node at (1.5,-0.35) {$+$};
\node at (1,3.35) {$+$};
\draw[thick,dashed] (0,3) -- (2,3);
\draw[thick,dashed] (0,0) -- (2,0);

\draw[thick] (1,3) -- (1,2.65);
\draw[thick,<-] (1,2.65) -- (1,2.25);

\draw[thick] (0.5,1.75) .. controls (0.6,2.25) and (0.9,2.25) .. (1,2.25);
\draw[thick] (1.5,1.75) .. controls (1.4,2.25) and (1.1,2.25) .. (1,2.25);

\draw[thick,<-] (0.5,1.75) .. controls (0.6,1) and (1.4,1) .. (1.5,0);
\draw[thick,->] (0.5,0) .. controls (0.6,1) and (1.4,1) .. (1.5,1.75);

\node at (2.5,1.5) {$=$};
\end{scope}

\begin{scope}[shift={(14,0)}]
%\draw[thin,yellow] (0,0) grid (4,4);
\node at (0.5,-0.35) {$+$};
\node at (1.5,-0.35) {$+$};
\node at (1,3.35) {$+$};
\draw[thick,dashed] (0,3) -- (2,3);
\draw[thick,dashed] (0,0) -- (2,0);

\draw[thick] (1,3) -- (1,2.47);
\draw[thick,<-] (1,2.5) -- (1,1.5);

\draw[thick] (0.5,0.85) .. controls (0.6,1.5) and (0.9,1.5) .. (1,1.5);
\draw[thick] (1.5,0.85) .. controls (1.4,1.5) and (1.1,1.5) .. (1,1.5);

\draw[thick,<-] (0.5,0.85) -- (0.5,0);
\draw[thick,<-] (1.5,0.85) -- (1.5,0);
\end{scope}

\begin{scope}[shift={(20,0)}]
%\draw[thin,yellow] (0,0) grid (4,4);
\node at (0.25,-0.35) {$+$};
\node at (1,3.35) {$+$};
\draw[thick,dashed] (0,3) -- (2,3);
\draw[thick,dashed] (0,0) -- (2,0);

\draw[thick] (1,3) -- (1,2.25);
\draw[thick,<-] (1,2.255) -- (1,1.5);

\draw[thick,<-] (1.35,1.05) -- (1.75,0.5);
\draw[thick] (1,1.5) -- (1.355,1.04);

\draw[thick] (1,1.5) -- (0.75,1.02);
\draw[thick,<-] (0.75,1.03) -- (0.25,0);

\node at (2.5,1.5) {$=$};
\end{scope}

\begin{scope}[shift={(23,0)}]
%\draw[thin,yellow] (0,0) grid (3,3);
\node at (0.5,3.35) {$+$};
\node at (0.5,-0.35) {$+$};
\draw[thick,dashed] (0,3) -- (1,3);
\draw[thick,dashed] (0,0) -- (1,0);

\draw[thick,->] (0.5,0) -- (0.5,1.5);
\draw[thick] (0.5,1.47) -- (0.5,3);
\end{scope}

\end{tikzpicture}
    \caption{Associativity, commutativity, and the unit relations in $\mcUX$.  }
    \label{figure-D4}
\end{figure}
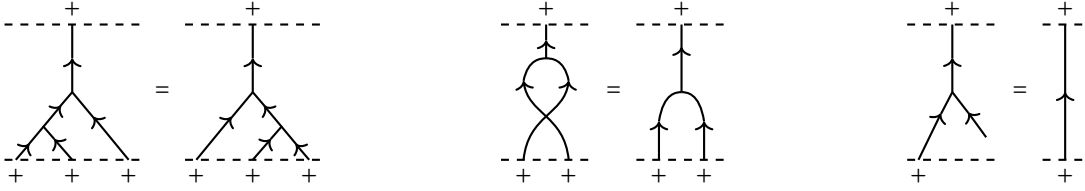

\begin{figure}
    \centering
\begin{tikzpicture}[scale=0.6]
\begin{scope}[shift={(0,0)}]
%\draw[thin,yellow] (0,0) grid (8,4);
\node at (0.25,3.35) {$+$};
\node at (1.5,3.35) {$+$};
\node at (2.75,3.35) {$+$};
\node at (1.5,-0.35) {$+$};
\draw[thick,dashed] (0,3) -- (3,3);
\draw[thick,dashed] (0,0) -- (3,0);

\draw[thick] (1.5,1.5) -- (1.5,0.75);
\draw[thick,<-] (1.5,0.75) -- (1.5,0);

\draw[thick,->] (1.5,1.5) -- (2,2.1);
\draw[thick] (1.97,2.05) -- (2.75,3);

\draw[thick] (0.25,3) -- (0.55,2.65);
\draw[thick,<-] (0.5,2.7) -- (1.25,1.800);
\draw[thick,<-] (1.2,1.85) -- (1.5,1.5);

\draw[thick,->] (0.85,2.25) -- (1.30,2.77);
\draw[thick] (1.27,2.745) -- (1.5,3);

\node at (3.5,1.5) {$=$};
\end{scope}

\begin{scope}[shift={(4,0)}]
%\draw[thin,yellow] (0,0) grid (8,4);
\node at (0.25,3.35) {$+$};
\node at (1.5,3.35) {$+$};
\node at (2.75,3.35) {$+$};
\node at (1.5,-0.35) {$+$};
\draw[thick,dashed] (0,3) -- (3,3);
\draw[thick,dashed] (0,0) -- (3,0);

\draw[thick] (1.5,1.5) -- (1.5,0.75);
\draw[thick,<-] (1.5,0.75) -- (1.5,0);

\draw[thick,->] (1.5,1.5) -- (1,2.1);
\draw[thick] (1.03,2.05) -- (0.25,3);

\draw[thick] (2.75,3) -- (2.45,2.65);
\draw[thick,<-] (2.5,2.7) -- (1.75,1.800);
\draw[thick,<-] (1.8,1.85) -- (1.5,1.5);

\draw[thick,->] (2.15,2.25) -- (1.7,2.77);
\draw[thick] (1.73,2.745) -- (1.5,3);
\end{scope}

\begin{scope}[shift={(11,0)}]
%\draw[thin,yellow] (0,0) grid (4,4);
\node at (0.5,3.35) {$+$};
\node at (1.5,3.35) {$+$};
\node at (1,-0.35) {$+$};
\draw[thick,dashed] (0,3) -- (2,3);
\draw[thick,dashed] (0,0) -- (2,0);

\draw[thick,->] (1,0) -- (1,0.5);
\draw[thick] (1,0.5) -- (1,0.75);

\draw[thick,<-] (0.5,1.5) .. controls (0.6,0.75) and (0.9,0.75) .. (1,0.75);
\draw[thick,<-] (1.5,1.5) .. controls (1.4,0.75) and (1.1,0.75) .. (1,0.75);

\draw[thick] (0.5,1.5) .. controls (0.6,2) and (1.4,2) .. (1.5,3);
\draw[thick] (0.5,3) .. controls (0.6,2) and (1.4,2) .. (1.5,1.5);

\node at (2.5,1.5) {$=$};
\end{scope}

\begin{scope}[shift={(14,0)}]
%\draw[thin,yellow] (0,0) grid (4,4);
\node at (0.5,3.35) {$+$};
\node at (1.5,3.35) {$+$};
\node at (1,-0.35) {$+$};
\draw[thick,dashed] (0,3) -- (2,3);
\draw[thick,dashed] (0,0) -- (2,0);

\draw[thick,->] (1,0) -- (1,1);
\draw[thick] (1,0.97) -- (1,1.5);

\draw[thick,<-] (0.5,2.25) .. controls (0.6,1.5) and (0.9,1.5) .. (1,1.5);
\draw[thick,<-] (1.5,2.25) .. controls (1.4,1.5) and (1.1,1.5) .. (1,1.5);

\draw[thick] (0.5,2.25) -- (0.5,3);
\draw[thick] (1.5,2.25) -- (1.5,3);
\end{scope}

\begin{scope}[shift={(20,0)}]
%\draw[thin,yellow] (0,0) grid (4,4);
\node at (1.75,3.35) {$+$};
\node at (1,-0.35) {$+$};
\draw[thick,dashed] (0,3) -- (2,3);
\draw[thick,dashed] (0,0) -- (2,0);

\draw[thick,->] (1,0) -- (1,1);
\draw[thick] (1,0.97) -- (1,1.5);

\draw[thick] (0.65,1.99) -- (0.25,2.5);
\draw[thick,->] (1,1.5) -- (0.65,2.00);

\draw[thick,->] (1,1.5) -- (1.5,2.5);
\draw[thick] (1.5,2.47) -- (1.75,3);

\node at (2.5,1.5) {$=$};
\end{scope}

\begin{scope}[shift={(23,0)}]
%\draw[thin,yellow] (0,0) grid (3,3);
\node at (0.5,3.35) {$+$};
\node at (0.5,-0.35) {$+$};
\draw[thick,dashed] (0,3) -- (1,3);
\draw[thick,dashed] (0,0) -- (1,0);

\draw[thick,->] (0.5,0) -- (0.5,1.5);
\draw[thick] (0.5,1.47) -- (0.5,3);
\end{scope}

\end{tikzpicture}
    \caption{Coassociativity, cocomutativity and the counit relations in $\mcUX$.}
    \label{figure-D4.1}
\end{figure}

\begin{figure}
    \centering
\begin{tikzpicture}[scale=0.6]
\begin{scope}[shift={(0,0)}]
%\draw[thin,yellow] (0,0) grid (4,4);

\node at (1, 4.35) {$+$};
\node at (1,-0.35) {$+$};

\draw[thick,dashed] (0,4) -- (2,4);
\draw[thick,dashed] (0,0) -- (2,0);

\draw[thick] (1,3.60) -- (1,4);
\draw[thick,->] (1,3) -- (1,3.60);

\draw[thick] (0.25,2) .. controls (0.25,2.25) and (0.25,2.75) .. (1,3);

\draw[thick] (1.75,2) .. controls (1.75,2.25) and (1.75,2.75) .. (1,3);

\draw[thick,->] (1,1) .. controls (0.25,1.25) and (0.25,1.75) .. (0.25,2);

\draw[thick,->] (1,1) .. controls (1.75,1.25) and (1.75,1.75) .. (1.75,2);

\draw[thick] (1,0.6) -- (1,1);
\draw[thick,->] (1,0) -- (1,0.6);

\node at (3,2) {$=$};
\end{scope}

\begin{scope}[shift={(4,0)}]
%\draw[thin,yellow] (0,0) grid (4,4);
\node at (0.5, 4.35) {$+$};
\node at (0.5,-0.35) {$+$};

\draw[thick,dashed] (0,4) -- (1,4);
\draw[thick,dashed] (0,0) -- (1,0);

\draw[thick] (0.5,2) -- (0.5,4);
\draw[thick,->] (0.5,0) -- (0.5,2);

\end{scope}

\end{tikzpicture}
    \caption{Split-merge composition is the identity.}
    \label{figure-D5}
\end{figure}
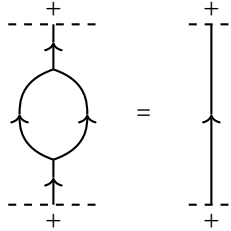

\begin{remark}  In general, $m$ and $\Delta$ do not satisfy the bialgebra axiom: $\Delta$ is not a homomorphism of semirings, $\Delta \circ m \not= (m\otimes m)\circ P_{23}\circ (\Delta\otimes \Delta)$. However,
\[
\Delta \circ m (U\otimes V) \ \le \ (m\otimes m)\circ P_{23}\circ (\Delta\otimes \Delta)(U\otimes V), 
 \ \ \forall \, U,V\in \mcUX,  
\] 
where $u\le v$ means that $u+v=v$, for elements $u,v$ of a semilattice $M$. This relation, then, holds for the operators $\Delta\circ m$ and $(\Delta\otimes \Delta)(U\otimes V)$, and we can write 
\[
\Delta\circ m + (m\otimes m)\circ P_{23}\circ (\Delta\otimes \Delta) = (m\otimes m)\circ P_{23}\circ (\Delta\otimes \Delta). 
\]
We did not look for general inequalities of this form and feel that they likely hold for random reasons and will not naturally generalize from the Boolean to more general commutative semirings. 
\end{remark}  

\vspace{0.1in} 

Motivated by these relations, one can introduce rigid symmetric monoidal category $\FCob$ with sign sequences as objects. It contains $\Cob$ as a subcategory and has additional trivalent vertex and inner endpoint generators as in Figure~\ref{figure-D2}. Relations in Figures~\ref{figure-D3}, \ref{figure-D4}, and \ref{figure-D4.1} hold, as well as the standard relations from the symmetric structure on trivalent and univalent graphs, such as sliding a disjoint line over a vertex. One may or may not impose Figure~\ref{figure-D5} relation. 

Any finite topological space $X$ defines a symmetric monoidal functor 
\[\mcF_X \ : \FCob \lra \Bool\dmod.
\] 
This functor is  given on objects by 
\[
 \mcF_X(+)= \mcUX, \ \  \mcF_X=\mcU(X^{\ast}),
\]
and on generating morphisms by the formulas in Figure~\ref{figure-D2} (also formulas \eqref{eq_multX} and \eqref{eq_comult}) and the cup and cap formulas \eqref{eq_pairing_coev} and \eqref{eq_pairing_ev}. 

It is natural to view morphisms in $\FCob$ as describing \emph{one-foams with boundary}. Two-foams appear when constructing link homology theories~\cite{Kh3,MV,RW1,RW16}. 

\vspace{0.1in} 

\begin{figure}
    \centering
\begin{tikzpicture}[scale=0.6]
\begin{scope}[shift={(0,0)}]
%\draw[thin,yellow] (0,0) grid (4,4);
\draw[thick,dashed] (-1,4) -- (5,4);
\draw[thick,dashed] (-1,0) -- (5,0);

\draw[thick,->] (0.5,4) .. controls (0.6,6) and (3.4,6) .. (3.5,4);

\draw[thick,->] (1.5,4) .. controls (1.6,4.75) and (2.4,4.75) .. (2.5,4);

\draw[thick] (1.5,0) .. controls (1.6,-0.75) and (2.4,-0.75) .. (2.5,0);

\draw[thick] (0.5,0) .. controls (0.6,-2) and (3.4,-2) .. (3.5,0);

\draw[thick,->] (0.5,0) -- (0.5,0.75);
\draw[thick] (0.5,0.75) -- (0.5,1.00);

%\draw[thick,->] (1.5,0) -- (1.5,0.75);
%\draw[thick] (1.5,0.75) -- (1.5,1.00);

\draw[thick,->] (0.5,1) -- (0.80,1.30);
\draw[thick] (0.80,1.30) -- (1.0,1.5);

%\draw[thick,->] (1.5,1) -- (1.2,1.3);
\draw[thick] (1.2,1.3) -- (1.0,1.5);

\draw[thick,<-] (1.2,1.3) .. controls (1.3,1.2) and (1.5,0.1) .. (1.5,0);

\draw[thick,->] (1.0,1.5) .. controls (1.1,2.5) and (1.4,1.5) .. (1.5,3);

\draw[thick] (1.5,3) -- (1.5,4);

\draw[thick,->] (0.5,1) -- (0.15,1.70);
\draw[thick] (0.15,1.70) -- (0,2);

\draw[thick] (0,2) .. controls (0.1,3) and (0.4,2) .. (0.5,4);

\draw[thick] (2.5,0) -- (2.6,0.25);
\draw[thick,<-] (2.6,0.25) -- (3,1);

\draw[thick] (3.5,0) -- (3.40,0.25);
\draw[thick,<-] (3.40,0.25) -- (3.00,1);

\draw[thick] (3,1) -- (3,1.35);
\draw[thick,<-] (3,1.35) -- (3,2);

\draw[thick,->] (2.5,2.5) -- (2.8,2.2);
\draw[thick] (2.8,2.2) -- (3,2);

\draw[thick] (3,2) -- (3.2,2.2);
\draw[thick,<-] (3.2,2.2) -- (3.5,2.5);

\draw[thick,->] (2.5,4) -- (2.5,3);
\draw[thick] (2.5,3) -- (2.5,2.5);

\draw[thick,->] (3.5,4) -- (3.5,3);
\draw[thick] (3.5,3) -- (3.5,2.5);

\node at (9, 5) {caps};
\node at (9,2) {trivalent vertices};
\node at (9,-1) {cups};

\end{scope}

\end{tikzpicture}
    \caption{A closed one-foam $C$ (an endomorphism of the identity object $\one$ of $\FCob$) in standard position.}
    \label{figure-G2}
\end{figure}
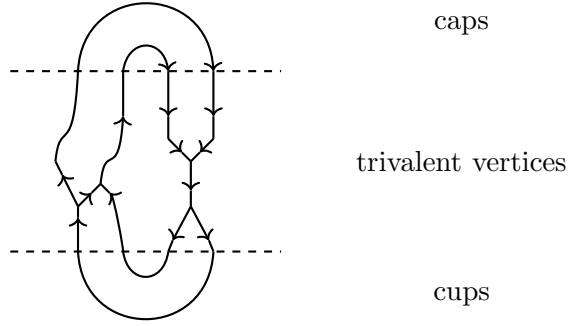

\begin{prop} Any closed diagram $C\in \End_{\FCob}(\one)$ evaluates to $1$, i.e.,  $\mcF_X(C)=1$. 
\end{prop}
\begin{proof} Any closed connected diagram $C$ can be simplified not to have inner endpoints (as on the right of Figure~\ref{figure-D2}), by canceling them against an adjacent trivalent vertex. After this simplification $C$ is either an interval (and evaluates to $1$) or it can be put in the form as shown in Figure~\ref{figure-G2}, with all cups at the bottom, all caps at the top and merge and split trivalent vertices in the middle. To evaluate this diagram one uses multiplication and comultiplication in $\mcUX$ given by \eqref{eq_multX} and \eqref{eq_comult} as well as the evaluation and coevaluation maps in equations \eqref{eq_pairing_ev} and \eqref{eq_pairing_coev}. 

Each cup in the diagram
The diagram can be evaluated from the bottom up, with each cup  contributing the sum $\sum_x U_x\otimes V_x$, see  \eqref{eq_pairing_coev}, where, recall, $U_x$, respectively $V_x$,  is the smallest open, respectively closed, set that contains $x$. Let us fix $y\in X$ and restrict to one term in the sum over all cups which is $U_y\otimes V_y$ for each cup. Evaluating the diagram for this single term, minimal open and closed sets $U_y,V_y$ propagate via multiplication and comultiplication to the tensor product of $U_y$ and $V_y$ at the ``caps'' dashed line. Coupling them in pairs via caps results in $1$, since $(U_y,V_y)=1$. 

For each $y\in X$ we obtain a term in the sum which is $1$. In $\Bool$ any sum with at least one term $1$ equals $1$. 
\end{proof}

\begin{remark} Replacing $\Bool$ by a field $\kk$ and avoiding the relation in Figure~\ref{figure-D5}, the $\kk$-linear analogue of the structure above should be a commutative associative unital finite-dimensional $\kk$-algebra $A$. 
Such an algebra defines a symmetric monoidal functor $\mcF_A$ from the cobordism category 
$\FCob$ to $\kkvect$  taking 0-manifold $(+)$ to $A$, $(-)$ to $A^{\ast}$, duality morphisms to the usual duality maps between $A$ and $A^{\ast}$, 
trivalent vertices to the multiplication map on $A$ and the dual comultiplication map coming from the isomorphism $A\cong A^{\ast}$. Univalent vertex is given by the inclusion of the unit element into $A$ and, for the opposite orientation, by the dual map $A^{\ast} \lra \kk$.
\end{remark}

%%%%%%%%%%%%%%%%%%%
% Quasiautomata 
%%%%%%%%%%%%%%%%%%%

\subsection{Automata on finite topological spaces (\texorpdfstring{$\Tau$}{tau}-automata)}
\label{subsec_aut_top}
Recall that a one-dimensional Boolean TQFT $\mcF:\Cob\lra \Bool\dmod$ associates a finite projective $\Bool$-module $P=\mcF(+)$ to the plus point and the dual module $P^{\ast}\cong \mcF(-)$ to the minus point. A finite projective $\Bool$-module $P$ is isomorphic to the module of open sets $\mcUX$ of a finite topological space $X$, with $P^{\ast}\cong \mcU(X^{\ast})$ and duality morphisms in the TQFT coming from the standard duality maps $\coev$ and $\ev$ for $\mcUX$ and $\mcU(X^\ast)$, respectively. 

Let us see how such a theory extends to a functor 
$\mcF:\CobSI\lra \Bool\dmod$, allowing cobordisms with inner endpoints and $\Sigma$-defects. A map $\Bool\lra \mcUX$ takes $1\in\Bool$ to an open set $X_{\init}\subset X$. A map $\mcUX\lra \Bool$ is determined by picking a closed set $X_{\t}$ in $X$ (equivalently, an open set $X_{\t}$ in $X^{\ast}$) and taking open $U\subset X $ to $\delta_{X_{\t}\cap U}$, see the notation in Figure~\ref{figure-D1} caption. That is, $U$ goes to $0$ if and only if it is disjoint from $X_{\t}$.

Each $a\in \Sigma$ determines an endomorphism $m_a\in \End(\mcUX)$ which sends open sets to open set and respects the union: 
\[
m_a: \mcUX \lra \mcUX, \ \ 
m_a(U_1 \cup U_2) = m_a(U_1) \cup m_a(U_2), \ \ m_a(\emptyset)=\emptyset. 
\]

An endomorphism $T$ of $\mcUX$ can be described by its action on minimal open sets $U_x$, $x\in X$. Reducing topological space $X$, if necessary, we can assume that $U_x\not= U_y$, for $x\not= y\in X$. Then $T(U_x)$ is an open set in $X$, for each $x\in X$, and $T(U_y)\subset T(U_x)$ whenever $y\in U_x$, so that $U_y\subset U_x$, so that $T$ preserves the inclusions of minimal open sets. Vice versa, an assignment $x\longmapsto T(U_x)\in \mcUX$ subject to the above condition describes an endomorphism $T$ of $\mcUX$. 

The trace of $T$ is given by (see also \eqref{eq_pairing_coev}, \eqref{eq_pairing_ev}) 
\begin{equation}\label{eq_trace_T}
\tr(T) \ = \ \ev_{X^{\ast}} \circ (T\otimes \id_{X^{\ast}})\circ \coev_X = \sum_{x\in X} \delta_{T(U_x)\cap V_x} = \sum_{x\in X} \delta(U_x\subset T(U_x)) = \sum_{x\in X} \delta(x\in T(U_x)) .  \end{equation} 
Here $\delta(A)=1$ if and only if the statement $A$ is true and $\delta_W=1$ if and only if the set $W$ is nonempty, otherwise the value of $\delta$  is $0\in \Bool$. 
Thus, the trace is $1$ if and only if for some $x\in X$ the image $T(U_x)$ has a nonempty intersection with $V_x$, the smallest closed subset of $X$ that contains $x$. This condition is equivalent to 
$U_x\subset T(U_x)$ for some $x$ and to $x\in T(U_x)$ for some $x$. 

\begin{corollary}\label{corollary_trace}
If $\tr(T)=1$, for an endomorphism $T$ of a finite projective $\Bool$-module $P$, then $\tr(T^n)=1$ for any $n\in \Z_+$. 
\end{corollary} 
Since $\tr(T)=1$ is equivalent to $U_x\subset T(U_x)$ for some $x\in X$, iterating the inclusion implies the corollary. 
$\square$

By a \emph{$\Tau$-automaton} or a \emph{quasi-automaton} we may call the data as above: 
\[    (X) := (X,X_{\init},X_{\t},\{m_a\}_{a\in \Sigma}).
\]
It consists of  a finite topological space $X$, an \emph{initial} open set $X_{\init}$, an \emph{accepting} closed set $X_{\t}$, and endomorphisms $m_a$ of  
$\mcUX$, for $a\in \Sigma$. 
A $\Tau$-automaton is equivalent to a one-dimensional oriented $\Bool$-valued TQFT with inner endpoints and $\Sigma$-defects. 

The interval language $L_{\I,(X)}$ of a $\Tau$-automaton $(X)$ consists of words $\omega\in \Sigma^{\ast}$ such that the intersection $X_{\t} \cap m_\omega (X_{\init})\not= \emptyset$. 
Here $m_{\omega}= m_{a_1}\cdots m_{a_n}$ for $\omega=a_1\cdots a_n$. 

The trace language $L_{\circ,(X)}$ of $(X)$ consists of words 
$\omega$ such that for some $x\in X$ the set $m_{\omega}(U_x)$ contains $x$, see \eqref{eq_trace_T}.

\begin{prop} \label{prop_only_strongly} The interval trace languages $(L_{\I,(X)},L_{\circ,(X)})$ of a $\Tau$-automaton $(X)$ are regular and the trace language $L_{\circ,(X)}$ is strongly circular. 
\end{prop} 
The second statement follows from Corollary~\ref{corollary_trace}. $\square$

\vspace{0.05in} 

An automaton $(Q)$ is a special case of a $\Tau$-automaton. Nondeterministic finite state automata are precisely $\Tau$-automata for discrete topological spaces $X$. To match the definitions,  more than one initial state in a nondeterministic finite automaton is allowed.  

\vspace{0.1in} 

Categories $\CobSI$ and $\FCob$ can be combined into a category $\FCob_{\Sigma,\I}$ of one-foams with defects, where now edges of a one-foam carry dots (defects) labelled by elements of $\Sigma$. One also needs to allow two types of inner endpoints, for the unit element $X\in \mcUX$ and its dual and for the initial (and accepting) $\Tau$-automata sets. Projective $\Bool$-module $\mcUX$ comes with a commutative associative multiplication $U\cdot V := U \cap V$, giving rise to a monoidal functor 
\[
 \mcF_{(X)} \ : \ \FCob_{\Sigma,\I}\lra \Bool\dmod. 
\]
We leave the details to an interested reader.

%%%%%%%%%%%%%%%%%
% Table 
%%%%%%%%%%%%%%%%%

\subsection{Summary table} \label{subsec_table}

\begin{figure}
    \centering
\begin{tikzpicture}[scale=0.6]
\begin{scope}[shift={(0,0)}]
%\draw[thin,yellow] (0,0) grid (20,5);

\draw[thick] (0,6) -- (19.75,6);
\draw[thick] (0,5) -- (19.75,5);
\draw[thick] (0,4) -- (19.75,4);
\draw[thick] (0,3) -- (19.75,3);
\draw[thick] (0,2) -- (19.75,2);
\draw[thick] (0,1) -- (19.75,1);
\draw[thick] (0,0) -- (19.75,0);

\draw[thick] (0,0) -- (0,6);

\node at (1.175,4.5) {TQFT};
\node at (1.175,3.5) {$\mcF(+)$};
\node at (1.175,2.5) {$L_{\circ}$};
\node at (1.175,1.5) {$L_{\I}$};
\draw[thick,<-] (1.545,0.5) arc (0:360:0.37);

\draw[thick] (2.35,0) -- (2.35,6);

\node at (5.2,5.5) {Automaton $(Q)$};
\node at (5.2,4.5) {Yes};
\node at (5.2,3.5) {free $\Bool$-module};
\node at (5.2,2.5) {strongly circular};
\node at (5.2,1.5) {regular language};
\node at (5.2,0.5) {evaluates to $1\in\Bool$};

\draw[thick] (8,0) -- (8,6);

\node at (11.1,5.5) {$\Tau$-automaton $(X)$};
\node at (11.1,4.5) {Yes};
\node at (11.1,3.5) {projective $\Bool$-module};
\node at (11.1,2.5) {strongly circular};
\node at (11.1,1.5) {regular language};
\node at (11.1,0.5) {evaluates to $1\in\Bool$};

\draw[thick] (14.25,0) -- (14.25,6);

\node at (17,5.5) {pair $(L_{\I},L_{\circ})$};
\node at (17,4.5) {No, in general};
\node at (17,3.5) {finite $\Bool$-module};
\node at (17,2.5) {circular};
\node at (17,1.5) {regular language};
\node at (17,0.5) {either $0$ or $1\in\Bool$};

\draw[thick] (19.75,0) -- (19.75,6);

\end{scope}
    
\end{tikzpicture}
    \caption{TQFTs associated to automata and $\Tau$-automata (present paper) vs topological theories associated to pairs $(L_\I,L_{\circ})$ of a regular language and a regular circular language~\cite{IK-top-automata}. }
    \label{table-001}
\end{figure}

Table in Figure~\ref{table-001} summarizes  similarities and differences between TQFTs constructed in the present paper and associated to automata and $\Tau$-automata and topological theories associated to a pair of regular language (with a circular second language) in~\cite{IK-top-automata}. Topological theory associated to $(L_\I,L_{\circ})$ is not a TQFT, in general, and some pairs $(L_\I,L_{\circ})$ cannot be realized via any TQFT (for instance if $\emptyset\notin L_{\circ}$ and language $L_\I$ is nonempty). More generally, $(L_\I,L_{\circ})$  is not realizable in any TQFT if $L_{\circ}$ is not strongly circular.  For all three columns, the state space $\mcF(+)$ is  a finite $\Bool$-module; it is additionally free or projective for the first or second column, respectively.  

%%%%%%%%%%%%%%%%%%%%%
%
% MSC 2020 
%
%%%%%%%%%%%%%%%%%%%%%

% QA: Quantum Algebra,
%CS.FL: Computer Science, Formal Languages,
% CT: Category Theory. 

%%%%%%%%%%%%%%%%%%%%%
%%
%%   REFERENCES 
%%
%%%%%%%%%%%%%%%%%%%%

%\addcontentsline{toc}{section}{References}
%\def\refname{}

\bibliographystyle{amsalpha} 

\bibliography{top-automata}

\end{document}